\documentclass[11pt]{article}
\usepackage{fullpage}
\usepackage{amsfonts}
\usepackage{graphicx}
\usepackage{amsmath,amssymb,amsthm}
\usepackage[ruled]{algorithm2e}
\usepackage{algorithmic}
\usepackage{enumerate}

\newtheorem{theorem}{Theorem}
\newtheorem{lemma}{Lemma}

\newtheorem{definition}{Definition}
\newtheorem{corollary}{Corollary}
\newtheorem{assumption}{Assumption}

\newcommand{\ie}{\emph{i.e.}}
\newcommand{\eg}{\emph{e.g.}}
\newcommand{\cf}{\emph{cf.}}
\newcommand{\etc}{\emph{etc.}}
\newcommand{\minimize}{\mathop{\mathrm{minimize}{}}}

\newcommand{\eqdef}{\,\stackrel{\mathrm{def}}{=}\,}
\newcommand{\defeq}{{\,:=\,}}

\newcommand{\order}{\mathcal{O}}

\newcommand{\E}{\mathbb{E}}
\newcommand{\Prob}{\mathbb{P}}
\newcommand{\R}{\mathbb{R}}

\newcommand{\norms}[1]{\|{#1}\|}

\newcommand{\ltwos}[1]{\norms{#1}_2}

\newcommand{\dom}{\mathrm{dom}}
\newcommand{\sign}{\mathrm{sign}}
\newcommand{\supt}{{(t)}}
\newcommand{\suptp}{{(t+1)}}

\newcommand{\what}{\widehat w}

\newcommand{\fp}{f'}
\newcommand{\fpp}{f''}
\newcommand{\fppp}{f'''}
\newcommand{\utilde}{\widetilde u}
\newcommand{\vtilde}{\widetilde v}
\newcommand{\ztilde}{\widetilde z}

\newcommand{\gradbound}{G}
\newcommand{\hessianbound}{L}
\newcommand{\tensorbound}{M}
\newcommand{\reg}{\lambda}
\newcommand{\wstar}{w_\star}
\newcommand{\fastpcg}{\mathcal{A}}
\newcommand{\slowpcg}{\bar{\mathcal{A}}}
\newcommand{\mumax}{\mu_{\mathrm{max}}}
\newcommand{\pcd}{P}

\newcommand{\communicateR}[2]
{\makebox{\qquad#1}\hfill $\overrightarrow{\qquad\qquad}$ \hfill \makebox{#2}}
\newcommand{\communicateL}[2]
{\makebox{\qquad#1}\hfill $\overleftarrow{\qquad\qquad}$ \hfill \makebox{#2}}

\title{Communication-Efficient Distributed Optimization of 
       Self-Concordant Empirical Loss}

\author{
Yuchen Zhang\thanks{
    Department of Electrical Engineering and Computer Science, 
    University of California, Berkekey, CA 94720, USA.
    Email: \texttt{yuczhang@eecs.berkeley.edu}.
    (This work was perfomed during an internship at Microsoft Research.)} 
    \and
Lin Xiao\thanks{
    Machine Learning Groups, Microsoft Research, Redmond, WA 98053, USA.
    Email: \texttt{lin.xiao@microsoft.com}.}
}

\date{January 1, 2015}

\begin{document}
\maketitle

\begin{abstract}
We consider distributed convex optimization problems originated from
sample average approximation of stochastic optimization,
or empirical risk minimization in machine learning. 
We assume that each machine in the distributed computing system has access
to a local empirical loss function, constructed with i.i.d.\ data 
sampled from a common distribution. 
We propose a communication-efficient distributed algorithm to minimize
the overall empirical loss, which is the average of the local empirical losses.
The algorithm is based on an inexact damped Newton method, where
the inexact Newton steps are computed by a distributed 
preconditioned conjugate gradient method.
We analyze its iteration complexity and communication efficiency 
for minimizing self-concordant empirical loss functions, 
and discuss the results for distributed ridge regression, 
logistic regression and binary classification with a smoothed hinge loss.
In a standard setting for supervised learning, 
the required number of communication rounds 
of the algorithm does not increase with the sample size, 
and only grows slowly with the number of machines.
\end{abstract}

\section{Introduction}
\label{sec:introduction}

Many optimization problems in data science 
(including statistics, machine learning, data mining, \etc) 
are formulated with a large amount of data as input.
They are typically solved by iterative algorithms which need to access
the whole dataset or at least part of it during each iteration.
With the amount of data we collect and process growing at a fast pace,
it happens more often that the dataset involved in an optimization problem 
cannot fit into the memory or storage of a single computer (machine).
To solve such ``big data'' optimization problems, 
we need to use distributed algorithms that rely on inter-machine communication.

In this paper, we focus on distributed optimization problems generated through 
\emph{sample average approximation} (SAA) of stochastic optimization problems. 
Consider the problem
\begin{equation}\label{eqn:stoch-opt}
    \minimize_{w\in\R^d} \quad \E_z[\phi(w, z)],
\end{equation}
where $z$ is a random vector whose probability distribution is supported
on a set $\mathcal{Z}\subset\R^p$, and the cost function 
$\phi:\R^d\times\mathcal{Z}\to\R$ is convex in~$w$ for every $z\in\mathcal{Z}$.
In general, evaluating the expected objective function with respect to~$z$
is intractable, even if the distribution is given.
The idea of SAA is to approximate the solution to~\eqref{eqn:stoch-opt} by 
solving a deterministic problem defined over a large number of i.i.d.\ 
(independent and identically distributed) samples 
generated from the distribution of~$z$ 
(see, \eg, \cite[Chapter~5]{Shapiro09book}).
Suppose our distributed computing system consists of~$m$ machines,
and each has access to~$n$ samples 
$z_{i,1},\ldots,z_{i,n}$, for $i=1,\ldots,m$.
Then each machine can evaluate a local empirical loss function
\[
    f_i(w) \eqdef \frac{1}{n} \sum_{j=1}^n \phi(w, z_{i,j}),
    \qquad i=1,\ldots,m.
\]
Our goal is to minimize the overall empirical loss defined with
all $mn$ samples:
\begin{equation}\label{eqn:average-obj}
        f(w) \eqdef \frac{1}{m}\sum_{i=1}^m f_i(w) 
        = \frac{1}{mn}\sum_{i=1}^m \sum_{j=1}^n \phi(w, z_{i,j}) .
\end{equation}

In machine learning applications, 
the probability distribution of~$z$ is usually unknown,
and the SAA approach is referred as \emph{empirical risk minimization} (ERM). 
As a concrete example, 
we consider ERM of linear predictors for supervised learning. 
In this case, each sample has the form $z_{i,j}=(x_{i,j},y_{i,j})\in\R^{d+1}$,
where $x_{i,j}\in\R^d$ is a feature vector and
$y_{i,j}$ can be a target response in~$\R$ (for regression)
or a discrete label (for classification).
Examples of the loss function include
\begin{itemize} \itemsep 0pt
\item linear regression: 
    $x\in\R^d$, $y\in\R$, and 
    $\phi(w, (x,y)) = (y - w^T x)^2$.
    \item logistic regression: 
    $x\in\R^d$, $y\in\{+1, -1\}$, and
    $\phi(w,(x,y))=\log(1+\exp(-y(w^T x)))$.
\item hinge loss: 
    $x\in\R^d$, $y\in\{+1,-1\}$, and 
    $\phi(w,(x,y)) = \max\left\{0,\; 1-y(w^T x)\right\}$.
\end{itemize}
For stability and generalization purposes, we often add a regularization term
$(\lambda/2)\|w\|_2^2$ to make the empirical loss function strongly convex.
More specifically, we modify the definition of $f_i(w)$ as
\begin{equation}\label{eqn:regularized-ERM}
    f_i(w) \eqdef \frac{1}{n} \sum_{j=1}^n \phi(w,z_{i,j}) 
    + \frac{\lambda}{2} \|w\|_2^2, \qquad i=1,\ldots,m.
\end{equation}
For example, when~$\phi$ is the hinge loss,
this formulation yields the \emph{support-vector machine} 
\cite{CortesVapnik95svm}.

Since the functions $f_i(w)$ can be accessed only locally,
we consider distributed algorithms that alternate between
a local computation procedure at each machine, and a communication round
involving simple map-reduce type of operations
\cite{Dean08MapReduce,MPIForum}.
Compared with local computation at each machine,
the cost of inter-machine communication is much higher
in terms of both speed/delay and energy consumption
(e.g., \cite{Bekkerman2011scaling,Shalf10});
thus it is often considered as the bottleneck for distributed computing.
Our goal is to develop \emph{communication-efficient} distributed algorithms, 
which try to use a minimal number of communication rounds 
to reach certain precision in minimizing $f(w)$.

\subsection{Communication efficiency of distributed algorithms}
\label{sec:communication-efficiency}

We assume that each communication round requires only simple map-reduce
type of operations, such as 
broadcasting a vector in $\R^d$ to the~$m$ machines and
computing the sum or average of~$m$ vectors in $\R^d$.
Typically, if a distributed iterative algorithm takes~$T$ iterations to 
converge, then it communicates at least~$T$ rounds
(usually one or two communication rounds per iteration).
Therefore, we can measure the communication efficiency of a distributed 
algorithm by its iteration complexity $T(\epsilon)$, which
is the number of iterations required by the algorithm to find a solution
$w_T$ such that $f(w_T)-f(w_\star)\leq\epsilon$.

For a concrete discussion, we make the following assumption:
\begin{assumption} \label{asmp:Hessian-bounds}
   The function $f:\R^d\to\R$ is twice continuously differentiable,
   and there exist constants $L\geq \lambda>0$ such that
   \[
       \lambda I \preceq f''(w) \preceq L I, \qquad \forall\,w\in\R^d,
   \]
   where $f''(w)$ denotes the Hessian of~$f$ at~$w$, and~$I$ is
   the $d\times d$ identity matrix.
\end{assumption}
Functions that satisfy Assumption~\ref{asmp:Hessian-bounds} are often called
$L$-smooth and $\lambda$-strongly convex.
The value $\kappa = L/\lambda \geq 1$ is called the \emph{condition number}
of~$f$, which is a key quantity in characterizing the complexity
of iterative algorithms. 
We focus on ill-conditioned cases where $\kappa\gg 1$.

A straightforward approach for minimizing $f(w)$ is 
distributed implementation of the classical gradient descent method.
More specifically, at each iteration~$k$, 
each machine computes the local gradient $f'_i(w_k)\in\R^d$
and sends it to a master node to compute
$f'(w_k)=(1/m)\sum_{i=1}^m f'_i(w_k)$. 
The master node takes a gradient step to compute $w_{k+1}$, 
and broadcasts it to each machine for the next iteration.
The iteration complexity of this method 
is the same as the classical gradient method:
$\order(\kappa\log(1/\epsilon))$,
which is linear in the condition number~$\kappa$ (\eg, \cite{Nesterov04book}).
If we use accelerated gradient methods 
\cite{Nesterov04book, Nesterov13composite, LinXiao14acclprox}, 
then the iteration complexity can be improved to
$\order(\sqrt{\kappa}\log(1/\epsilon))$.

Another popular technique for distributed optimization is to use the 
alternating direction method of multipliers (ADMM);
see, e.g., \cite[Section~8]{Boyd10ADMM}.
Under the assumption that each local function $f_i$ is $L$-smooth 
and $\lambda$-strongly convex, the ADMM approach can achieve linear convergence,
and the best known complexity is $\order(\sqrt{\kappa}\log(1/\epsilon))$
\cite{DengYin12linearADMM}.
This turns out to be the same order as for accelerated gradient methods.
In this case, ADMM can actually be considered as an accelerated primal-dual 
first-order method; see the discussions in \cite[Section~4]{ChambollePock11}.

The polynomial dependence of the iteration complexity
on the condition number can be unsatifactory.
For machine learning applications, both the precision~$\epsilon$ and 
the regularization parameter~$\lambda$ should decrease while
the overall sample size $mn$ increases, typically on the order of
$\Theta(1/\sqrt{mn})$ (\eg, \cite{BousquetElisseeff02,SSSS09stochastic}).
This translates into the condition number~$\kappa$ being $\Theta(\sqrt{mn})$.
In this case, the iteration complexity, 
and thus the number of communication rounds,
scales as $(mn)^{1/4}$ for both accelerated gradient methods and 
ADMM (with careful tuning of the penalty parameter).
This suggests that the number of communication rounds grows with the total
sample size. 

Despite the rich literature on distributed optimization 
(e.g., \cite{BertsekasTsitsiklis89book, RamNedic2010distributed,
Boyd10ADMM,agarwal2011distributed,duchi2012dual,dekel2012optimal,
recht2011hogwild,zhang2012communication,ShamirSrebro14}),
most algorithms involve high communication cost. 
In particular, their iteration complexity have similar or worse dependency 
on the condition number as the methods discussed above.
It can be argued that the iteration complexity 
$\order(\sqrt{\kappa}\log(1/\epsilon))$
cannot be improved in general for distributed first-order methods --- after all,
it is optimal for centralized first-order methods under the same 
assumption that $f(w)$ is $L$-smooth and $\lambda$-strongly convex
\cite{NemirovskiYudin83book,Nesterov04book}.
Thus in order to obtain better communication efficiency, we need to look into
further problem structure and/or alternative optimization methods. 
And we need both in this paper. 

First, we note that the above discussion on iteration complexity
does not exploit the fact that each function~$f_i$ is generated by,
or can be considered as, SAA of a stochastic optimization problem.
Since the data $z_{i,j}$ are i.i.d.\ samples from a common distribution,
the local empirical loss functions $f_i(w)=(1/n)\sum_{j=1}^n \phi(w, z_{i,j})$
will be similar to each other if the local sample size~$n$ is large. 
Under this assumption, Zhang et al.~\cite{zhang2012communication} studied
a one-shot averaging scheme that approximates the minimizer of function~$f$ 
by simply averaging the minimizers of $f_i$. 
For a fixed condition number,
the one-shot approach is communication efficient because it achieves optimal
dependence on the overall sample size $mn$
(in the sense of statistical lower bounds).
But their conclusion doesn't allow the regularization parameter $\lambda$ 
to decrease to zero as~$n$ goes to infinity 
(see discussions in \cite{ShamirSrebroZhang14DANE}).

Exploiting the stochastic nature alone seems not enough to
overcome ill-conditioning in the regime of first-order methods.
This motivates the development of distributed second-order methods.
Recently, Shamir et al.~\cite{ShamirSrebroZhang14DANE} 
proposed a distributed approximate Newton-type (DANE) method.
Their method takes advantage of the fact that,
under the stochastic assumptions of SAA, 
the Hessians $f''_1,f''_2,\dots,f''_m$ are similar to each other.
For quadratic loss functions, DANE is shown to converge in 
$\widetilde\order\bigl( (L/\lambda)^2 n^{-1}\log(1/\epsilon) \bigr)$ 
iterations with high probability, where the notation $\widetilde\order(\cdot)$
hides additional logarithmic factors involving~$m$ and~$d$.
If $\lambda \sim 1/\sqrt{mn}$ as in machine learning applications, 
then the iteration complexity becomes $\widetilde \order(m\log(1/\epsilon))$, 
which scales linearly with the number of machines~$m$,
not the total sample size $mn$.
However, the analysis in \cite{ShamirSrebroZhang14DANE} does not guarantee 
that DANE has the same convergence rate on non-quadratic functions.

\subsection{Outline of our approach}
\label{sec:outline}

In this paper, we propose a communication-efficient distributed
second-order method for minimizing the overall empirical loss $f(w)$
defined in~\eqref{eqn:average-obj}.
Our method is based on an \emph{inexact} damped Newton method.
Assume $f(w)$ is strongly convex and has continuous second derivatives.
In the \emph{exact} damped Newton method 
(e.g., \cite[Section~4.1.5]{Nesterov04book}),
we first choose an initial point $w_0\in\R^d$, 
and then repeat 
\begin{equation}\label{eqn:damped-Newton}
    w_{k+1} = w_k - \frac{1}{1+\delta(w_k)} \Delta w_k ,
\qquad k=0,1,2,\ldots,
\end{equation}
where $\Delta w_k$ and $\delta(w_k)$ are the Newton step and the Newton
decrement, respectively, defined as
\begin{align}
    \Delta w_k &= [f''(w_k)]^{-1} f'(w_k) \; , \nonumber \\
\delta(w_k) & = \sqrt{f'(w_k)^T [f''(w_k)]^{-1} f'(w_k)} 
=\sqrt{(\Delta w_k)^T f''(w_k) \Delta w_k} \; .
\label{eqn:Newton-decrement}
\end{align}
Since $f$ is the average of $f_1,\ldots,f_m$, 
its gradient and Hessian can be written as
\begin{equation}\label{eqn:Hessian-gradient}
    f'(w_k) = \frac{1}{m}\sum_{i=1}^m f'_i(w_k) , \qquad
    f''(w_k) = \frac{1}{m}\sum_{i=1}^m f''_i(w_k) .
\end{equation}

In order to compute $\Delta w_k$ in a distributed setting, 
the naive approach would require all the 
machines to send their gradients and Hessians to a master node
(say machine~1). 
However, the task of transmitting the Hessians (which are $d\times d$ matrices)
can be prohibitive for large dimensions~$d$.
A better alternative is to use the conjugate gradient (CG) method to compute 
$\Delta w_k$ as the solution to a linear system $f''(w_k) \Delta w_k=f'(w_k)$.
Each iteration of the CG method requires 
a matrix-vector product of the form
\[
    f''(w_k) v = \frac{1}{m} \sum_{i=1}^m f''_i(w_k) v ,
\]
where $v$ is some vector in $\R^d$.
More specifically, the master node can broadcast the vector~$v$ to each 
machine, each machine computes $f''_i(w_k) v\in\R^d$ locally and
sends it back to the master node,
which then forms the average $f''(w_k) v$ and performs the CG update. 
Due to the iterative nature of the CG method, we can only compute 
the Newton direction and Newton decrement approximately, especially
with limited number of communication rounds.

The overall method has two levels of loops: the outer-loop of the damped
Newton method, and the inner loop of the CG method
for computing the inexact Newton steps.
A similar approach (using a distributed truncated Newton method)
was proposed in \cite{ZhuangChinJuanLin14,LinTsaiLeeLin14}
for ERM of linear predictors, 
and it was reported to perform very well in practice.
However, the total number of CG iterations 
(each takes a round of communication) may still be high.

First, consider the outer loop complexity.
It is well-known that Newton-type methods have asymptotic superlinear
convergence. 
However, in classical analysis of Newton's method 
(e.g., \cite[Section~9.5.3]{BoydVandenberghe04book}),
the number of steps needed to reach the superlinear convergence zone
still depends on the condition number;
more specifically, it scales quadratically in~$\kappa$.
To solve this problem, we resort to the machinery of 
self-concordant functions \cite{NesterovNemirovski94book,Nesterov04book}.
For self-concordant empirical losses, we show that the iteration complexity
of the inexact damped Newton method has a much weaker dependence on
the condition number.

Second, consider the inner loop complexity. 
The convergence rate of the CG method also depends on 
the condition number~$\kappa$:
it takes $\order(\sqrt{\kappa}\log(1/\varepsilon))$ CG iterations to
compute an $\varepsilon$-precise Newton step.
Thus we arrive at the dilemma that the overall complexity of the CG-powered
inexact Newton method is no better than accelerated gradient methods or ADMM.
To overcome this difficulty, we exploit the stochastic nature of the problem 
and propose to use a preconditioned CG (PCG) method 
for solving the Newton system. 
Roughly speaking, if the local Hessians $f''_1(w_k),\ldots, f''_m(w_k)$ are 
``similar'' to each other, then we can use any local Hessian 
$f''_i(w_k)$ as a preconditioner. 
Without loss of generality, let $\pcd=f''_1(w_k)+\mu I$,
where~$\mu$ is an estimate of the spectral norm $\|f''_1(w_k)-f''(w_k)\|_2$.
Then we use CG to solve the pre-conditioned linear system
\[
    \pcd^{-1} f''(w_k) \Delta w_k = \pcd^{-1} f'(w_k),
\]
where the preconditioning (multiplication by $\pcd^{-1}$)
can be computed locally at machine~$1$ (the master node).
The convergence rate of PCG depends on the condition number of the matrix
$\pcd^{-1} f''(w_k)$, which is close to~1 if the spectral norm 
$\|f''_1(w_k)-f''(w_k)\|_2$ is small.

To exactly characterize the similarity between $f''_1(w_k)$ and $f''(w_k)$,
we rely on stochastic analysis in the framework of SAA or ERM.
We show that with high probability, $\|f''_1(w_k)-f''(w_k)\|_2$ 
decreases as $\widetilde\order(\sqrt{d/n})$ in general,
and $\widetilde\order(\sqrt{1/n})$ for quadratic loss.
Therefore, when~$n$ is large, the preconditioning is very effective and
the PCG method converges to sufficient precision
within a small number of iterations.
The stochastic assumption is also critical for obtaining an initial
point $w_0$ which further brings down the overall iteration complexity.

Combining the above ideas, we propose and analyze an algorithm for
Distributed Self-Concordant Optimization
(DiSCO, which also stands for Distributed Se\textsc{c}ond-Order method,
or Distributed Stochastic Convex Optimization).
We show that several popular empirical loss functions in machine learning,
including ridge regression, regularized logistic regression and 
a (new) smoothed hinge loss, are actually self-concordant.
For ERM with these loss functions, 
Table~\ref{tab:complexities} lists the number of communication rounds required
by DiSCO and several other algorithms to find an $\epsilon$-optimal solution.
As the table shows, the communication cost of DiSCO weakly depends on the 
number of machines~$m$ and on the feature dimension~$d$, 
and is independent of the local sample size~$n$ (excluding logarithmic factors).
Comparing to DANE \cite{ShamirSrebroZhang14DANE}, 
DiSCO not only improves the communication efficiency on quadratic loss, 
but also handles non-quadratic classification tasks. 

\begin{table}[t]
\renewcommand{\arraystretch}{1.2}
\begin{center}
	\begin{tabular}{|c|c|c|}
	  \hline
			 &	 \multicolumn{2}{c|}{Number of Communication Rounds $\widetilde\order(\cdot)$} \\\cline{2-3}
             Algorithm & Ridge Regression & Binary Classification\\[-1ex]
			& (quadratic loss) & (logistic loss, smoothed hinge loss)\\
            \hline
			Accelerated Gradient & $(mn)^{1/4}\log(1/\epsilon)$ & $(mn)^{1/4}\log(1/\epsilon)$\\
            \hline
			ADMM & $(mn)^{1/4}\log(1/\epsilon)$ & $(mn)^{1/4}\log(1/\epsilon)$\\
            \hline
            DANE \cite{ShamirSrebroZhang14DANE} & $ m \log(1/\epsilon)$ & $(mn)^{1/2}\log(1/\epsilon)$\\\hline
			DiSCO (this paper) & $ m^{1/4} \log(1/\epsilon)$ & $ m^{3/4}d^{1/4} + m^{1/4}d^{1/4}\log(1/\epsilon)$ \\
            \hline
	\end{tabular}
\end{center}
\caption{Communication efficiency of several distributed algorithms for ERM
of linear predictors, when the regularization parameter~$\lambda$
in~\eqref{eqn:regularized-ERM} is on the order of $1/\sqrt{mn}$.
All results are deterministic or high probability upper bounds,
except that the last one, DiSCO for binary classification, 
is a bound in expectation 
(with respect to the randomness in generating the i.i.d.\ samples).
For DiSCO, the dependence on~$\epsilon$ can be improved
to $\log\log(1/\epsilon)$ with superlinear convergence.}
\label{tab:complexities}
\end{table}

The rest of this paper is organized as follows.
In Section~\ref{sec:self-concordance}, 
we review the definition of self-concordant functions, and show that several 
popular empirical loss functions used in machine learning are either 
self-concordant or can be well approximated by self-concordant functions.
In Section~\ref{sec:approx-Newton}, 
we analyze the iteration complexity of an inexact damped Newton method 
for minimizing self-concordant functions.
In Section~\ref{sec:DiSCO-algorithm},
we show how to compute the inexact Newton step 
using a distributed PCG method, describe the overall DiSCO algorithm, 
and discuss its communication complexity.
In Section~\ref{sec:stochastic-analysis}, 
we present our main theoretical results based on stochastic analysis,
and apply them to linear regression and classification.
In Section~\ref{sec:experiments}, 
we report experiment results to illustrate the advantage of DiSCO in 
communication efficiency, compared with other algorithms listed in
Table~\ref{tab:complexities}.
Finally, we discuss the extension of DiSCO to distributed minimization of
composite loss functions in Section~\ref{sec:composite-minimization},
and conclude the paper in Section~\ref{sec:conclusion}.

\section{Self-concordant empirical loss}
\label{sec:self-concordance}

The theory of self-concordant functions were developed by 
Nesterov and Nemirovski for the analysis of interior-point methods
\cite{NesterovNemirovski94book}.
Roughly speaking, 
a function is called self-concordant if its third derivative can be controlled,
in a specific way, by its second derivative.
Suppose the function $f: \R^d\to \R$ has continuous third derivatives.
We use $\fpp(w)\in \R^{d\times d}$ to denote its Hessian at~$w\in \R^d$, 
and use $\fppp(w)[u]\in \R^{d\times d}$ to denote the limit
\[
\fppp(w)[u] \eqdef \lim_{t\to 0} \frac{1}{t}\big(\fpp(w+t u) - \fpp(w)\big).
\]

\begin{definition}\label{def:self-concordance}
A convex function $f:\R^d\to\R$ is self-concordant with parameter $M_f$ 
if the inequality
\begin{align*}
  \left|u^T (\fppp(w)[u]) u \right| \leq M_f \left( u^T \fpp(w) u \right)^{3/2}
\end{align*}
holds for any $w\in\dom(f)$ and $u\in\R^d$. 
In particular, a self-concordant function with parameter $2$
is called standard self-concordant.
\end{definition}

The reader may refer to the books 
\cite{NesterovNemirovski94book,Nesterov04book} 
for detailed treatment of self-concordance. 
In particular, the following lemma \cite[Corollary~4.1.2]{Nesterov04book}
states that any self-concordant function can be rescaled to become 
standard self-concordant.

\begin{lemma}\label{lemma:rescale-self-concordance}
If a function $f$ is self-concordant with parameter $M_f$, 
then $\frac{M_f^2}{4}f$ is standard self-concordant (with parameter $2$).
\end{lemma}

In the rest of this section, we show that several popular regularized 
empirical loss functions for linear regression and binary classification
are either self-concordant 
or can be well approximated by self-concordant functions.

First we consider regularized linear regression (ridge regression) with
\[
f(w) = \frac{1}{N}\sum_{i=1}^N (y_i-w^T x_i)^2 + \frac{\lambda}{2}\|w\|_2^2 .
\]
To simplify notation, here we use a single subscript~$i$ running from~$1$ to 
$N=mn$, instead of the double subscripts $\{i,j\}$ used in the introduction.
Since~$f$ is a quadratic function, its third derivatives are all zero.
Therefore, it is self-concordant with parameter~0,
and by definition is also standard self-concordant.

For binary classification, we consider the following regularized
empirical loss function
\begin{align}\label{eqn:binary-class-erm}
	\ell(w) \eqdef \frac{1}{N} \sum_{i=1}^N \varphi(y_i w^T x_i) + \frac{\gamma}{2}\ltwos{w}^2,
\end{align}
where $x_i\in \mathcal{X}\subset\R^d$, $y_i\in\{-1,1\}$, and 
$\varphi: \R\to\R$ is a convex surrogate function for the binary loss function 
which returns~$0$ if $y_{i}=\sign(w^T x_{i})$ and~$1$ otherwise.
We further assume that the elements of $\mathcal{X}$ are bounded, that is,
we have $\sup_{x\in \mathcal{X}} \ltwos{x} \leq B$ for some finite $B$. 
Under this assumption, the following lemma shows that the regularized 
loss $\ell(w)$ is self-concordant.

\begin{lemma}
\label{lemma:regularized-self-concordance}
Assume that $\gamma > 0$ and there exist $Q>0$ and $\alpha \in [0,1)$ such that
$|\varphi'''(t)| \leq Q (\varphi''(t))^{1-\alpha}$ for every $t\in\R$. Then:
\begin{enumerate}
\item[(a)] The function~$\ell(w)$ defined by 
    equation~\eqref{eqn:binary-class-erm} is self-concordant with parameter 
    $\frac{B^{1+2\alpha} Q}{\gamma^{1/2+\alpha} }$.
\item[(b)] The scaled function 
    $f(w)=\frac{B^{2+4\alpha} Q^2}{4 \gamma^{1+2\alpha} }\ell(w)$ 
    is standard self-concordant.
\end{enumerate}
\end{lemma}

\begin{proof}
We need to bound the third derivative of $\ell$ appropriately.
Using equation~\eqref{eqn:binary-class-erm} and the assumption on~$\varphi$, 
we have
\begin{align*}
	\left| u^T (\ell'''(w)[u]) u \right| 
    &\leq \frac{1}{N}\sum_{i=1}^N \left|\varphi'''(y_i w^Tx_i) ( y_i u^T x_i)^3 \right| \\
    &\stackrel{(i)}{\leq} \frac{Q}{N}\sum_{i=1}^N \left((u^T x_i)^2\varphi''(y_i w^T x_i)\right)^{1-\alpha} (B \ltwos{u})^{1+2\alpha} \\
    &\stackrel{(ii)}{\leq} B^{1+2\alpha} Q \biggl( \frac{1}{N} \sum_{i=1}^N (u^T x_i)^2 \varphi''(y_i w^T x_i) \biggr)^{1-\alpha} (\ltwos{u})^{1+2\alpha} \\
    &\stackrel{(iii)}{\leq} B^{1+2\alpha} Q\, \left(u^T \ell''(w) u\right)^{1-\alpha} (\ltwos{u})^{1+2\alpha} .
\end{align*}
In the above derivation, 
inequality (i) uses the property that $|y_i|=1$ and $|u^T x_i|\leq B\ltwos{u}$,
inequality (ii) uses H\"{o}lder's inequality and concavity of 
$(\cdot)^{1-\alpha}$, and inequality (iii)
uses the fact that the additional regularization term in $\ell(w)$ is convex.

Since~$\ell$ is $\gamma$-strongly convex, we have 
$u^T \ell''(w) u \geq \gamma \ltwos{u}^2$. 
Thus, we can upper bound $\ltwos{u}$ by 
$\ltwos{u} \leq \gamma^{-1/2}(u^T \ell''(w) u)^{1/2}$. 
Substituting this inequality into the above upper bound completes the proof 
of part~(a). Given part~(a), part~(b) follows immediately from 
Lemma~\ref{lemma:rescale-self-concordance}.
\end{proof}

It is important to note that the self-concordance of~$\ell$ essentially 
relies on the regularization parameter~$\gamma$ being positive. 
If $\gamma = 0$, then the function will no longer be self-concordant, 
as pointed out by Bach~\cite{Bach10selfconcordance} on logistic regression.
Since we have the freedom to choose $\varphi$, 
Lemma~\ref{lemma:regularized-self-concordance} handles 
a broad class of empirical loss functions.
Next, we take the logistic loss and a smoothed hinge loss
as two concrete examples.

\paragraph{Logistic regression}
For logistic regression, we minimize the objective 
function~\eqref{eqn:binary-class-erm}
where $\varphi$ is the logistic loss: $\varphi(t) = \log(1+e^{-t})$. 
We can calculate the second and the third derivatives of $\varphi(t)$:
\begin{align*}
	\varphi''(t) &= \frac{e^t}{(e^t+1)^2}\;,\\
	\varphi'''(t) &= \frac{e^t(1-e^t)}{(e^t+1)^3} = \frac{1-e^t}{1+e^t}\varphi''(t)\;.
\end{align*}
Since $|\frac{1-e^t}{1+e^t}| \leq 1$ for all $t\in \R$, we conclude that 
$|\varphi'''(t)|\leq\varphi''(t)$ for all $t\in\R$.
This implies that the condition in 
Lemma~\ref{lemma:regularized-self-concordance} 
holds with $Q=1$ and $\alpha = 0$.
Therefore, the regularized empirical loss $\ell(w)$ is self-concordant with
parameter $B/\sqrt{\gamma}$, and the scaled loss function 
$f(w) = (B^2/(4\gamma))\ell(w)$ is standard self-concordant.

\begin{figure}
\centering
\includegraphics[width = 0.45\textwidth]{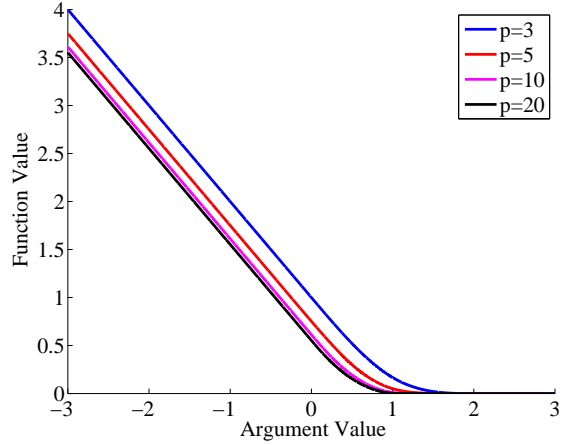}
\caption{Smoothed hinge loss $\varphi_p$ with $p = 3,5,10,20$. }
\label{fig:smooth-hinge}
\end{figure}

\paragraph{Smoothed hinge loss}
In classification tasks, it is sometimes more favorable to use the hinge loss 
$\varphi(t)=\max\{0, 1-t\}$ than using the logistic loss.
We consider a family of smoothed hinge loss functions $\varphi_p$ parametrized by a positive number $p \geq 3$.
The function is defined by
\begin{equation}\label{eqn:smoothed-hinge}
	\varphi_p(t) = \left\{
	\begin{array}{ll}
        \frac{3}{2} - \frac{p-2}{p-1} - t & \mbox{for $t < - \frac{p-3}{p-1}$},\\[1ex]
        \frac{3}{2} - \frac{p-2}{p-1} - t + \frac{(t + (p-3)/(p-1))^p}{p(p-1)} & \mbox{for $- \frac{p-3}{p-1} \leq t <  1 - \frac{p-3}{p-1}$},\\[1ex]
        \frac{p+1}{p(p-1)} - \frac{t}{p-1} + \frac{1}{2}(1-t)^2 & \mbox{for $1 - \frac{p-3}{p-1} \leq t < 1$},  \\[1ex]
        \frac{(2-t)^p}{p(p-1)} & \mbox{for $1 \leq t < 2$}, \\[1ex]
		0 & \mbox{for $t \geq 2$}.
	\end{array}
	\right.
\end{equation}
We plot the functions $\varphi_p$ for $p=3,5,10,20$ on 
Figure~\ref{fig:smooth-hinge}. 
As the plot shows, $\varphi_p(t)$ is zero for $t > 2$, 
and it is a linear function with unit slope for $t < - \frac{p-3}{p-1}$. 
These two linear zones are connected by three smooth non-linear segments
on the interval $[- \frac{p-3}{p-1},2]$.

The smoothed hinge loss $\varphi_p$ satisfies the condition of 
Lemma~\ref{lemma:regularized-self-concordance} with
$Q = p-2$ and $\alpha = \frac{1}{p-2}$.
To see this, we note that the third derivative of $\varphi_p(t)$ 
is nonzero only when 
$t\in [- \frac{p-3}{p-1}, 1- \frac{p-3}{p-1}]$ and when $t\in [1,2]$. 
On the first interval, we have
\[
	\varphi''_p(t) = \left(t + \frac{p-3}{p-1}\right)^{p-2}, \qquad
    \varphi'''_p(t) = (p-2)\left(t + \frac{p-3}{p-1}\right)^{p-3}.
\]
On the second interval, we have
\[
	\varphi''_p(t) = \left(2-t\right)^{p-2}, \qquad
    \varphi'''_p(t) = -(p-2)\left(2-t\right)^{p-3}.
\]
For both cases we have the inequality 
\[
    |\varphi'''_p(t)| \leq (p-2) (\varphi''_p(t))^{1 - \frac{1}{p-2}},
\]
which means $Q=p-2$ and $\alpha=\frac{1}{p-2}$.
Therefore, according to Lemma~\ref{lemma:regularized-self-concordance},
the regularized empirical loss $\ell(w)$ is self-concordant with parameter
\begin{equation}\label{eqn:smoothed-sc-p}
M_p = \frac{(p-2)B^{1+\frac{2}{p-2}}}{\gamma^{\frac{1}{2}+\frac{1}{p-2}}} \; ,
\end{equation}
and the scaled loss function $f(w) = (M_p^2/4)\ell(w)$ is standard 
self-concordant.


\section{Inexact damped Newton method}
\label{sec:approx-Newton}

\begin{algorithm}[t]
\DontPrintSemicolon
\vspace{2pt}
\textbf{input:} initial point $w_0$
    and specification of a nonnegative sequence $\{\epsilon_k\}$. \;
\vspace{2pt}
\textbf{repeat} for $k = 0,1,2,\dots$ 
\begin{enumerate} \itemsep 0pt
  \item Find a vector $v_k$ such that 
   $\ltwos{\fpp(w_k)v_k - \fp(w_k)}\leq\epsilon_k$.
  \item Compute $\delta_k=\sqrt{v_k^T f''(w_k) v_k}$ 
        and update $w_{k+1} = w_k - \frac{1}{1 + \delta_k}v_k$.
\end{enumerate}
\vspace{-5pt}
\textbf{until} a stopping criterion is satisfied.
\caption{Inexact damped Newton method}
\label{alg:damped-newton-method}
\end{algorithm}

In this section, we propose and analyze an inexact damped Newton method
for minimizing self-concordant functions. 
Without loss of generality, we assume the objective function
$f:\R^d\to\R$ is standard self-concordant.
In addition, we assume that Assumption~\ref{asmp:Hessian-bounds} holds.
Our method is described in Algorithm~\ref{alg:damped-newton-method}.
If we let $\epsilon_k=0$ for all $k\geq0$, then $v_k=[f''(w_k)]^{-1} f'(w_k)$
is the exact Newton step and $\delta_k$ is the Newton decrement
defined in~\eqref{eqn:Newton-decrement},
so the algorithm reduces to the exact damped Newton method
given in~\eqref{eqn:damped-Newton}.
But here we allow the computation of the Newton step 
(hence also the Newton decrement)
to be inexact and contain approximation errors. 

The explicit account of approximation errors is essential for 
distributed optimization.
In particular, if~$f(w)=(1/m)\sum_{i=1}^m f_i(w)$ and the components~$f_i$ 
locate on separate machines, then we can only perform Newton updates 
approximately with limited communication budget.
Even in a centralized setting on a single machine, 
analysis of approximation errors can be important
if the Newton system is solved by iterative algorithms such as 
the conjugate gradient method.

Before presenting the convergence analysis, 
we need to introduce two auxiliary functions
\begin{align*}
    \omega(t)   & = t - \log(1+t), \qquad  t\geq 0, \\
    \omega_*(t) & = -t - \log(1-t), \qquad 0\leq t < 1.
\end{align*}
These two functions are very useful for characterizing the properties of 
self-concordant functions; 
see \cite[Section~4.1.4]{Nesterov04book} for a detailed account.
Here, we simply note that $\omega(0)=\omega_*(0)=0$,
both are strictly increasing for $t\geq 0$,
and $\omega_*(t)\to\infty$ as $t\to 1$.

We also need to define two auxiliary vectors
\begin{align*}
    \utilde_k & = [\fpp(w_k)]^{-1/2}\fp(w_k), \\
    \vtilde_k & = [\fpp(w_k)]^{1/2}v_k.
\end{align*}
The norm of the first vector,
$\|\utilde_k\|_2=\sqrt{f'(w_k)^T [f''(w_k)]^{-1} f'(w_k)}$,
is the exact Newton decrement.
The norm of the second one is $\ltwos{\vtilde_k}=\delta_k$,
which is computed during each iteration of 
Algorithm~\ref{alg:damped-newton-method}.
Note that we do \emph{not} compute $\utilde_k$ or $\vtilde_k$
in Algorithm~\ref{alg:damped-newton-method}.
They are introduced solely for the purpose of convergence analysis.
The following Theorem is proved in 
Appendix~\ref{sec:proof-newton-method-convergence-rate}.

\begin{theorem}\label{thm:damped-newton-convergence}
    Suppose $f:\R^d\to\R$ is a standard self-concordant function 
    and Assumption~\ref{asmp:Hessian-bounds} holds. 
    If we choose the sequence $\{\epsilon_k\}_{k\geq 0}$ in
    Algorithm~\ref{alg:damped-newton-method} as
\begin{align}\label{eqn:choose-epsilon-k}
	\epsilon_k =  \beta (\reg/\hessianbound)^{1/2}\ltwos{\fp(w_k)}\quad \mbox{with $\beta = 1/20$},
\end{align}
then:
\begin{enumerate}
    \item[(a)] For any $k\geq 0$, we have $f(w_{k+1}) \leq f(w_k) - \frac{1}{2}\omega(\ltwos{\utilde_k})$.
    \item[(b)] If $\ltwos{\utilde_k} \leq 1/6$, then we have $\omega(\ltwos{\utilde_{k+1}}) \leq \frac{1}{2}\omega(\ltwos{\utilde_{k}})$.
\end{enumerate}
\end{theorem}

As mentioned before, when $\epsilon_k=0$, 
the vector $v_k=[f''(w_k)]^{-1}f'(w_k)$ becomes the exact Newton step.
In this case, we have $\vtilde_k=\utilde_k$, and
it can be shown that 
$f(w_{k+1})\leq f(w_k) - \omega(\ltwos{\utilde_k})$ for all $k\geq 0$
and the exact damped Newton method has quadratic convergence 
when $\ltwos{\utilde_k}$ is small (see \cite[Section~4.1.5]{Nesterov04book}).
With the approximation error~$\epsilon_k$ specified 
in~\eqref{eqn:choose-epsilon-k},  we have
\begin{align*}
\ltwos{\vtilde_k - \utilde_k} &\leq \ltwos{(\fpp(w_k))^{-1/2}}\ltwos{\fpp(w_k)v_k - \fp(w_k)} \leq \reg^{-1/2} \epsilon_k \\&
	= \beta \hessianbound^{-1/2} \ltwos{\fp(w_k)} \leq \beta \ltwos{\utilde_k},
\end{align*}
which implies
\begin{equation}\label{eqn:bound-vk-by-uk}
    (1-\beta)\ltwos{\utilde_k} \leq \ltwos{\vtilde_k} \leq (1+\beta)\ltwos{\utilde_k}.
\end{equation}
Appendix~\ref{sec:proof-newton-method-convergence-rate} shows that 
when~$\beta$ is sufficiently small, the above inequality leads to 
the conclusion in part~(a).
Compared with the exact damped Newton method,
the guaranteed reduction of the objective value per iteration is cut by half.

Part~(b) of Theorem~\ref{thm:damped-newton-convergence} suggests a linear rate
of convergence when $\ltwos{\utilde_k}$ is small. 
This is slower than the quadratic convergence rate of the exact damped
Newton method, due to the allowed approximation errors in computing 
the Newton step. 
Actually, superlinear convergence can be established if 
we set the tolerances $\epsilon_k$ to be small enough;
see Appendix~\ref{sec:superlinear-convergence} for detailed analysis.
However, when~$v_k$ is computed through a distributed iterative algorithm
(like the distributed PCG algorithm in Section~\ref{sec:distributed-pcg}),
a smaller~$\epsilon_k$ would require more local computational effort 
and more rounds of inter-machine communication.
The choice in equation~\eqref{eqn:choose-epsilon-k} is a reasonable trade-off
in practice.

Using Theorem~\ref{thm:damped-newton-convergence}, we can derive
the iteration complexity of Algorithm~\ref{alg:damped-newton-method}
for obtaining an arbitrary accuracy. 
We present this result as a corollary.

\begin{corollary}\label{coro:damped-newton-complexity}
    Suppose $f:\R^d\to\R$ is a standard self-concordant function 
    and Assumption~\ref{asmp:Hessian-bounds} holds. 
    If we choose the sequence $\{\epsilon_k\}$ 
    in Algorithm~\ref{alg:damped-newton-method}
    as in~\eqref{eqn:choose-epsilon-k},
    then for any $\epsilon >0$, 
    we have $f(w_k)-f(w_\star)\leq\epsilon$ whenever $k\geq K$ where
\begin{align}\label{eqn:damped-newton-complexity}
	K = \left\lceil \frac{f(w_0) - f(\wstar)}{\frac{1}{2}\omega(1/6)} \right\rceil + \left\lceil \log_2 \Big( \frac{ 2\,\omega(1/6)}{\epsilon}\Big) \right\rceil .
\end{align}
Here $\lceil t\rceil$ denotes the smallest \emph{nonnegative} integer
that is larger or equal to~$t$.
\end{corollary}

\begin{proof}
Since $\omega(t)$ is strictly increasing for $t\geq0$, 
part~(a) of Theorem~\ref{thm:damped-newton-convergence} implies that
if $\|\utilde_k\|_2>1/6$, one step of Algorithm~\ref{alg:damped-newton-method}
decreases the value of $f(w)$ by at least a constant $\frac{1}{2}\omega(1/6)$.
So within at most 
$K_1 = \lceil \frac{f(w_0) - f(\wstar)}{\frac{1}{2}\omega(1/6)} \rceil$ 
iterations, we are guaranteed that $\ltwos{\utilde_k} \leq 1/6$. 

According to \cite[Theorem~4.1.13]{Nesterov04book}, 
if $\ltwos{\utilde_k}<1$, then we have
\begin{equation}\label{eqn:opt-gap-bounds}
    \omega(\ltwos{\utilde_k}) \leq f(w_k)-f(w_\star) \leq \omega_*(\ltwos{\utilde_k}) .
\end{equation}
Moreover, it is easy to check that $\omega_*(t)\leq 2\,\omega(t)$ for 
$0\leq t\leq 1/6$.
Therefore, using part~(b) of Theorem~\ref{thm:damped-newton-convergence},
we conclude that when $k\geq K_1$, 
\begin{align*}
    f(w_k) - f(w_\star) ~\leq~ 2\,\omega(\ltwos{\utilde_k})
    ~\leq~ 2 (1/2)^{k-K_1} \omega(\ltwos{\utilde_{K_1}})
    ~\leq~ 2 (1/2)^{k-K_1} \omega(1/6) .
\end{align*}
Bounding the right-hand side of the above inequality by~$\epsilon$, 
we have $f(w_k)-f(w_\star)\leq\epsilon$ whenever
$k\geq K_1+\left\lceil\log_2\left(\frac{2\,\omega(1/6)}{\epsilon}\right) \right\rceil = K$,
which is the desired result.

We note that when $\ltwos{\utilde_k}\leq 1/6$ (as long as $k\geq K_1$), 
we have $f(w_k)-f(\wstar)\leq 2\,\omega(1/6)$.
Thus for $\epsilon>2\,\omega(1/6)$, it suffices to have $k\geq K_1$.
\end{proof}

\subsection{Stopping criteria}

We discuss two stopping criteria for Algorithm~\ref{alg:damped-newton-method}.
The first one is based on the strong convexity of~$f$, which leads to
the inequality (e.g., \cite[Theorem~2.1.10]{Nesterov04book})
\[
    f(w_k)-f(w_\star) \leq \frac{1}{2\lambda}\|f'(w_k)\|_2^2 .
\]
Therefore, we can use the stopping criterion 
$\|f'(w_k)\|_2 \leq \sqrt{2\lambda\epsilon}$,
which implies $f(w_k)-f(w_\star)\leq\epsilon$.
However, this choice can be too conservative in practice
(see discussions in \cite[Section~9.1.2]{BoydVandenberghe04book}).

Another choice for the stopping criterion is based on self-concordance.
Using the fact that $\omega_*(t)\leq t^2$ for $0\leq t\leq 0.68$
(see \cite[Section~9.6.3]{BoydVandenberghe04book}), we have
\begin{equation}\label{eqn:self-concordant-stop}
    f(w_k)-f(w_\star) \leq \omega_*(\ltwos{\utilde_k}) 
    \leq \ltwos{\utilde_k}^2
\end{equation}
provided $\ltwos{\utilde_k}\leq 0.68$.
Since we do not compute $\ltwos{\utilde_k}$ (the exact Newton decrement)
directly in Algorithm~\ref{alg:damped-newton-method}, 
we can use~$\delta_k$ as an approximation.
Using the inequality~\eqref{eqn:bound-vk-by-uk}, and noticing that 
$\ltwos{\vtilde_k}=\delta_k$, we conclude that
\[
    \delta_k \leq (1-\beta) \sqrt{\epsilon} 
\]
implies $f(w_k)-f(\wstar)\leq\epsilon$ when $\epsilon\leq 0.68^2$.
Since~$\delta_k$ is computed at each iteration of 
Algorithm~\ref{alg:damped-newton-method}, this can serve as a good 
stopping criterion.

\subsection{Scaling for non-standard self-concordant functions}
\label{sec:inexact-newton-scaling}

In many applications, we need to deal with empirical loss functions that are
not standard self-concordant; 
see the examples in Section~\ref{sec:self-concordance}.
Suppose a regularized loss function~$\ell(w)$ is self-concordant with 
parameter~$M_\ell>2$. 
By Lemma~\ref{lemma:rescale-self-concordance}, the scaled function
$f = \eta \ell$ with $\eta = M_\ell^2/4$ is standard self-concordant.
We can apply Algorithm~\ref{alg:damped-newton-method} 
to minimize the scaled function~$f$, and rewrite it
in terms of the function~$\ell$ and the scaling constant~$\eta$.

Using the sequence $\{\epsilon_k\}$ defined in~\eqref{eqn:choose-epsilon-k}, 
the condition for computing~$v_k$ in Step~1 is
\[
    \|f''(w_k) v_k - f'(w_k)\|_2 \leq \beta (\lambda/L)^{1/2} \|f'(w_k)\|_2.
\]
Let $\lambda_\ell$ and $\hessianbound_\ell$ be the strong convexity and 
smoothness parameters of the function $\ell$. 
With the scaling, we have $\lambda=\eta\lambda_\ell$ and $L=\eta L_\ell$, 
thus their ratio (the condition number) does not change.
Therefore the above condition is equivalent to
\begin{align}\label{eqn:condition-for-vk-before-rescaling}
	\ltwos{\ell''(w_k)v_k - \ell'(w_k)} \leq \beta(\lambda_{\ell}/\hessianbound_\ell)^{1/2} \ltwos{\ell'(w_k)} .
\end{align}
In other words, the precision requirement in Step~1 is \emph{scaling invariant}.

Step~2 of Algorithm~\ref{alg:damped-newton-method} can be rewritten as
\begin{align}\label{eqn:update-wk-before-rescaling}
 w_{k+1} = w_k - \frac{v_k}{1 + \sqrt{\eta} \cdot \sqrt{v_k^T\ell''(w_k)v_k}}.
\end{align}
Here, the factor~$\eta$ explicitly appears in the formula. 
By choosing a larger scaling factor~$\eta$, 
the algorithm chooses a smaller stepsize. 
This adjustment is intuitive because the convergence of Newton-type method 
relies on local smoothness conditions. 
By multiplying a large constant to $\ell$, the function's Hessian becomes 
less smooth, so that the stepsize should shrink.

In terms of complexity analysis, 
if we target to obtain $\ell(w_k)-\ell(w_\star)\leq\epsilon$,
then the iteration bound in~\eqref{eqn:damped-newton-complexity} becomes
\begin{equation}\label{eqn:scaled-complexity}
    \left\lceil 
    \frac{\eta \bigl( \ell(w_0) - \ell(\wstar)\bigr)}{\frac{1}{2}\omega(1/6)} 
    \right\rceil  
     + \left\lceil 
     \log_2 \Big( \frac{ 2\,\omega(1/6)}{\eta\epsilon}\Big)
    \right\rceil .
\end{equation}
For ERM problems in supervised learning, the self-concordant parameter~$M_\ell$, 
and hence the scaling factor $\eta=M_\ell^2/4$, can grow with the number of samples.
For example, the regularization parameter~$\gamma$ 
in~\eqref{eqn:binary-class-erm} often scales as $1/\sqrt{N}$ 
where~$N=mn$ is the total number of samples.
Lemma~\ref{lemma:regularized-self-concordance} suggests that~$\eta$ grows
on the order of $\sqrt{mn}$. 
A larger~$\eta$ will render the second term in~\eqref{eqn:scaled-complexity}
less relevant, but the first term grows with the sample size $mn$.
In order to counter the effect of the growing scaling factor,
we need to choose the initial point~$w_0$ judiciously to guarantee a small
initial gap. This will be explained further in the next sections.


\section{The DiSCO algorithm}
\label{sec:DiSCO-algorithm}

In this section, we adapt the inexact damped Newton method 
(Algorithm~\ref{alg:damped-newton-method}) to a distributed system,
in order to minimize
\begin{equation}\label{eqn:dist-average-obj}
    f(w) = \frac{1}{m}\sum_{i=1}^m f_i(w),
\end{equation}
where each function~$f_i$ can only be evaluated locally at machine~$i$
(see background in Section~\ref{sec:introduction}).
This involves two questions: 
(1) how to set the initial point $w_0$ and 
(2) how to compute the inexact Newton step $v_k$ in a distributed manner.
After answering these two questions, we will present the overall
DiSCO algorithm and analyze its communication complexity.

\subsection{Initialization}
\label{sec:DiSCO-initialization}

In accordance with the averaging structure in~\eqref{eqn:dist-average-obj},
we choose the initial point based on averaging. 
More specifically, we let 
\begin{equation}\label{eqn:initialization}
	w_0 = \frac{1}{m}\sum_{i=1}^m \what_i ,
\end{equation}
where each $\what_i$ is the solution to a local optimization problem at 
machine~$i$:
\begin{align}\label{eqn:local-solution}
    \what_i = \arg\min_{w\in\R^d} 
    \left\{ f_i(w) + \frac{\rho}{2}\ltwos{w}^2 \right\},
    \qquad i=1,\ldots,m.
\end{align}
Here $\rho\geq 0$ is a regularization parameter, which we will discuss
in detail in the context of stochastic analysis
in Section~\ref{sec:stochastic-analysis}. 
Roughly speaking, if each~$f_i$ is constructed 
with~$n$ i.i.d.\ samples as in~\eqref{eqn:regularized-ERM},
then we can choose $\rho\sim 1/\sqrt{n}$ to make
$\E[f(w_0) - f(\wstar)]$ decreasing as $\order(1/\sqrt{n})$.
In this section, we simply regard it as an input parameter.

Here we comment on the computational cost of 
solving~\eqref{eqn:local-solution} locally at each machine.
Suppose each $f_i(w)$ has the form in~\eqref{eqn:regularized-ERM}, then
the local optimization problems in~\eqref{eqn:local-solution} become
\begin{equation}\label{eqn:local-init}
 \what_i = \arg\min_{w\in\R^d} ~\biggl\{ \frac{1}{n}\sum_{j=1}^n\phi(w,z_{i,j})
 +\frac{\lambda+\rho}{2} \ltwos{w}^2 \biggr\}, \qquad i=1,\ldots,m.
\end{equation}
The finite average structure of the above objective function can be effectively
exploited by the stochastic average gradient (SAG) method
\cite{LeRouxSchmidtBach12,SchmidtLeRouxBach13} or its new variant SAGA
\cite{DefazioBach14SAGA}.
Each step of these methods processes only one component function 
$\phi(w,z_{i,j})$, picked uniformly at random.
Suppose $f_i(w)$ is $L$-smooth, then SAG returns an $\epsilon$-optimal solution 
with $\order\bigl((n+\frac{L+\rho}{\lambda+\rho})\log(1/\epsilon)\bigr)$
steps of stochastic updates. 
For ERM of linear predictors, we can also use the stochastic dual coordinate 
ascent (SDCA) method \cite{SSZhang13SDCA}, which has the same complexity. 
We also mention some recent progress in accelerated stochastic
coordinate gradient methods 
\cite{SSZhang13acclSDCA,LinLuXiao14APCG,ZhangXiao14SPDC},
which can be more efficient both in theory and practice.


\subsection{Distributed computing of the inexact Newton step}
\label{sec:distributed-pcg}

\begin{algorithm}[t]
\DontPrintSemicolon
\vspace{0.5ex}
\underline{\emph{master machine} ($i=1$)}
\hfill \underline{\emph{machines $i=1,\ldots,m$}}

\vspace{1ex}
\textbf{input:} $w_k\in\R^d$ and $\mu\geq 0$.\;
\qquad let $H=f''(w_k)$ and $\pcd=f''_1(w_k)+\mu I$.\;
\vspace{1ex}
\textbf{communication:}\;
\communicateR{broadcasts $w_k$ to other machines;\qquad\quad}{compute $f'_i(w_k)$}\\
\communicateL{aggregate $f'_i(w_k)$ to form $f'(w_k)$.\qquad}{\qquad\qquad\qquad}

\vspace{1ex}
\textbf{initialization:}
compute $\epsilon_k$ given in~\eqref{eqn:choose-epsilon-k} and set\;
\vspace{0.5ex}

\qquad
$v^{(0)} = 0$, \qquad\qquad
$s^{(0)} = \pcd^{-1} r^{(0)}$,  \\
\qquad
$r^{(0)} = f'(w_k)$, \qquad
\!\!$u^{(0)} = s^{(0)}$.\;

\vspace{1ex}
\Repeat(~for $t= 0,1,2\dots,$){$\ltwos{r^{(t+1)}}\leq \epsilon_k$}
{
\vspace{-1ex}
\begin{enumerate}
    \item \textbf{communication:}\;
        \communicateR{broadcast $u^{(t)}$ and $v^{(t)}$;~\quad\qquad\qquad}{compute $f''_i(w_k)u^{(t)}$\qquad\qquad}\\
        \communicateL{aggregate to form $H u^{(t)}$ and $H v^{(t)}$.}{compute $f''_i(w_k)v^{(t)}$\qquad\qquad}
	\item compute 
		$\alpha_t = \textstyle\frac{\langle r^{(t)}, s^{(t)} \rangle}{\langle u^{(t)}, H u^{(t)} \rangle}$ and update\\
        \vspace{1ex}
        \qquad $v^{(t+1)} = v^{(t)} + \alpha_t u^{(t)}$, \\[1ex]
        \qquad $r^{(t+1)} = r^{(t)} - \alpha_t H u^{(t)}$.
	\item compute $\beta_t= \frac{\langle r^{(t+1)}, s^{(t+1)} \rangle}{\langle r^{(t)}, s^{(t)} \rangle}$ and update \\
        \vspace{1ex}
        \qquad $s^{(t+1)}  = \pcd^{-1} r^{(t+1)}$, \\[1ex]
        \qquad $u^{(t+1)} = s^{(t+1)} + \beta_t u^{(t)}$.
\end{enumerate}
\vspace{-1ex}
}
\vspace{1ex}
\Return $v_k = v^{(t+1)}$, $r_k=r^{(t+1)}$, and
$\delta_k = \sqrt{v_k^T H v^{(t)}+\alpha^{(t)} v_k^T H u^{(t)}}$. \;
\caption{Distributed PCG algorithm (given $w_k$ and~$\mu$, compute $v_k$ and $\delta_k$)}
\label{alg:distributed-pcg}
\end{algorithm}

In each iteration of Algorithm~\ref{alg:damped-newton-method},
we need to compute an inexact Newton step~$v_k$ such that 
$\|f''(w_k)v_k-f'(w_k)\|_2 \leq \epsilon_k$. This boils down to 
solving the Newton system $f''(w_k)v_k=f'(w_k)$ approximately.
When the objective~$f$ has the averaging form~\eqref{eqn:dist-average-obj}, 
its Hessian and gradient are given in~\eqref{eqn:Hessian-gradient}.
In the setting of distributed optimization,
we propose to use a preconditioned conjugate gradient (PCG) method 
to solve the Newton system.

To simplify notation, we use $H$ to represent $\fpp(w_k)$ and 
use $H_i$ to represent $\fpp_i(w_k)$.
Without loss of generality, 
we define a preconditioning matrix using the local Hessian at the first machine
(the master node):
\begin{align*}
  \pcd \eqdef H_1 + \mu I,
\end{align*}
where $\mu>0$ is a small regularization parameter.
Algorithm~\ref{alg:distributed-pcg} describes our distributed PCG method for
solving the preconditioned linear system
\[
    \pcd^{-1} H v_k = \pcd^{-1} \fp(w_k) .
\]
In particular, the master machine carries out the main steps of 
the classical PCG algorithm (e.g., \cite[Section~10.3]{GolubVanLoan96book}), 
and all machines (including the master) compute the local gradients 
and Hessians and perform matrix-vector multiplications.
Communication between the master and other machines are used to
form the overall gradient $f'(w_k)$ and 
the matrix-vector products 
\[
    H u^{(t)} = \frac{1}{m}\sum_{i=1}^m f''_i(w_k) u^{(t)}, \qquad
    H v^{(t)} = \frac{1}{m}\sum_{i=1}^m f''_i(w_k) v^{(t)}.
\]
We note that the overall Hessian~$H=f''(w_k)$ is never formed and the master
machine only stores and updates the vectors $H u^{(t)}$ and $H v^{(t)}$.

As explained in Section~\ref{sec:outline}, 
the motivation for preconditioning is that
when $H_1$ is sufficiently close to $H$, 
the condition number of $\pcd^{-1} H$ might be close to $1$,
which is much smaller than that of~$H$ itself.
As a result, the PCG method may converge much faster than 
CG without preconditioning.
The following lemma characterizes the extreme eigenvalues of $\pcd^{-1}H$ 
based on the closeness between $H_1$ and~$H$.

\begin{lemma}\label{lemma:reduce-condition-number}
    Suppose Assumption~\ref{asmp:Hessian-bounds} holds.
If $\ltwos{H_1 - H}\leq \mu$, then we have
\begin{align}
    \sigma_{\max} (\pcd^{-1} H) &\leq 1, \label{eqn:eigenvalue-ub}\\
    \sigma_{\min} (\pcd^{-1} H) &\geq \frac{\reg}{\reg + 2\mu} \label{eqn:eigenvalue-lb}.
\end{align}
Here $\|\cdot\|_2$ denote the spectral norm of a matrix, and 
$\sigma_\mathrm{max}(\cdot)$ and $\sigma_\mathrm{min}(\cdot)$ 
denote the largest and smallest eigenvalues of a diagonalizable matrix, 
respectively.
\end{lemma}

\begin{proof}
Since both $\pcd$ and $H$ are symmetric and positive definite, all eigenvalues
of $\pcd^{-1}H$ are positive real numbers 
(e.g., \cite[Section~7.6]{HornJohnson85}).
The eigenvalues of $\pcd^{-1} H$ are identical to that of 
$\pcd^{-1/2} H \pcd^{-1/2}$.
Thus, it suffices to prove inequalities~\eqref{eqn:eigenvalue-ub}
and~\eqref{eqn:eigenvalue-lb} for the matrix $\pcd^{-1/2} H \pcd^{-1/2}$.
To prove inequality~\eqref{eqn:eigenvalue-ub}, 
we need to show that $H \preceq \pcd = H_1 +\mu I$. 
This is equivalent to $H - H_1 \preceq \mu I$, which is a direct consequence
of the assumption $\|H_1-H\|_2\leq\mu I$.

Similarly, the second inequality~\eqref{eqn:eigenvalue-lb} 
is equivalent to $H \succeq \frac{\reg}{\reg + 2\mu} (H_1 + \mu I)$, 
which is the same as $\frac{2\mu}{\reg} H - \mu I \succeq H_1 - H$. 
Since $H\succeq \reg I$ (by Assumption~\ref{asmp:Hessian-bounds}),
we have $\frac{2\mu}{\reg} H - \mu I \succeq \mu I $.
The additional assumption $\|H_1-H\|_2\leq\mu I$ implies 
$\mu I \succeq H_1 - H$, which complete the proof.
\end{proof}

By Assumption~\ref{asmp:Hessian-bounds}, the condition number of the Hessian
matrix is $\kappa(H)=L/\lambda$, which can be very large if~$\lambda$ is small. 
Lemma~\ref{lemma:reduce-condition-number} establishes that
the condition number of the preconditioned linear system is
\begin{equation}\label{eqn:preconditioned-kappa}
    \kappa(\pcd^{-1}H) 
    = \frac{\sigma_\mathrm{max}(\pcd^{-1}H)}{\sigma_\mathrm{min}(\pcd^{-1}H)}
    = 1+\frac{2\mu}{\lambda} ,
\end{equation}
provided that $\|H_1-H\|_2\leq \mu$.
When~$\mu$ is small (comparable with~$\lambda$), the condition number 
$\kappa(\pcd^{-1}H)$ is close to one and can be much smaller than $\kappa(H)$.
Based on classical convergence analysis of the CG method
(e.g., \cite{Luenberger73,Avriel76}),
the following lemma shows that Algorithm~\ref{alg:distributed-pcg}
terminates in $\order(\sqrt{1 + \mu/\reg})$ iterations.
See Appendix~\ref{sec:proof-lemma-pcg} for the proof.

\begin{lemma}\label{lemma:pcg-complexity}
Suppose Assumption~\ref{asmp:Hessian-bounds} holds and 
assume that $\ltwos{H_1 - H}\leq \mu$. 
Let
\[
    T_{\mu} = \biggl\lceil \sqrt{1 +\frac{2\mu}{\reg}}\log\biggl(\frac{2 \sqrt{L/\lambda} \ltwos{\fp(w_k)}}{\epsilon_k}\biggr) \biggr\rceil .
\]
Then Algorithm~\ref{alg:distributed-pcg} terminates in $T_{\mu}$ iterations and
the output $v_k$ satisfies $\ltwos{H v_k - \fp(w_k)} \leq \epsilon_k$.
\end{lemma}

When the tolerance $\epsilon_k$ is chosen as in~\eqref{eqn:choose-epsilon-k},
the iteration bound $T_{\mu}$ is independent of $f'(w_k)$, i.e.,
\begin{equation}\label{eqn:choose-T-mu}
    T_{\mu} = \biggl\lceil \sqrt{1 +\frac{2\mu}{\reg}}\log\biggl(\frac{2 L}{\beta \reg} \biggr) \biggr\rceil .
\end{equation}
Under Assumption~\ref{asmp:Hessian-bounds}, we always have
$\ltwos{H_1 - H}\leq \hessianbound$. 
If we choose $\mu = L$, then Lemma~\ref{lemma:pcg-complexity} implies that
Algorithm~\ref{alg:distributed-pcg} terminates
in $\widetilde\order(\sqrt{\hessianbound/\reg})$ iterations.
where the notation $\widetilde\order(\cdot)$ hides logarithmic factors.
In practice, however, the matrix norm $\ltwos{H_1 - H}$ is usually much 
smaller than $\hessianbound$ due to the stochastic nature of $f_i$. 
Thus, we can choose~$\mu$ to be a tight upper bound on $\ltwos{H_1 - H}$,
and expect the algorithm terminating in $\widetilde\order(\sqrt{\mu/\reg})$ 
iterations. 
In Section~\ref{sec:stochastic-analysis}, we show that if the local empirical
losses~$f_i$ are constructed with~$n$ i.i.d.\ samples 
from the same distribution, 
then $\ltwos{H_1 - H} \sim 1/\sqrt{n}$ with high probability. 
As a consequence, 
the iteration complexity of Algorithm~\ref{alg:distributed-pcg} 
is upper bounded by $\widetilde\order(1+\lambda^{-1/2}n^{-1/4})$.

\bigskip

We wrap up this section by discussing the computation and 
communication complexities of Algorithm~\ref{alg:distributed-pcg}.
The bulk of computation is at the master machine, especially computing 
the vector $s^{(t)} = \pcd^{-1} r^{(t)}$ in Step~3, 
which is equivalent to minimize the quadratic
function $(1/2)s^T \pcd s - s^T r^{(t)}$.
Using $\pcd=f''_1(w_k)+\mu I$ 
and the form of $f_1(w)$ in~\eqref{eqn:regularized-ERM}, 
this is equivalent to
\begin{align}\label{eqn:computation-linear-system}
	s^{(t)} = \arg\min_{s\in \R^d} \biggl\{ \frac{1}{n} \sum_{j=1}^n 
    \frac{s^T \phi''(w_k,z_{i,j}) s}{2}  + \langle r^{(t)}, s \rangle 
    +\frac{\lambda + \mu}{2} \ltwos{s}^2 \biggr\} .
\end{align}
This problem has the same structure as~\eqref{eqn:local-init}, 
and an $\epsilon$-optimal solution can be obtained with
$\order\bigl((n+\frac{L+\mu}{\lambda+\mu})\log(1/\epsilon)\bigr)$ 
stochastic-gradient type of steps
(see discussions at the end of Section~\ref{sec:DiSCO-initialization}).

As for the communication complexity, we need one round of communication 
at the beginning of Algorithm~\ref{alg:distributed-pcg} to compute $\fp(w_k)$. 
Then, each iteration takes one round of communication
to compute $H u^\supt$ and $H v^\supt$. 
Thus, the total rounds of communication is bounded by $T_\mu + 1$. 

\begin{algorithm}[t]
\DontPrintSemicolon
\textbf{input:} parameters~$\rho,\mu\geq 0$ and precision~$\epsilon> 0$. \;
\textbf{initialize:} 
compute $w_0$ according to~\eqref{eqn:initialization} 
and~\eqref{eqn:local-solution}.\;
\Repeat( for $k = 0,1,2,\dots$){$\delta_k \leq (1-\beta)\sqrt{\epsilon}$.}{
\vspace{-2ex}
\begin{enumerate} \itemsep 0pt
        \item Run Algorithm~\ref{alg:distributed-pcg}: given $w_k$ and $\mu$, compute $v_k$ and $\delta_k$.
  \item Update $w_{k+1} = w_k - \frac{1}{1 + \delta_k}v_k$.
\end{enumerate}
\vspace{-2ex}
}
\textbf{output:} $\what=w_{k+1}$.
\caption{DiSCO}
\label{alg:DiSCO}
\end{algorithm}

\subsection{Communication efficiency of DiSCO} 

Putting everything together, we present the DiSCO algorithm in
Algorithm~\ref{alg:DiSCO}.
Here we study its communication efficiency.
Recall that by one round of communication, 
the master machine broadcasts a message of $\order(d)$ bits to all machines,
and every machine processes the aggregated message and
sends a message of $\order(d)$ bits back to the master.
The following proposition gives an upper bound on the number of communication 
rounds taken by the DiSCO algorithm.

\begin{theorem}\label{thm:deterministic-bd}
Assume that $f$ is a standard self-concordant function and 
it satisfies Assumption~\ref{asmp:Hessian-bounds}. 
Suppose the input parameter~$\mu$ in Algorithm~\ref{alg:DiSCO} 
is an upper bound on $\ltwos{f''_1(w_k)-f''(w_k)}$ for all $k\geq 0$.
Then for any $\epsilon>0$, in order to find a solution $\what$ 
satisfying $f(\what) - f(\wstar) < \epsilon$, 
the total number of communication rounds~$T$ is bounded by
\begin{align}\label{eqn:disco-communication-bound}
  T \leq 1 + \left( \biggl\lceil \frac{f(w_0) - f(\wstar)}{\frac{1}{2}\omega(1/6)} \biggr\rceil + \left\lceil \log_2 \biggl( \frac{ 2 \omega(1/6)}{\epsilon}\biggr) \right\rceil \right)
  \left(2 + \sqrt{1 +\frac{2\mu}{\reg}}\,\log\biggl(\frac{2 L}{\beta \reg}\biggr) \right).
\end{align}
Ignoring logarithmic terms and universal constants,
the rounds of communication~$T$ is bounded by 
\[	
	\widetilde \order\left(\bigl(f(w_0) - f(\wstar) + \log(1/\epsilon)\bigr)
    \sqrt{1+2\mu/\lambda}\right) .
\]
\end{theorem}

\begin{proof}
First we notice that the number of communication rounds in each call
of Algorithm~\ref{alg:distributed-pcg} is no more than $1+T_\mu$, 
where $T_\mu$ is given in~\eqref{eqn:choose-T-mu}, 
and the extra~1 accounts for the communication round to form $f'(w_k)$.
Corollary~\ref{coro:damped-newton-complexity} states that in order to
guarantee $f(w_k) - f(\wstar) \leq \epsilon$,
the total number of calls of Algorithm~\ref{alg:distributed-pcg} in DiSCO
is bounded by~$K$ given in~\eqref{eqn:damped-newton-complexity}.
Thus the total number of communication rounds is bounded by
$1+K(1+T_\mu)$, where the extra one count is for computing the 
initial point $w_0$ defined in~\eqref{eqn:initialization}.
\end{proof}

\begin{algorithm}[t]
\DontPrintSemicolon
\textbf{input:} parameters~$\rho\geq 0$ and~$\mu_0>0$, 
    and precision $\epsilon>0$. \;
\textbf{initialize:} 
compute $w_0$ according to~\eqref{eqn:initialization} 
and~\eqref{eqn:local-solution}.\;
\Repeat( for $k = 0,1,2,\dots$){$\delta_k \leq (1-\beta)\sqrt{\epsilon}$.}{
\vspace{-2ex}
\begin{enumerate} \itemsep 0pt
        \item Run Algorithm~\ref{alg:distributed-pcg} up to $T_{\mu_k}$ PCG iterations, with output $v_k$, $\delta_k$, $r_k$ and $\epsilon_k$.
        \item \textbf{if} $\ltwos{r_k}>\epsilon_k$ \textbf{then}\\
            \qquad set $\mu_k:=2\mu_k$ and go to Step~1; \\
            \textbf{else}\\
            \qquad set $\mu_{k+1}:=\mu_k/2$ and go to Step~3.
  \item Update $w_{k+1} = w_k - \frac{1}{1 + \delta_k}v_k$.
\end{enumerate}
\vspace{-2ex}
}
\textbf{output:} $\what=w_{k+1}$.
\caption{Adaptive DiSCO}
\label{alg:ada-DiSCO}
\end{algorithm}

It can be hard to give a good \emph{a priori} estimate of~$\mu$ that 
satisfies the condition in Theorem~\ref{thm:deterministic-bd}.
In practice, we can adjust the value of~$\mu$ adaptively while running the algorithm.
Inspired by a line search procedure studied in \cite{Nesterov13composite},
we propose an adaptive DiSCO method, described in 
Algorithm~\ref{alg:ada-DiSCO}. 
The following proposition bounds the rounds of communication required by
this algorithm.

\begin{theorem}\label{thm:adaptive-deterministic-bd}
Assume that $f$ is a standard self-concordant function and 
it satisfies Assumption~\ref{asmp:Hessian-bounds}. 
Let~$\mumax$ be the largest value of $\mu_k$ generated by
Algorithm~\ref{alg:ada-DiSCO},
\ie, $\mumax=\max\{\mu_0,\mu_1,\ldots,\mu_K\}$ 
where $K$ is the number of outer iterations.
Then for any $\epsilon>0$, in order to find a solution~$\what$ 
satisfying $f(\what) - f(\wstar) < \epsilon$, 
the total number of communication rounds $T$ is bounded by
\[
    T \leq 1 + \left( 2\biggl\lceil \frac{f(w_0) - f(\wstar)}{\omega(1/6)} \biggr\rceil
    + 2\left\lceil \log_2 \biggl( \frac{ 2 \omega(1/6)}{\epsilon}\biggr) \right\rceil
    + \log_2\left(\frac{\mumax}{\mu_0}\right) \right)
  \left(2 + \sqrt{1 +\frac{2\mumax}{\reg}}\,
  \log\biggl(\frac{2 L}{\beta \reg}\biggr) \right).
\]
\end{theorem}
\begin{proof}
Let $n_k$ be the number of calls to Algorithm~\ref{alg:distributed-pcg}
during the $k$th iteration of Algorithm~\ref{alg:ada-DiSCO}. We have
\[
    \mu_{k+1} = \frac{1}{2}\mu_k 2^{n_k-1} = \mu_k 2^{n_k-2},
\]
which implies
\[
    n_k = 2 + \log_2\frac{\mu_{k+1}}{\mu_k}.
\]
The total number of calls to Algorithm~\ref{alg:distributed-pcg} is
\[
    N_K = \sum_{k=0}^{K-1} n_k 
    = \sum_{k=0}^{K-1} \left(1+\log_2\frac{\mu_{k+1}}{\mu_k}\right)
    = 2K + \log_2\frac{\mu_K}{\mu_0} 
    \leq 2K + \log_2\frac{\mumax}{\mu_0}  .
\]
Since each call of Algorithm~\ref{alg:distributed-pcg} involves
no more than $T_{\mumax}+1$ communication rounds, we have
\[
 T \leq  1+N_K(T_{\mumax}+1).
\]
Plugging in the expression of~$K$ in~\eqref{eqn:damped-newton-complexity}
and $T_{\mumax}$ in~\eqref{eqn:choose-T-mu}, we obtain the desired result.
\end{proof}

From the above proof, we see that 
the average number of calls to Algorithm~\ref{alg:distributed-pcg}
at each iteration is 
$2+\frac{1}{K}\log_2\left(\frac{\mu_{K}}{\mu_0}\right)$,
roughly twice as the non-adaptive Algorithm~\ref{alg:DiSCO}.
Ignoring logarithmic terms and universal constants,
the number of communication round~$T$ used by Algorithm~\ref{alg:ada-DiSCO}
is bounded by 
\[	
	\widetilde \order\left(\bigl(f(w_0) - f(\wstar) + \log_2(1/\epsilon)
    \bigr) \sqrt{1+2\mumax/\lambda}\right) .
\]
In general, we can update $\mu_k$ in Algorithm~\ref{alg:ada-DiSCO} as follows:
\[
    \mu_k:=\left\{ \begin{array}{ll}
        \theta_\mathrm{inc} \mu_k & \mbox{if}~ \ltwos{r_k}>\epsilon_k, \\
        \mu_k/\theta_\mathrm{dec} & \mbox{if}~ \ltwos{r_k}\leq \epsilon_k,
    \end{array} \right.
\]
with any $\theta_\mathrm{inc}>1$ and $\theta_\mathrm{dec}\geq 1$
(see \cite{Nesterov13composite}). 
We have used $\theta_\mathrm{inc}=\theta_\mathrm{dec}=2$ to simplify
presentation.

\subsection{A simple variant without PCG iterations}
We consider a simple variant of DiSCO where the approximate Newton step $v_k$
is computed without using the PCG method described in 
Algorithm~\ref{alg:distributed-pcg}.
Instead, we simply set
\begin{equation}\label{eqn:simple-variant-no-pcg}
    v_k = \pcd^{-1} f'(w_k) = (f''_1(w_k)+\mu I)^{-1} f'(w_k) ,
\end{equation}
which is equivalent to setting $v_k=s^{(0)}$ in the initialization phase
of Algorithm~\ref{alg:distributed-pcg}, 
or forcing it to always exit during the first PCG iteration.
(The latter choice gives the same search direction but with a slightly 
different scaling.)
In this variant, 
each iteration of the inexact damped Newton method requires two communication
rounds: one to form $f'(w_k)$ and another to compute the stepsize parameter
$\delta_k=(v_k^T f''(w_k) v_k)^{1/2}$.

A distributed algorithm that is similar to this variant of DiSCO is 
proposed in \cite{MahajanKeerthi13}. 
It does not compute $\delta_k$; instead it uses line search to determine
the step size, which also requires extra round(s) of communication.
It is shown in \cite{MahajanKeerthi13} that this method works well 
in experiments, requiring less number of iterations to converge than ADMM.
However, according to their theoretical analysis, 
its iteration complexity still depends badly on the condition number.

Here we examine the theoretical conditions under which this variant of 
DiSCO enjoys a low iteration complexity.
Recall the two auxiliary vectors defined in Section~\ref{sec:approx-Newton}:
\[
    \utilde_k = H^{-1/2}f'(w_k), \qquad \vtilde_k = H^{1/2} v_k .
\]
The norm of their difference can be bounded as
\begin{align*}
    \ltwos{\vtilde_k-\utilde_k} 
    & ~=~ \bigl\| H^{1/2} \pcd^{-1} f'(w_k) - \utilde_k \bigr\|_2 
     ~=~ \bigl\| H^{1/2} \pcd^{-1} H^{1/2} \utilde_k - \utilde_k \bigr\|_2 \\
    & ~\leq~ \bigl\|I-H^{1/2}\pcd^{-1} H^{1/2} \bigr\|_2 \cdot\ltwos{\utilde_k}
    ~=~ \bigl\|I-\pcd^{-1}H\bigr\|_2 \cdot \ltwos{\utilde_k} .
\end{align*}
From Lemma~\ref{lemma:reduce-condition-number}, we know that when 
$\ltwos{H_1-H}\leq \mu$, the eigenvalues
of $\pcd^{-1}H$ are located within the interval 
$[\frac{\lambda}{\lambda+2\mu}, 1]$.
Therefore, we have
\[
    \ltwos{\vtilde_k-\utilde_k} 
    \leq \left(1-\frac{\lambda}{\lambda+2\mu}\right) \ltwos{\utilde_k}
    = \frac{2\mu}{\lambda+2\mu} \ltwos{\utilde_k}.
\]
The above inequality implies
\[
    \left(1-\frac{2\mu}{\lambda+2\mu} \right) \ltwos{\utilde_k} 
    \leq \ltwos{\vtilde_k}
    \leq \left(1+\frac{2\mu}{\lambda+2\mu} \right) \ltwos{\utilde_k}.
\]
This inequality has the same form as~\eqref{eqn:bound-vk-by-uk}, which is
responsible to obtain the desired low complexity result if 
$\frac{2\mu}{\lambda+\mu}$ is sufficiently small.
Indeed, if $\frac{2\mu}{\lambda+2\mu} \leq \beta=\frac{1}{20}$ as specified 
in~\eqref{eqn:choose-epsilon-k}, 
the same convergence rate and complexity result stated
in Theorem~\ref{thm:damped-newton-convergence}
and Corollary~\ref{coro:damped-newton-complexity} apply.
Since each iteration of the damped Newton method involves only two 
communication rounds (to compute $f'(w_k)$ and $\delta_k$ respectively), 
we have the following corollary.

\begin{corollary}\label{coro:simple-variant-complexity}
Assume that $f$ is a standard self-concordant function and 
it satisfies Assumption~\ref{asmp:Hessian-bounds}. 
In the DiSCO algorithm, 
we compute the inexact Newton step using~\eqref{eqn:simple-variant-no-pcg}.
Suppose $\frac{2\mu}{\lambda+2\mu}\leq\frac{1}{20}$ 
and $\ltwos{f''_1(w_k)-f''(w_k)}\leq\mu$ for all $k\geq 0$.
Then for any $\epsilon>0$,
in order to find a solution $\what$ satisfying $f(\what)-f(\wstar)<\epsilon$, 
the total number of communication rounds $T$ is bounded by
\begin{align}\label{eqn:simple-DiSCO-complexity}
  T \leq 1 + 2 \left( \biggl\lceil \frac{f(w_0) - f(\wstar)}{\frac{1}{2}\omega(1/6)} \biggr\rceil + \left\lceil \log_2 \biggl( \frac{ 2 \omega(1/6)}{\epsilon}\biggr) \right\rceil \right) .
\end{align}
\end{corollary}

In Corollary~\ref{coro:simple-variant-complexity}, 
the requirement on $\mu$, which upper bounds $\ltwos{f''_1(w_k)-f''(w_k)}$
for all $k\geq 0$, is quite strong.
In particular, it requires $\mu$ to be a small fraction of $\lambda$
in order to satisfy $\frac{2\mu}{\lambda+2\mu} \leq \frac{1}{20}$.
As we will see from the stochastic analysis in the next section, 
the spectral bound~$\mu$ decreases on the order of $1/\sqrt{n}$. 
Therefore, in the standard setting where the regularization parameter 
$\lambda\sim 1/\sqrt{mn}$, the condition in 
Corollary~\ref{coro:simple-variant-complexity} cannot be satisfied,
and the convergence of this simple variant may be slow.
In contrast, DiSCO with PCG iterations is much more tolerant of
a relatively large~$\mu$, and can achieve superlinear convergence
with a smaller~$\epsilon_k$.

\section{Stochastic analysis}
\label{sec:stochastic-analysis}

From Theorems~\ref{thm:deterministic-bd} and~\ref{thm:adaptive-deterministic-bd}
of the previous section, we see that the communication complexity of the 
DiSCO algorithm mainly depends on two quantities: 
the initial objective gap $f(w_0)-f(\wstar)$ and 
the upper bound~$\mu$ on the spectral norms $\|f''_1(w_k)-f''(w_k)\|_2$
for all $k\geq 0$.
As we discussed in Section~\ref{sec:inexact-newton-scaling}, the initial gap
$f(w_0)-f(\wstar)$ may grow with the number of samples due to the scaling
used to make the objective function standard self-concordant.
On the other hand, the upper bound~$\mu$ may decrease as the number of 
samples increases based on the intuition that the local Hessians and
the global Hessian become similar to each other.
In this section, we show how to exploit the stochastic origin of the problem
(SAA or ERM, as explained in Section~\ref{sec:introduction}) 
to mitigate the effect of objective scaling and 
quantify the notion of similarity between local and global Hessians.
These lead to improved complexity results.

We focus on the setting of distributed optimization of \emph{regularized}
empirical loss. That is, our goal is to minimize 
$f(w)=(1/m)\sum_{i=1}^m f_i(w)$, where
\begin{equation}\label{eqn:regularized-loss}
    f_i(w) = \frac{1}{n}\sum_{j=1}^n \phi(w, z_{i,j}) + \frac{\lambda}{2}\ltwos{w}^2, \qquad i=1,\ldots,m.
\end{equation}
We assume that $z_{i,j}$ are i.i.d.\ samples from a common distribution.
Our theoretical analysis relies on refined assumptions on the smoothness of 
the loss function~$\phi$. 
In particular, we assume that for any~$z$ in the sampling space~$\mathcal{Z}$,
the function $\phi(\cdot,z)$ has bounded first derivative in a compact set, 
and its second derivatives are bounded and Lipschitz continuous.
We formalize these statements in the following assumption.

\begin{assumption}\label{asmp:smoothness}
There are finite constants $(V_0,\gradbound,\hessianbound,\tensorbound)$, 
such that for any $z\in\mathcal{Z}$:
\begin{enumerate}[(i)]
    \item $\phi(w,z)\geq 0$ for all $w\in\R^d$, and $\phi(0,z) \leq V_0$;
    \item $\ltwos{\phi'(w,z)} \leq \gradbound$ for any $\ltwos{w} \leq \sqrt{2V_0/\reg}$;
    \item $\ltwos{\phi''(w,z)} \leq \hessianbound - \reg$ for any $w\in \R^d$;
    \item $\ltwos{\phi''(u,z) - \phi''(w,z)} \leq \tensorbound \ltwos{u-w}$ for any $u,w\in \R^d$.
\end{enumerate}
\end{assumption}

For the regularized empirical loss in~\eqref{eqn:regularized-loss},
condition~$(iii)$ in the above assumption implies 
$\reg I \preceq \fpp_i(w) \preceq \hessianbound I$ for $i=1,\ldots,m$,
which in turn implies Assumption~\ref{asmp:Hessian-bounds}.

Recall that the initial point $w_0$ is obtained as the average of the solutions
to~$m$ regularized local optimization problems;
see equations~\eqref{eqn:initialization} and~\eqref{eqn:local-solution}.
The following lemma shows that expected value of the initial gap 
$f(w_0)-f(\wstar)$ decreases with
order $1/\sqrt{n}$ as the local sample size~$n$ increases.
The proof uses the notion and techniques of \emph{uniform stability} 
for analyzing the generalization performance of 
ERM~\cite{BousquetElisseeff02}. 
See Appendix~\ref{sec:initialization-accuracy-proof} for the proof.

\begin{lemma}\label{lemma:initialization-accuracy}
Suppose that Assumption~\ref{asmp:smoothness} holds and 
$\E[\ltwos{\wstar}^2] \leq D^2$ for some constant~$D>0$.
If we choose $\rho=\frac{\sqrt{6}G}{\sqrt{n}D}$ in~\eqref{eqn:local-solution}
to compute $\what_i$, then the initial point 
$w_0=\frac{1}{m}\sum_{i=1}^m \what_i$ satisfies
\begin{equation}\label{eqn:w0-norm-bound}
	\max\{\ltwos{\wstar},\ltwos{w_0}\}\leq \sqrt{\frac{2 V_0}{\reg}}
\end{equation}
and
\begin{equation}\label{eqn:initial-obj-bound}
    \E[f(w_0) - f(\wstar)] \leq \frac{\sqrt{6}\gradbound D}{\sqrt{n}}.
\end{equation}
Here the expectation is taken with respect to the randomness 
in generating the i.i.d.\ data.
\end{lemma}

Next, we show that with high probability, 
$\ltwos{\fpp_i(w) - \fpp(w)} \sim \sqrt{d/n}$ 
for any $i\in\{1,\ldots,m\}$ and for any vector~$w$ in an $\ell_2$-ball.
Thus, if the number of samples~$n$ is large, 
the Hessian matrix of~$f$ can be approximated well by that of $f_i$. 
The proof uses random matrix
concentration theories~\cite{mackey2014matrix}. We defer the proof to Appendix~\ref{sec:proof-uniform-matrix-concentration}.

\begin{lemma}\label{lemma:uniform-matrix-concentration}
    Suppose Assumption~\ref{asmp:smoothness} holds.
    For any $r > 0$ and any $i\in \{1,\ldots,m\}$, 
    we have with probability at least $1-\delta$,
\begin{align*}
    \sup_{\ltwos{w}\leq r} \ltwos{\fpp_i(w) - \fpp(w)} \leq \mu_{r,\delta},
\end{align*}
where
\begin{equation}\label{eqn:choose-mu}
    \mu_{r,\delta} \eqdef
    \min\left\{ \hessianbound, \sqrt{\frac{32 \hessianbound^2 d}{n}}
	\cdot \sqrt{ \log \Big(1+ \frac{r \tensorbound \sqrt{2n}}{\hessianbound} \Big) + \frac{\log(m d/\delta)}{d}}\right\}.
\end{equation}
\end{lemma}

If $\phi(w,z_{i,j})$ are quadratic functions in~$w$, then we have $M=0$
in Assumption~\ref{asmp:smoothness}. 
In this case, Lemma~\ref{lemma:uniform-matrix-concentration} implies
$\ltwos{\fpp_i(w) - \fpp(w)} \sim \sqrt{1/n}$.
For general non-quadratic loss, Lemma~\ref{lemma:uniform-matrix-concentration} 
implies $\ltwos{\fpp_i(w) - \fpp(w)} \sim \sqrt{d/n}$. 
We use this lemma to obtain an upper bound on the spectral norm of the
Hessian distances $\ltwos{\fpp_1(w_k) - \fpp(w_k)}$, where the vectors $w_k$
are generated by Algorithm~\ref{alg:damped-newton-method}.

\begin{corollary}\label{coro:high-prob-bd-mu}
Suppose Assumption~\ref{asmp:smoothness} holds and the sequence 
$\{w_k\}_{k\geq 0}$ is generated by Algorithm~\ref{alg:damped-newton-method}.
Let $r=\Big( \frac{2V_0}{\reg} + \frac{2 \gradbound}{\reg} \sqrt{\frac{2V_0}{\reg }} \Big)^{1/2}$.
Then with probability at least $1-\delta$, we have for all $k\geq 0$,
\begin{align}\label{eqn:hessian-approx-bd}
	\ltwos{\fpp_1(w_k) - \fpp(w_k)}
    \leq \min\left\{ \hessianbound, \sqrt{\frac{32 \hessianbound^2 d}{n}}
	\cdot \sqrt{ \log \Big(1+ \frac{r \tensorbound \sqrt{2n}}{\hessianbound} \Big) + \frac{\log(m d/\delta)}{d}}\right\} .
\end{align}
\end{corollary}

\begin{proof}
We begin by upper bounding the $\ell_2$-norm of $w_k$, for $k=0,1,2\ldots$, 
generated by Algorithm~\ref{alg:damped-newton-method}.
By Theorem~\ref{thm:damped-newton-convergence}, we have
$f(w_k)\leq f(w_0)$ for all $k\geq 0$. 
By Assumption~\ref{asmp:smoothness}~(i), we have $\phi(w,z)\geq 0$ for
all $w\in\R^d$ and $z\in\mathcal{Z}$.
As a consequence,
\begin{align*}
	\frac{\reg}{2}\ltwos{w_k}^2 \leq f(w_k) \leq f(w_0) \leq f(0) + \gradbound \ltwos{w_0} \leq V_0 + \gradbound \ltwos{w_0}.
\end{align*}
Substituting $\ltwos{w_0} \leq \sqrt{2V_0/\reg}$ 
(see Lemma~\ref{lemma:initialization-accuracy})
into the above inequality yields
\begin{align*}
	\ltwos{w_k} \leq \left( \frac{2V_0}{\reg} + \frac{2 \gradbound}{\reg} \sqrt{\frac{2V_0}{\reg}} \right)^{1/2} = r.
\end{align*}
Thus, we have $\ltwos{w_k}\leq r$ for all $k\geq 0$.
Applying Lemma~\ref{lemma:uniform-matrix-concentration} 
establishes the corollary.
\end{proof}

Here we remark that the dependence on~$d$ of the upper bound 
in~\eqref{eqn:hessian-approx-bd} comes from 
Lemma~\ref{lemma:uniform-matrix-concentration},
where the bound needs to hold for all point in a $d$-dimensional ball
with radius~$r$.
However, for the analysis of the DiSCO algorithm, 
we only need the matrix concentration bound to hold for a finite number of 
vectors $w_0,w_1,\dots,w_K$,
instead of for all vectors satisfying $\ltwos{w} \leq r$. 
Thus we conjecture that the bound in~\eqref{eqn:hessian-approx-bd},
especially its dependence on the dimension~$d$, is too conservative
and very likely can be tightened.

We are now ready to present the main results of our stochastic analysis.
The following theorem provides an upper bound on the expected number of
communication rounds required by the DiSCO algorithm to find an 
$\epsilon$-optimal solution. 
Here the expectation is taken with respect to the randomness in
generating the i.i.d.\ data set $\{z_{i,j}\}$.

\begin{theorem}\label{thm:DiSCO-stochastic}
Let Assumption~\ref{asmp:smoothness} hold. 
Assume that the regularized empirical loss function~$f$ is standard 
self-concordant, and its minimizer $\wstar=\arg\min_w f(w)$ satisfies
$\E[\ltwos{\wstar}^2]\leq D^2$ for some constant~$D>0$.
Let the input parameters to Algorithm~\ref{alg:DiSCO} be
$\rho=\frac{\sqrt{6}G}{\sqrt{n}D}$ and $\mu=\mu_{r,\delta}$
in~\eqref{eqn:choose-mu} with
\begin{equation}\label{eqn:choose-delta}
r=\biggl( \frac{2V_0}{\reg} + \frac{2 \gradbound}{\reg} \sqrt{\frac{2V_0}{\reg }} \biggr)^{1/2}, \qquad
\delta = \frac{GD}{\sqrt{n}}\cdot\frac{\sqrt{\lambda/(4L)}}{4V_0+2G^2/\lambda}.
\end{equation}
Then for any $\epsilon>0$, the total number of communication rounds~$T$ 
required to reach $f(\what)-f(\wstar)\leq\epsilon$ is bounded by
\begin{align*}
  \E[T] \leq 1 + \left( C_1 + \frac{6}{\omega(1/6)} \cdot \frac{\gradbound D}{\sqrt{n}}\right)
  \Bigg(2 + C_2 \Bigg( 1 + 2 \sqrt{ \frac{32 \hessianbound^2 d \;C_3}{\reg^2 n} } \Bigg)^{1/2} \Bigg) ,
\end{align*}
where $C_1, C_2, C_3$ are $\widetilde\order(1)$ or logarithmic terms:
\begin{align*}
    C_1 &= \left(1+\left\lceil\log_2 \bigg( \frac{ 2 \omega(1/6) }{\epsilon}\bigg)\right\rceil \right)
           \left(1+\frac{1}{\sqrt{n}}\cdot\frac{GD}{4V_0+2G^2/\lambda}\right),\\
    C_2 &= \log\bigg(\frac{2 \hessianbound}{\beta \reg}\bigg), \\
    C_3 &= \log \bigg(1+ \frac{r \tensorbound \sqrt{2n}}{\hessianbound} \bigg) + \frac{\log(d m/\delta)}{d}.
\end{align*}
In particular, ignoring numerical constants and logarithmic terms, we have
\[
    \E[T] = \widetilde\order\left( \biggl(\log(1/\epsilon) + \frac{\gradbound D }{n^{1/2}} \biggr)\biggl(1 + \frac{\hessianbound^{1/2} d^{1/4}}{\reg^{1/2} n^{1/4}} \biggr)\right).
\]
\end{theorem}

\begin{proof}
Suppose Algorithm~\ref{alg:DiSCO} terminates in~$K$ iterations,
and let $t_k$ be the number of conjugate gradient steps in each call
of Algorithm~\ref{alg:distributed-pcg}, for $k=0,1,\ldots,K-1$.
For any given $\mu>0$, we define $T_\mu$ as in~\eqref{eqn:choose-T-mu}.
Let $\fastpcg$ denotes the event that $t_k\leq T_\mu$ for all 
$k\in\{0,\ldots,K-1\}$. 
Let $\slowpcg$ be the complement of~$\fastpcg$, \ie, 
the event that $t_k> T_\mu$ for some $k\in\{0,\ldots,K-1\}$.
In addition, let the probabilities of the events~$\fastpcg$ and~$\slowpcg$ 
be $1-\delta$ and~$\delta$ respectively.
By the law of total expectation, we have
\begin{align*}
    \E[T] = \E[T|\fastpcg] \Prob(\fastpcg) + \E[T|\slowpcg] \Prob(\slowpcg)
    = (1-\delta)\E[T|\fastpcg] + \delta\,\E[T|\slowpcg] .
\end{align*}
When the event~$\mathcal{A}$ happens, we have $T \leq 1 + K(T_{\mu} +1 )$
where $T_\mu$ is given in~\eqref{eqn:choose-T-mu};
otherwise we have $T \leq 1 + K(T_L+1)$, where 
\begin{equation}\label{eqn:T-L}
    T_L = \sqrt{2+\frac{2L}{\lambda}}\log\left(\frac{2L}{\beta\lambda}\right)
\end{equation}
is an upper bound on the number of PCG iterations 
in Algorithm~\ref{alg:distributed-pcg} 
when the event~$\bar{\mathcal{A}}$ happens
(see the analysis in Appendix~\ref{sec:analysis-T-L}).
Since Algorithm~\ref{alg:distributed-pcg} always return a~$v_k$
such that $\|f''(w_k)v_k -f'(w_k)\|_2\leq\epsilon_k$,
the outer iteration count~$K$ share the same 
bound~\eqref{eqn:damped-newton-complexity},
which depends on the random variable $f(w_0)-f(\wstar)$.
However, $T_\mu$ and $T_L$ are deterministic constants.
So we have
\begin{align}
    \E[T]    &\leq 1 + (1-\delta) \E[K(T_{\mu}+1)|\fastpcg] 
        + \delta\,\E[K(T_{L}+1)|\slowpcg]  \nonumber \\
    &= 1 + (1-\delta) (T_{\mu}+1)\E[K|\fastpcg] 
        + \delta(T_{L}+1)\E[K|\slowpcg] .
        \label{eqn:expected-T}
\end{align}

Next we bound $\E[K|\fastpcg]$ and $\E[K|\slowpcg]$ separately.
To bound $\E[K|\fastpcg]$, we use
\[
    \E[K] = (1-\delta) \E[K|\fastpcg] + \delta\,\E[K|\slowpcg]
    \geq (1-\delta) \E[K|\fastpcg] 
\]
to obtain 
\begin{equation}\label{eqn:expect-K-given-A}
\E[K|\fastpcg]\leq\E[K]/(1-\delta).
\end{equation}
In order to bound $\E[K|\slowpcg]$, we derive a deterministic bound on 
$f(w_0)-f(\wstar)$.
By Lemma~\ref{lemma:initialization-accuracy}, we have
$\ltwos{w_0}\leq\sqrt{2V_0/\lambda}$,
which together with Assumption~\ref{asmp:smoothness}~(ii) yields
\[
    \ltwos{f'(w)} \leq G + \lambda\|w\|_2 \leq G + \sqrt{2\lambda V_0}.
\]
Combining with the strong convexity of $f$, we obtain
\begin{align*}
    f(w_0)-f(\wstar) \leq \frac{1}{2\lambda}\ltwos{f'(w_0)}^2
    \leq\frac{1}{2\lambda}\left(G+\sqrt{2\lambda V_0}\right)^2
    \leq 2V + \frac{G^2}{\lambda}.
\end{align*}
Therefore by Corollary~\ref{coro:damped-newton-complexity}, 
\begin{equation}\label{eqn:absolute-K-bd}
    K \leq K_\mathrm{max} \eqdef 1 + \frac{4V_0+2G^2/\lambda}{\omega(1/6)} 
    + \left\lceil \log_2\left(\frac{2\omega(1/6)}{\epsilon}\right)\right\rceil ,
\end{equation}
where the additional~1 counts compensate for removing one
$\lceil\cdot\rceil$ operator in~\eqref{eqn:damped-newton-complexity}.

Using inequality~\eqref{eqn:expected-T},
the bound on $\E[K|\fastpcg]$ in~\eqref{eqn:expect-K-given-A} and 
the bound on $\E[K|\slowpcg]$ in~\eqref{eqn:absolute-K-bd},
we obtain
\[
    \E[T] \leq 1 + (T_{\mu}+1) \E[K] + \delta (T_{L}+1) K_\mathrm{max} .
\]
Now we can bound $\E[K]$ by Corollary~\ref{coro:damped-newton-complexity} 
and Lemma~\ref{lemma:initialization-accuracy}. More specifically, 
\begin{align}\label{eqn:expectation-of-outer-loop}
  \E[ K ] \leq \frac{\E[f(w_0) - f(\wstar)]}{\frac{1}{2}\omega(1/6)} + 
  \left\lceil \log_2 \Big( \frac{ 2 \omega(1/6)}{\epsilon}\Big) \right\rceil + 1
  = C_0 + \frac{2\sqrt{6}}{\omega(1/6)} \cdot \frac{\gradbound D}{\sqrt{n}},
\end{align}
where $C_0=1+\left\lceil\log_2(2\omega(1/6)/\epsilon) \right\rceil$.
With the choice of~$\delta$ in~\eqref{eqn:choose-delta} and the definition 
of $T_L$ in~\eqref{eqn:T-L}, we have
\begin{align*}
    \delta(T_L+1)K_\mathrm{max}
&= \frac{GD}{\sqrt{n}}\cdot\frac{\sqrt{\lambda/(4L)}}{4V_0 + 2G^2/\lambda}
    \left(2+\sqrt{2+\frac{2L}{\lambda}}\log\left(\frac{2L}{\beta\lambda}\right)\right)
    \left(C_0+\frac{4V_0+2G^2/\lambda}{\omega(1/6)}\right) \\
    &= \left(\frac{C_0}{\sqrt{n}}\cdot\frac{GD}{4V_0+2G^2/\lambda} 
       + \frac{1}{\omega(1/6)}\cdot\frac{GD}{\sqrt{n}}\right)
       \left(\sqrt{\frac{\lambda}{L}}+C_2\sqrt{\frac{\lambda}{2L}+\frac{1}{2}}\right) \\
    &\leq \left(\frac{C_0}{\sqrt{n}}\cdot\frac{GD}{4V_0+2G^2/\lambda} 
       + \frac{1}{\omega(1/6)}\cdot\frac{GD}{\sqrt{n}}\right)
       \left(2+C_2\sqrt{1+\frac{2\mu}{\lambda}}\right) \\
    &= \left(\frac{C_0}{\sqrt{n}}\cdot\frac{GD}{4V_0+2G^2/\lambda} 
       + \frac{1}{\omega(1/6)}\cdot\frac{GD}{\sqrt{n}}\right)
       \left(T_{\mu}+1\right)
\end{align*}
Putting everything together, we have
\begin{align*}
    \E[T] &\leq 1 + \left(C_0+\frac{C_0}{\sqrt{n}}\cdot\frac{GD}{4V_0+2G^2/\lambda}+\frac{2\sqrt{6}+1}{\omega(1/6)}\cdot \frac{GD}{\sqrt{n}} \right) (T_\mu+1)\\
    &\leq 1 + \left(C_1 + \frac{6}{\omega(1/6)}\cdot \frac{GD}{\sqrt{n}} \right) (T_\mu+1) .
\end{align*}
Replacing $T_{\mu}$ by its expression in~\eqref{eqn:choose-T-mu}
and applying Corollary~\ref{coro:high-prob-bd-mu}, we obtain the desired result.
\end{proof}

According to Theorem~\ref{thm:DiSCO-stochastic}, we need to set the two 
input parameters~$\rho$ and~$\mu$ in Algorithm~\ref{alg:DiSCO} appropriately
to obtain the desired communication efficiency.
Using the adaptive DiSCO method given in Algorithm~\ref{alg:ada-DiSCO},
we can avoid the explicit specification of $\mu=\mu_{r,\delta}$
defined in~\eqref{eqn:choose-mu} and~\eqref{eqn:choose-delta}.
This is formalized in the following theorem.

\begin{theorem}\label{thm:ada-DiSCO-stochastic}
Let Assumption~\ref{asmp:smoothness} hold. 
Assume that the regularized empirical loss function~$f$ is standard 
self-concordant, and its minimizer $\wstar=\arg\min_w f(w)$ satisfies
$\E[\ltwos{\wstar}^2]\leq D^2$ for some constant~$D>0$.
Let the input parameters to Algorithm~\ref{alg:ada-DiSCO} be
$\rho=\frac{\sqrt{6}G}{\sqrt{n}D}$ and any $\mu_0>0$.
Then the total number of communication rounds~$T$ 
required to reach $f(\what)-f(\wstar)\leq\epsilon$ is bounded by
\[
    \E[T] = \widetilde\order\left( \biggl(\log(1/\epsilon) + \frac{\gradbound D }{n^{1/2}} \biggr)\biggl(1 + \frac{\hessianbound^{1/2} d^{1/4}}{\reg^{1/2} n^{1/4}} \biggr)\right).
\]
\end{theorem}
\begin{proof}
In Algorithm~\ref{alg:ada-DiSCO}, the parameter~$\mu_k$ is automatically tuned
such that the number of PCG iterations in Algorithm~\ref{alg:distributed-pcg}
is no more than $T_{\mu_k}$. 
By Corollary~\ref{coro:high-prob-bd-mu}, 
with probability at least $1-\delta$, we have
\[
    \max\{\mu_0, \ldots, \mu_K\} \leq 2 \mu_{r,\delta} 
\]
where $\mu_{r,\delta}$ is defined in~\eqref{eqn:choose-mu},
and~$r$ and~$\delta$ are given in~\eqref{eqn:choose-delta}. 
Therefore we can use the same arguments in the proof of 
Theorem~\ref{thm:DiSCO-stochastic} to show that 
\begin{align*}
  \E[T] \leq 1 + \left( \tilde C_1 + \frac{6}{\omega(1/6)} \cdot \frac{\gradbound D}{\sqrt{n}}\right)
  \Bigg(2 + C_2 \Bigg( 1 + 4 \sqrt{ \frac{32 \hessianbound^2 d \;C_3}{\reg^2 n} } \Bigg)^{1/2} \Bigg)
\end{align*}
where 
\begin{align*}
    \tilde C_1 &= \left(2+2\left\lceil\log_2 \bigg( \frac{ 2 \omega(1/6) }{\epsilon}\bigg) \right\rceil
    + \log_2\left(\frac{L}{\mu_0}\right) \right)
           \left(1+\frac{1}{\sqrt{n}}\cdot\frac{GD}{4V_0+2G^2/\lambda}\right),
\end{align*}
and $C_2$ and $C_3$ are the same as given in Theorem~\ref{thm:DiSCO-stochastic}.
Ignoring constants and logarithmic terms, we obtain the desired result.
\end{proof}

In both Theorems~\ref{thm:DiSCO-stochastic} and~\ref{thm:ada-DiSCO-stochastic},
the parameter $\rho=\frac{\sqrt{6}G}{\sqrt{n}D}$ 
depends on a constant~$D$ such that $\E[\ltwos{\wstar}^2]\leq D^2$.
In practice, it may be hard to give a tight estimate of $\E[\ltwos{\wstar}^2]$.
An alternative is to fix a desired value of~$D$ and consider the 
constrained optimization problem
\[
    \minimize_{\ltwos{w} \leq D}~ f(w).
\]
To handle the constraint $\ltwos{w}\leq D$, 
we need to replace the inexact damped Newton method
in DiSCO with an inexact \emph{proximal} Newton method, and replace the
distributed PCG method for solving the Newton system with a 
preconditioned accelerated proximal gradient method. 
Further details of such an extension are given in 
Section~\ref{sec:composite-minimization}. 

\paragraph{Remarks}
The expectation bounds on the rounds of communication given 
in Theorems~\ref{thm:DiSCO-stochastic} and~\ref{thm:ada-DiSCO-stochastic} 
are obtained by combining two consequences of averaging over a large number 
of i.i.d.\ local samples. 
One is the expected reduction of the initial gap $f(w_0)-f(\wstar)$
(Lemma~\ref{lemma:initialization-accuracy}), 
which helps to mitigate the effect of objective scaling used to make~$f$
standard self-concordant.
The other is a high-probability bound that characterizes the similarity
between the local and global Hessians (Corollary~\ref{coro:high-prob-bd-mu}).
If the empirical loss~$f$ is standard self-concordant without scaling, 
then we can regard $f(w_0)-f(\wstar)$ as a constant, and only need to
use Corollary~\ref{coro:high-prob-bd-mu} to obtain a high-probability bound.
This is demonstrated for the case of linear regression in 
Section~\ref{sec:linear-regression}.

For applications where the loss function needs to be rescaled to be standard
self-concordant, the convexity parameter $\lambda$ as well as the ``constants''
$(V_0, G, L, M)$ in Assumption~\ref{asmp:smoothness} also need to be rescaled.
If the scaling factor grows with~$n$, then we need to rely on
Lemma~\ref{lemma:initialization-accuracy} to balance the effects of scaling.
As a result, we only obtain bounds on the expected number of 
communication rounds. 
These are demonstrated in Section~\ref{sec:classification}
for binary classification with logistic regression and a smoothed hinge loss.

\subsection{Application to linear regression}
\label{sec:linear-regression}

We consider linear regression with quadratic regularization (ridge regression).
More specifically, we minimize the overall empirical loss function
\begin{align}\label{eqn:quadratic-erm}
    f(w) = \frac{1}{mn} \sum_{i=1}^m \sum_{j=1}^n (y_{i,j} - w^T x_{i,j})^2 + \frac{\lambda}{2}\ltwos{w}^2,
\end{align}
where the i.i.d.~instances $(x_{i,j},y_{i,j})$ are sampled from 
$\mathcal{X}\times\mathcal{Y}$.
We assume that $\mathcal{X}\subset\R^d$ and $\mathcal{Y}\subset\R$
are bounded: there exist constants $B_x$ and $B_y$ such that
$\ltwos{x}\leq B_x$ and $|y|\leq B_y$ for any 
$(x,y)\in \mathcal{X}\times\mathcal{Y}$. 
It can be shown that the least-squares loss $\phi(w,(x,y)) = (y-w^T x)^2$ 
satisfies Assumption~\ref{asmp:smoothness} with
\[
    V_0 = B_y^2, \qquad G=2B_x\left(B_y+B_x B_y\sqrt{2/\lambda}\right) , 
   \qquad L=\lambda + 2 B_x^2, \qquad M=0.
\]
Thus we can apply 
Theorems~\ref{thm:DiSCO-stochastic} and~\ref{thm:ada-DiSCO-stochastic}
to obtain an expectation bound on the number of communication rounds for DiSCO.
For linear regression, however, we can obtain a stronger result. 

Since~$f$ is a quadratic function, it is self-concordant with parameter~$0$, 
and by definition also standard self-concordant (with parameter $2$).
In this case, we do not need to rescale the objective function, 
and can regard the initial gap $f(w_0)-f(\wstar)$ as a constant.
As a consequence, we can directly apply Theorem~\ref{thm:deterministic-bd}
and Corollary~\ref{coro:high-prob-bd-mu} to obtain a high probability bound
on the communication complexity, 
which is stronger than the expectation bounds in
Theorems~\ref{thm:DiSCO-stochastic} and~\ref{thm:ada-DiSCO-stochastic}.
In particular, Theorem~\ref{thm:deterministic-bd} states that if
\begin{equation}\label{eqn:hessian-mu-bd}
    \bigl\|f''_1(w_k)-f''(w_k)\bigr\|_2\leq\mu, 
    \quad\mbox{for all}\quad k=0,1,2,\ldots,
\end{equation}
then the number of communication rounds~$T$ is bounded as
\[
    T \leq 1 + \left(\biggl\lceil \frac{f(w_0)-f(\wstar)}{\omega(1/6)} \biggr\rceil
    + \left\lceil \log_2\left(\frac{2\omega(1/6)}{\epsilon}\right) \right\rceil \right)
    \left(2 + \sqrt{1+\frac{2\mu}{\lambda}}
    \log\left(\frac{2L}{\beta\lambda}\right) \right) .
\]
Since there is no scaling, the initial gap $f(w_0)-f(\wstar)$ can be considered
as a constant. For example, we can simply pick $w_0=0$ and have
\[
    f(0)-f(\wstar) \leq  f(0) = \frac{1}{N} \sum_{i=1}^N y_i^2 \leq B_y^2 .
\]
By Corollary~\ref{coro:high-prob-bd-mu} and the fact that $M=0$ for quadratic
functions, the condition~\eqref{eqn:hessian-mu-bd} holds with probability 
at least $1-\delta$ if we choose
\begin{equation}\label{eqn:quadratic-mu}
    \mu = \sqrt{\frac{32 L^2 d}{n}}\sqrt{\frac{\log(md/\delta)}{d}}
    = \frac{8 L}{\sqrt{n}} \sqrt{2\log(md/\delta)} .
\end{equation}
Further using $L\leq \lambda + 2 B_x^2$, we obtain the following corollary.

\begin{corollary}\label{coro:quadratic-loss-complexity}
Suppose we apply DiSCO (Algorithm~\ref{alg:DiSCO}) to minimize~$f(w)$
defined in~\eqref{eqn:quadratic-erm} with the input parameter~$\mu$
in~\eqref{eqn:quadratic-mu}, and let~$T$ be the total number of communication 
rounds required to find an $\epsilon$-optimal solution.
With probability at least $1-\delta$, we have
\begin{equation}\label{eqn:quadratic-comm-bd}
T =  \widetilde \order\Big( \Big( 1 + \frac{B_x  }{\lambda^{1/2} n^{1/4}} \Big)\log(1/\epsilon) \log(md/\delta) \Big).
\end{equation}
\end{corollary}

We note that the same conclusion also holds for the adaptive DiSCO algorithm
(Algorithm~\ref{alg:ada-DiSCO}), where we do not need to specify the input
parameter~$\mu$ based on~\eqref{eqn:quadratic-mu}.
For the adaptive DiSCO algorithm, the bound 
in~\eqref{eqn:quadratic-comm-bd} holds for any $\delta\in(0,1)$.

The communication complexity guaranteed by 
Corollary~\ref{coro:quadratic-loss-complexity} is strictly better than that 
of distributed implementation of the accelerated gradient method and ADMM
(\cf~Table~\ref{tab:complexities}).
If we choose $\lambda = \Theta(1/\sqrt{mn})$,
then Corollary~\ref{coro:quadratic-loss-complexity} implies 
\[
    T = \widetilde \order\left(m^{1/4} \log(1/\epsilon)\right)
\]
with high probability.
The DANE algorithm~\cite{ShamirSrebroZhang14DANE}, under the same setting, 
converges in $\widetilde \order(m \log(1/\epsilon))$ iterations with high 
probability (and each iteration requires two rounds of communication).
Thus DiSCO enjoys a better communication efficiency.

\subsection{Application to binary classification}
\label{sec:classification}

For binary classification, we consider the following regularized
empirical loss function
\begin{align}\label{eqn:classification-erm}
    \ell(w) \eqdef \frac{1}{mn} \sum_{i=1}^m\sum_{j=1}^n 
    \varphi(y_{i,j} w^T x_{i,j}) + \frac{\gamma}{2}\ltwos{w}^2,
\end{align}
where $x_{i,j}\in \mathcal{X}\subset\R^d$, $y_{i,j}\in\{-1,1\}$, and 
$\varphi: \R\to\R$ is a convex surrogate function for the binary loss.
We further assume that the elements of $\mathcal{X}$ are bounded, \ie,
we have $\sup_{x\in \mathcal{X}} \ltwos{x} \leq B$ for some finite $B$. 

Under the above assumptions, Lemma~\ref{lemma:regularized-self-concordance}
gives conditions on~$\varphi$ for~$\ell$ to be self-concordant.
As we have seen in Section~\ref{sec:self-concordance}, the function~$\ell$
usually needs to be scaled by a large factor to become standard self-concordant.
Let the scaling factor be~$\eta$, we can use DiSCO to minimize the scaled 
function $f(w)=\eta\ell(w)$. 
Next we discuss the theoretical implications for logistic regression and 
the smoothed hinge loss constructed in Section~\ref{sec:self-concordance}.
These results are summarized in Table~\ref{tab:complexities}.

\paragraph{Logistic Regression}
For logistic regression, we have $\varphi(t) = \log(1+e^{-t})$. 
In Section~\ref{sec:self-concordance}, we have shown that the logistic loss
satisfies the condition of Lemma~\ref{lemma:regularized-self-concordance}
with $Q=1$ and $\alpha = 0$.
Consequently, with the factor $\eta = \frac{B^2}{4\gamma}$, 
the rescaled function $f(w) = \eta \ell(w)$ is standard self-concordant. 
If we express $f$ in the standard form
\begin{equation}\label{eqn:standard-f}
    f(w) = \frac{1}{mn} \sum_{i=1}^m\sum_{j=1}^n 
    \phi(y_{i,j} w^T x_{i,j}) + \frac{\lambda}{2}\ltwos{w}^2,
\end{equation}
then we have $\phi(w, (x, y)) = \eta \varphi(y w^T x)$ and 
$\lambda = \eta \gamma$.
It is easy to check that Assumption~\ref{asmp:smoothness} holds with
\begin{align*}
	V_0 = \eta \log(2), \qquad 
    G = \eta B, \qquad 
    L = \eta (B^2/4 + \gamma), \qquad
    M = \eta B^3/10,
\end{align*}
which all containing the scaling factor~$\eta$.
Plugging these scaled constants into Theorems~\ref{thm:DiSCO-stochastic}
and~\ref{thm:ada-DiSCO-stochastic}, we have the following corollary.

\begin{corollary}\label{coro:logistic-regression-complexity}
For logistic regression, the number of communication rounds required by DiSCO 
to find an $\epsilon$-optimal solution is bounded by
\begin{align*}
	\E[T] =  \widetilde \order\Big( \Big(\log(1/\epsilon)+\frac{B^3 D}{\gamma n^{1/2}}  \Big)\Big(1+\frac{B d^{1/4}}{\gamma^{1/2} n^{1/4}} \Big)\Big).
\end{align*}
\end{corollary}

\noindent
In the specific case when $\gamma = \Theta(1/\sqrt{mn})$, Corollary~\ref{coro:logistic-regression-complexity}
implies
\[
	\E[T] = \widetilde \order\left(m^{3/4} d^{1/4} + m^{1/4} d^{1/4} \log(1/\epsilon)\right).
\]
If we ignore logarithmic terms, then the expected number of communication 
rounds is independent of the sample size~$n$, and only grows slowly with
the number of machines~$m$.

\paragraph{Smoothed Hinge Loss}
We consider minimizing~$\ell(w)$ in~\eqref{eqn:classification-erm} 
where the loss function~$\varphi$ is the smoothed hinge loss defined 
in~\eqref{eqn:smoothed-hinge}, which depends on a parameter $p\geq 3$.
Using Lemma~\ref{lemma:regularized-self-concordance}, we have shown in
Section~\ref{sec:self-concordance} that $\ell(w)$ is self-concordant
with parameter~$M_p$ given in~\eqref{eqn:smoothed-sc-p}.
As a consequence, by choosing
\[
    \eta = \frac{M_p^2}{4} 
    = \frac{(p-2)^2 B^{2 + \frac{4}{p-2}}}{4\gamma^{1 + \frac{2}{p-2}}},
\]
the function $f(w) = \eta \ell(w)$ is standard self-concordant.
If we express~$f$ in the form of~\eqref{eqn:standard-f}, then
$\phi(w,(x,y)) = \eta \varphi_p(y w^T x)$ and $\lambda = \eta \gamma$.
It is easy to verify that Assumption~\ref{asmp:smoothness} holds with
\begin{align*}
V_0= \eta , \qquad 
G = \eta B, \qquad 
L = \eta ( B^2 + \lambda), \qquad
M = \eta(p-2) B^3.
\end{align*}
If we choose $p = 2 + \log(1/\gamma)$, then applying 
Theorems~\ref{thm:DiSCO-stochastic} and~\ref{thm:ada-DiSCO-stochastic} yields
the following result.

\begin{corollary}\label{coro:smooth-hinge-complexity}
For the smoothed hinge loss $\varphi_p$ defined in~\eqref{eqn:smoothed-hinge}
with $p = 2 + \log(1/\gamma)$, the total number of communication rounds 
required by DiSCO to find an $\epsilon$-optimal solution is bounded by
\begin{align*}
	\E[T] =  \widetilde \order\Big( \Big(\log(1/\epsilon)+\frac{B^{3} D}{\gamma n^{1/2}} \Big)\Big(1+\frac{B d^{1/4}}{\gamma^{1/2} n^{1/4}} \Big)\Big).
\end{align*}
\end{corollary}

\noindent
Thus, the smoothed hinge loss enjoys the same communication efficiency as the
logistic loss.


\section{Numerical experiments}
\label{sec:experiments}

In this section, we conduct numerical experiments to compare the DiSCO 
algorithm with several state-of-the-art distributed optimization
algorithms: the ADMM algorithm (\eg, \cite{Boyd10ADMM}),
the accelerated full gradient method (AFG) \cite[Section~2.2]{Nesterov04book},
the L-BFGS quasi-Newton method (\eg, \cite[Section~7.2]{NocedalWrightbook}),
and the DANE algorithm~\cite{ShamirSrebroZhang14DANE}. 

The algorithms ADMM, AFG and L-BFGS are well known and each has a rich 
literature. 
In particular, using ADMM for empirical risk minimization in a distributed 
setting is straightforward; see \cite[Section~8]{Boyd10ADMM}.
For AFG and L-BFGS, we use the simple distributed implementation
discussed in Section~\ref{sec:communication-efficiency}: 
at each iteration~$k$, each machine computes the local gradients $f'_i(w_k)$
and sends it to the master machine to form 
$f'(w_k)=(1/m)\sum_{i=1}^m f'_i(w_k)$, and the master machine executes the 
main steps of the algorithm to compute $w_{k+1}$.
The iteration complexities of these algorithms stay the same as 
their classical analysis for a centralized implementation, 
and each iteration usually involves one or two rounds of communication.

Here we briefly describe the DANE (Distributed Approximate NEwton) algorithm
proposed by Shamir et al.~\cite{ShamirSrebroZhang14DANE}.
Each iteration of DANE takes two rounds of communication to compute $w_{k+1}$
from~$w_k$.
The first round of communication is used to compute the gradient
$f'(w_k)=(1/m)\sum_{i=1}^m f'_i(w_k)$.
Then each machine solves the local minimization problem
\[
    v_{k+1,i} = \arg\min_{w\in\R^d} ~\left\{f_i(w) - \langle \fp_i(w_k) - \fp(w_k), w \rangle + \frac{\mu}{2} \ltwos{ w - w_k}^2 \right\},
\]
and take a second round of communication to compute 
$w_{k+1} = (1/m)\sum_{i=1}^m v_{k+1,i}$. 
Here $\mu\geq 0$ is a regularization parameter with a similar role as in DiSCO.
For minimizing the quadratic loss in~\eqref{eqn:quadratic-erm},
the iteration complexity of DANE is 
$\widetilde\order((L/\lambda)^2 n^{-1}\log(1/\epsilon))$.
As summarized in Table~\ref{tab:complexities}, 
if the condition number $L/\lambda$ grows as $\sqrt{mn}$, then DANE is more
efficient than AFG and ADMM when~$n$ is large. 
However, the same complexity cannot be guaranteed for minimizing 
non-quadratic loss functions.
According to the analysis in \cite{ShamirSrebroZhang14DANE},
the convergence rate of DANE on non-quadratic functions
might be as slow as the ordinary full gradient descent method.

\subsection{Experiment setup}

\begin{table}
\renewcommand{\arraystretch}{1.1}
\centering
\begin{tabular}{|c|c|r|r|r|}
  \hline
  Dataset name & number of samples & number of features & sparsity \\
  \hline
  Covtype & 581,012 & 54  & 22\% \\\hline
  RCV1 & 20,242 & 47,236 & 0.16\% \\\hline
  News20 & 19,996 & 1,355,191  & 0.04\% \\
  \hline
\end{tabular}
\caption{Summary of three binary classification datasets.}
\label{tab:data-summary}
\end{table}

\begin{figure}[t]
\begin{tabular}{c|ccc}
$m$ & Covtype & RCV1 & News20 \\\hline
&&&\\
4&
\raisebox{-.5\height}{\includegraphics[width = 0.28\textwidth]{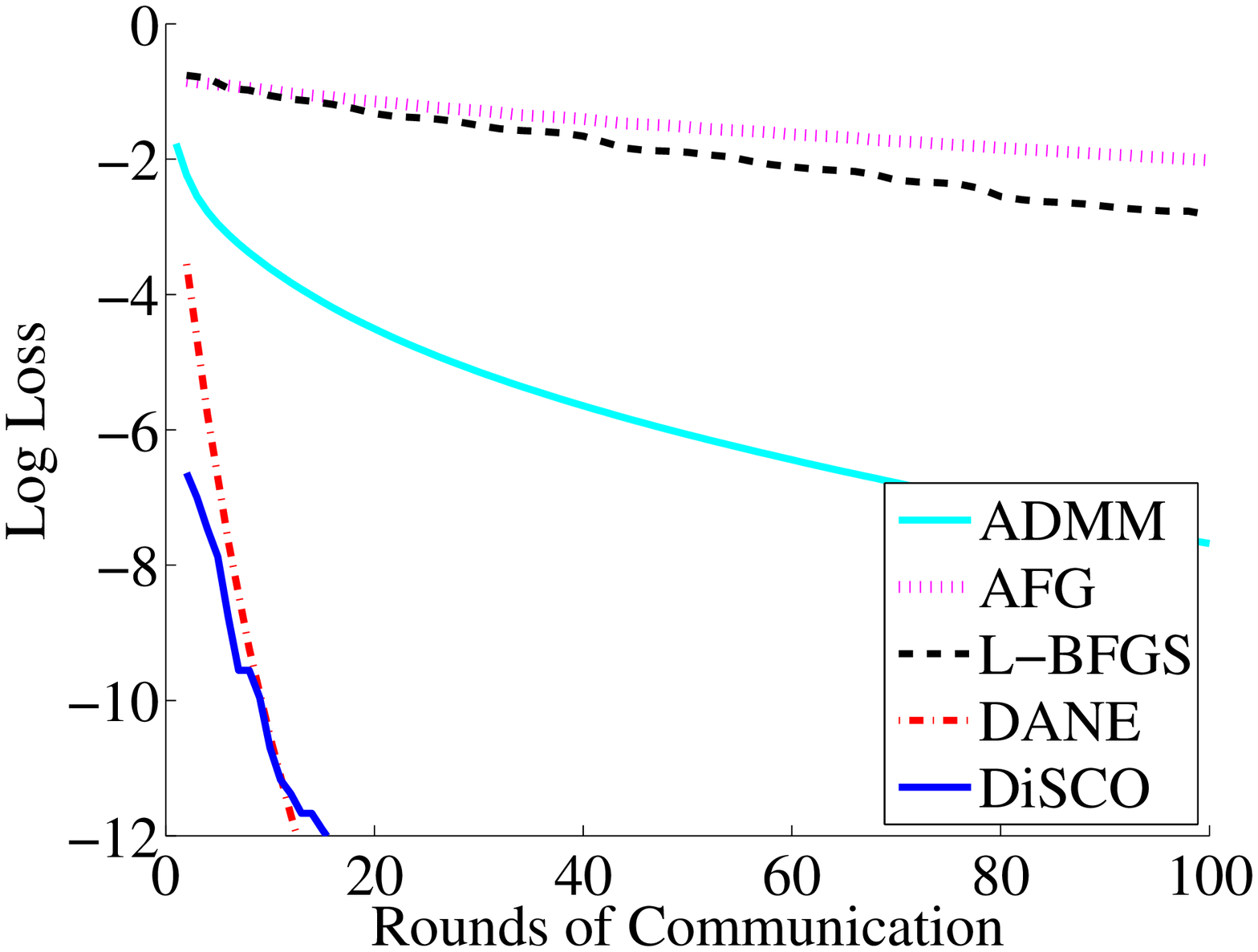}} &
\raisebox{-.5\height}{\includegraphics[width = 0.28\textwidth]{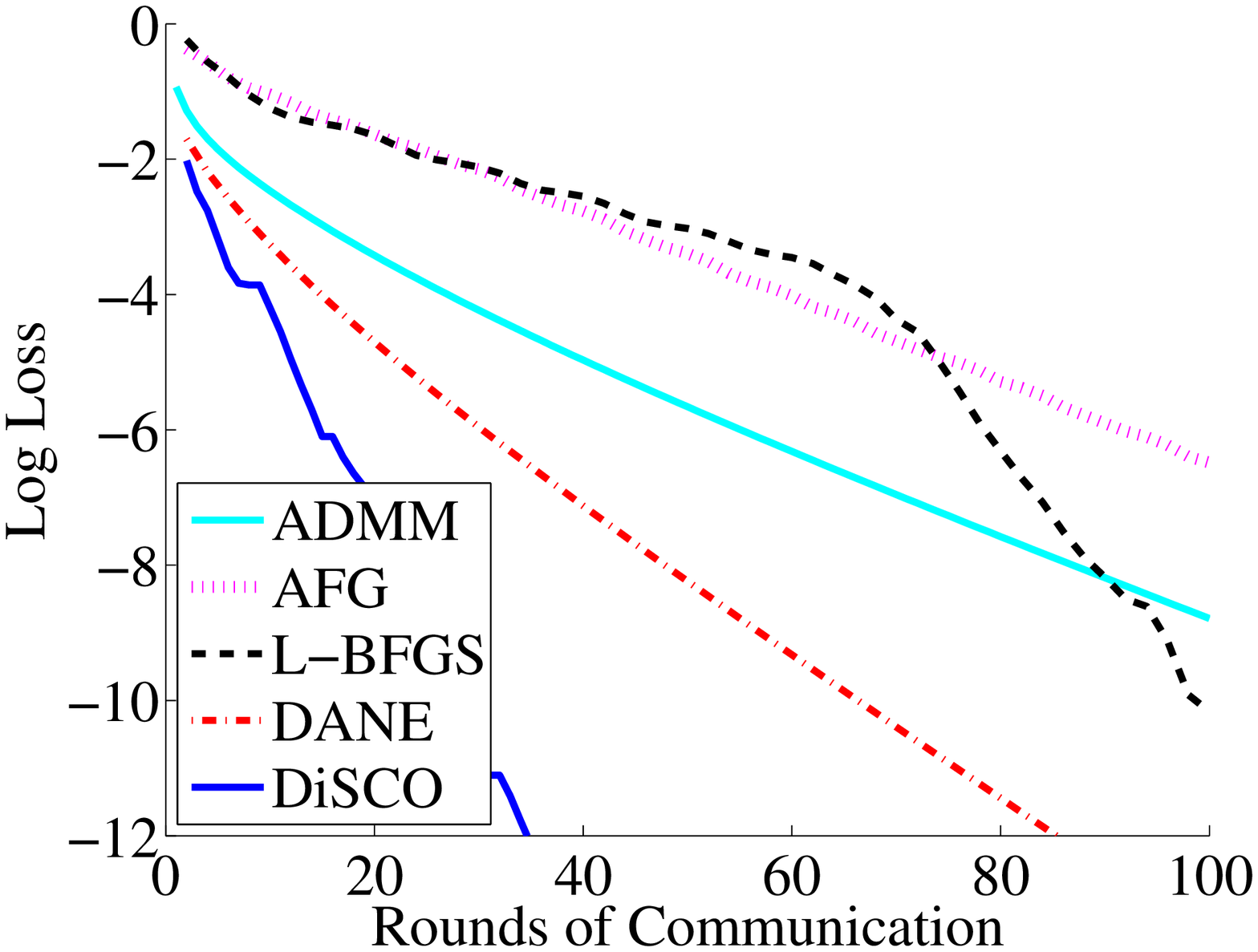}} &
\raisebox{-.5\height}{\includegraphics[width = 0.28\textwidth]{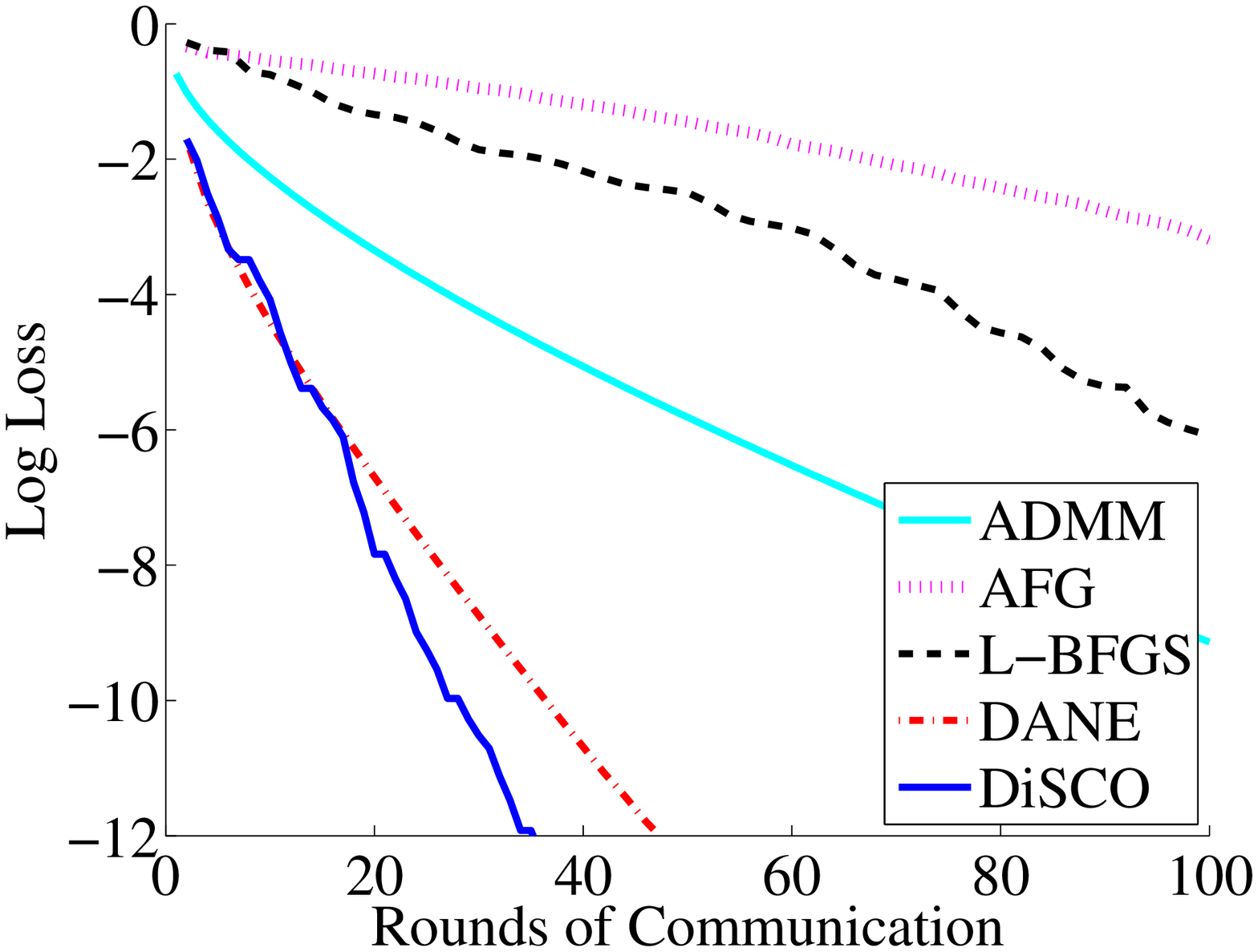}} \\
16&
\raisebox{-.5\height}{\includegraphics[width = 0.28\textwidth]{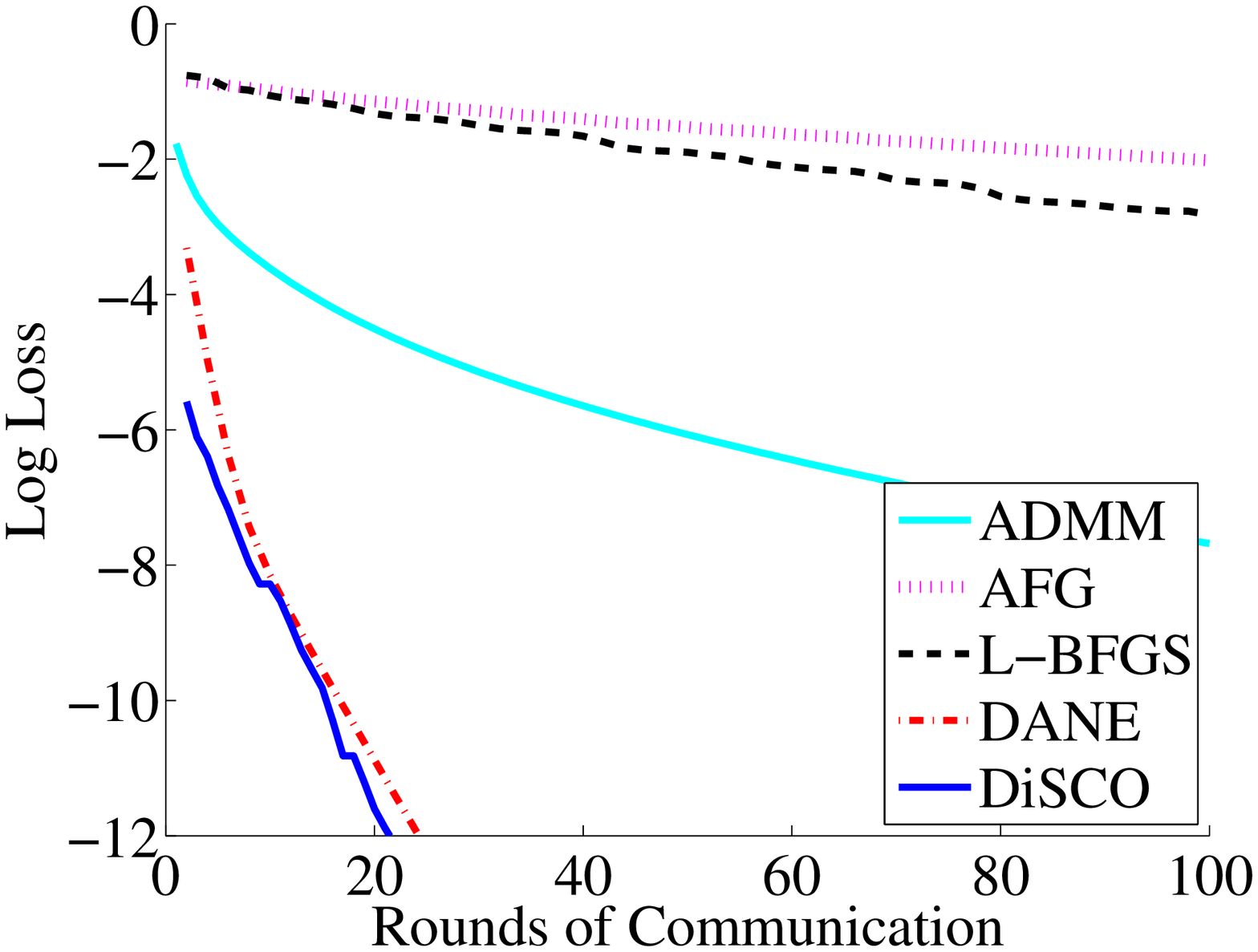}} &
\raisebox{-.5\height}{\includegraphics[width = 0.28\textwidth]{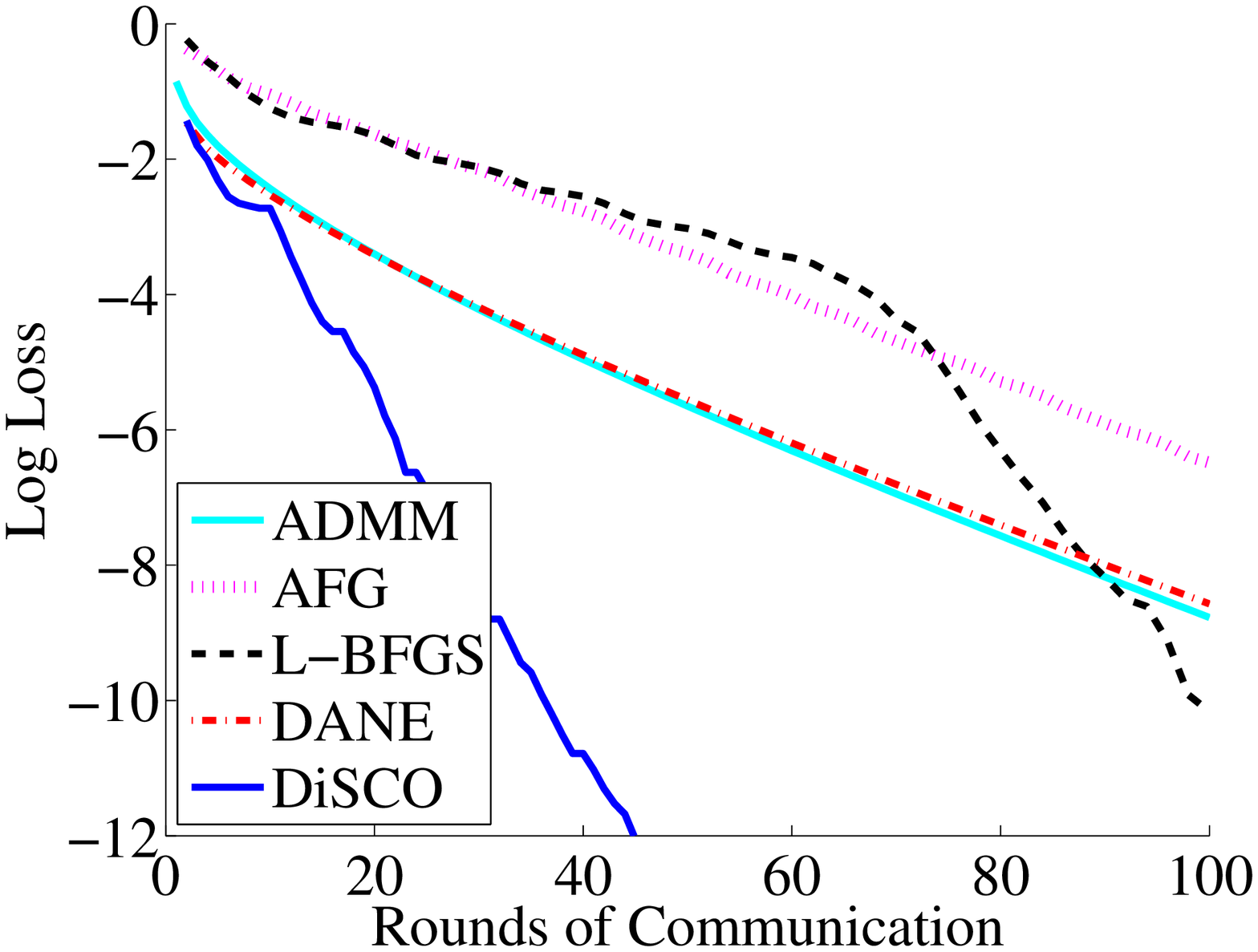}} &
\raisebox{-.5\height}{\includegraphics[width = 0.28\textwidth]{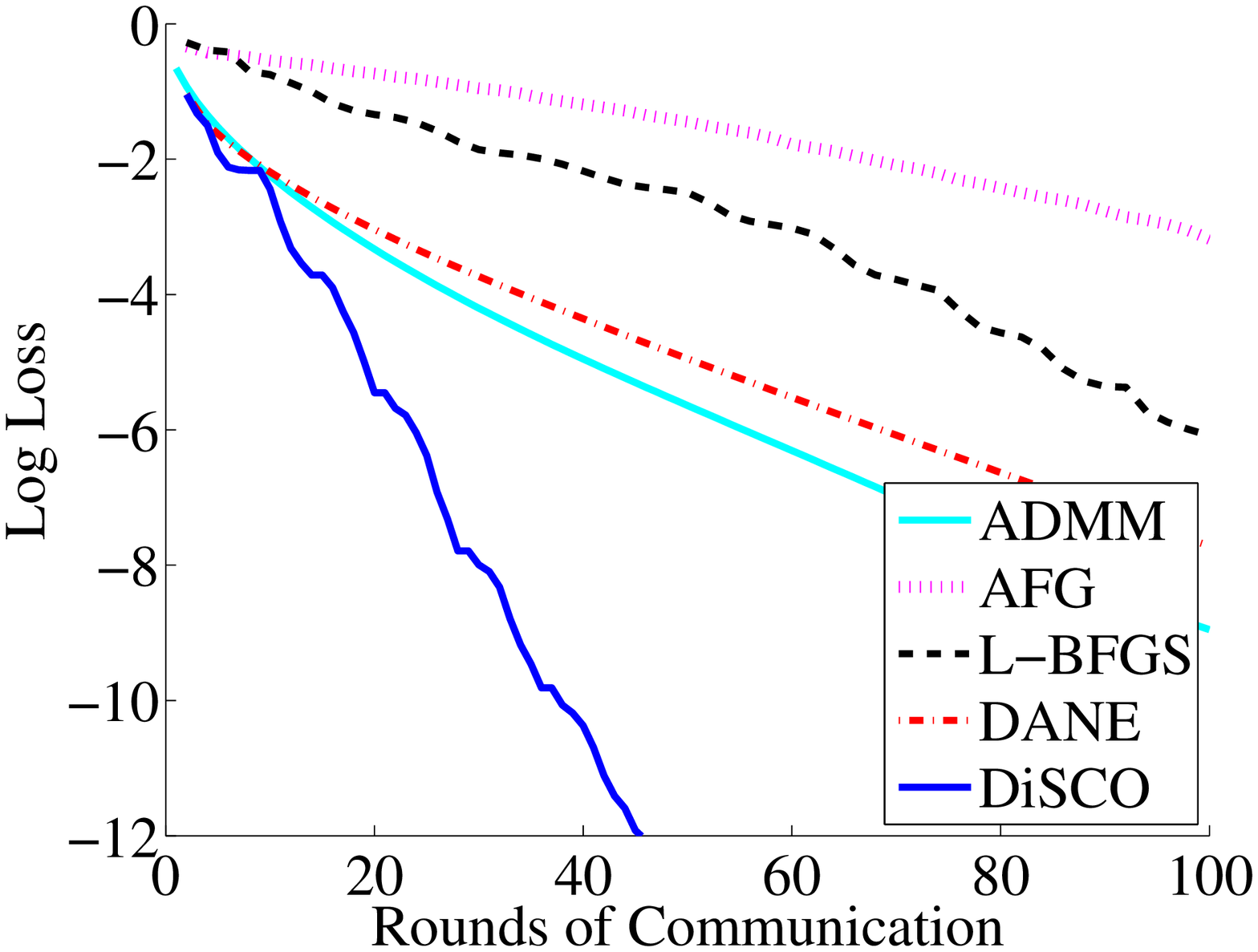}} \\
64&
\raisebox{-.5\height}{\includegraphics[width = 0.28\textwidth]{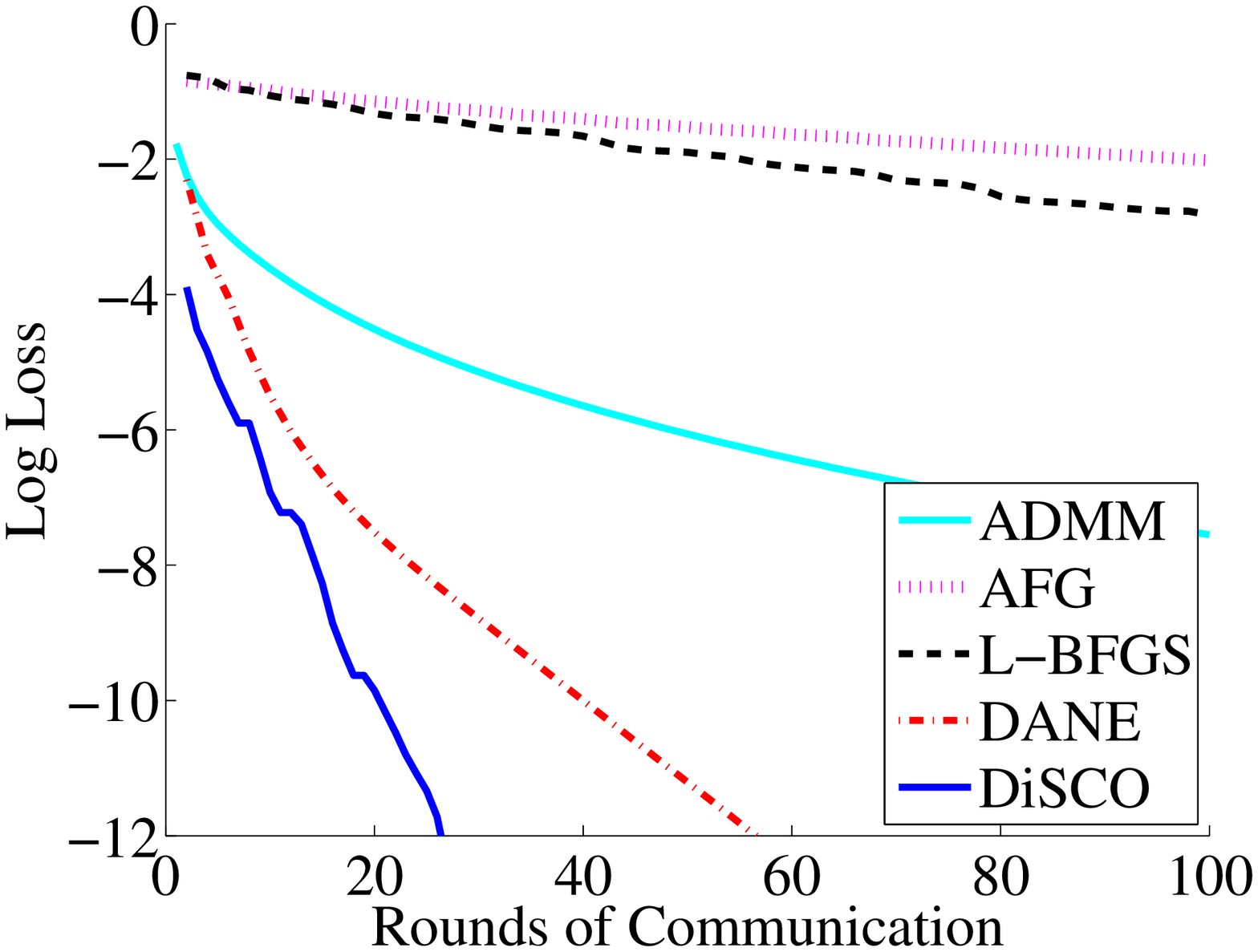}} &
\raisebox{-.5\height}{\includegraphics[width = 0.28\textwidth]{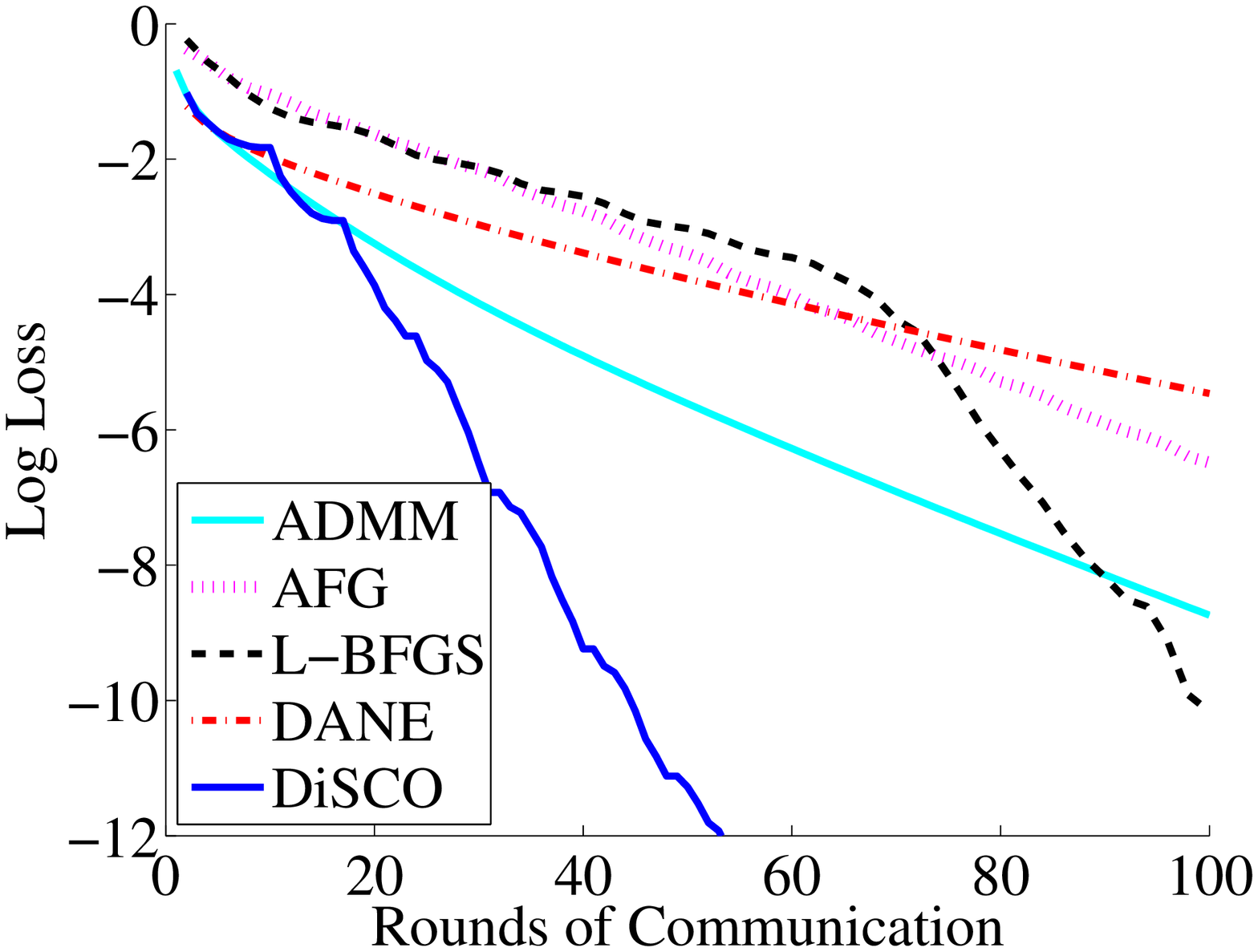}} &
\raisebox{-.5\height}{\includegraphics[width = 0.28\textwidth]{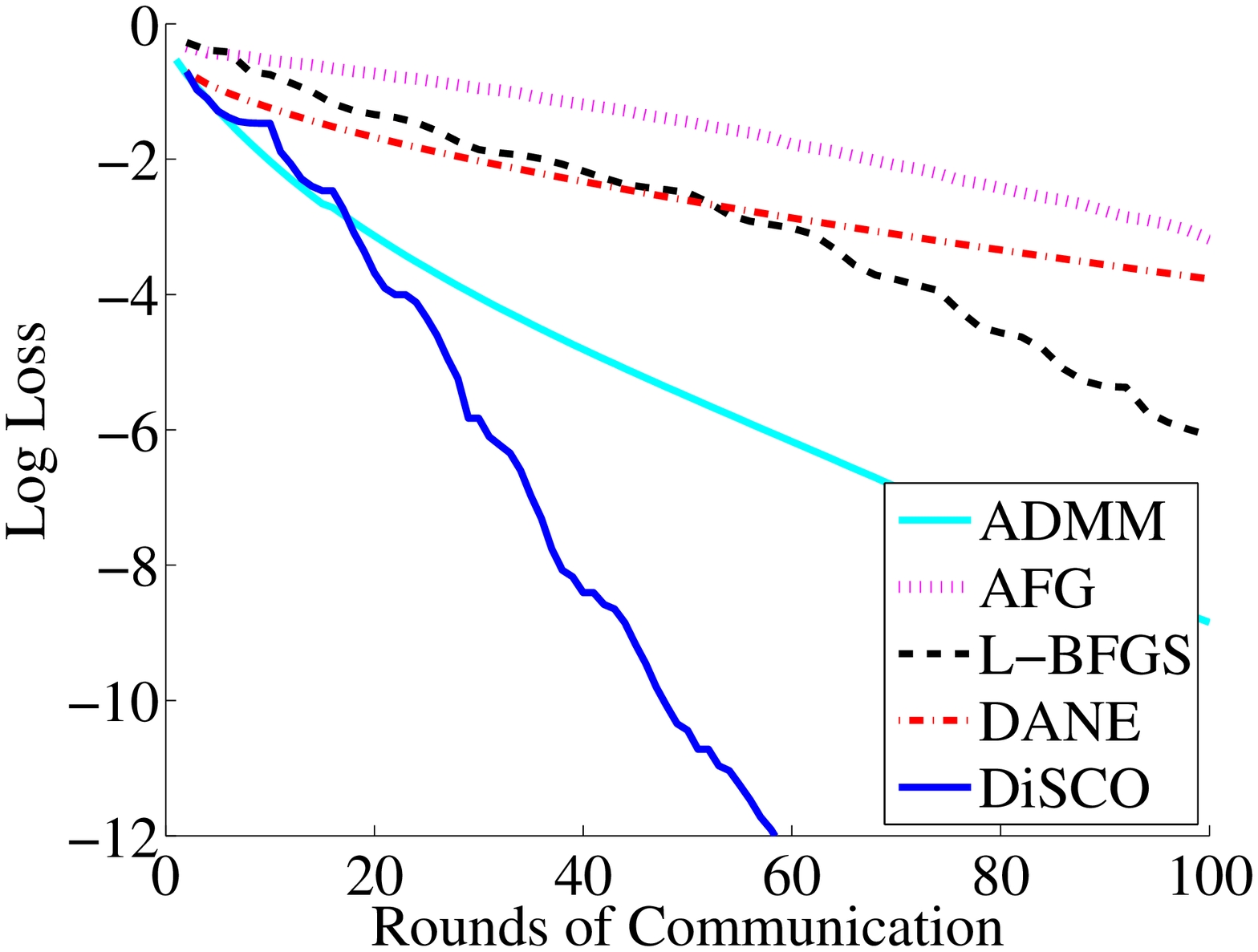}}
\end{tabular}
\caption{Comparing DiSCO with other distributed optimization algorithms. 
    We splits each dataset evenly to~$m$ machines, with $m \in \{4,16,64\}$. 
    Each plot above shows the reduction of the logarithmic gap 
    $\log_{10}(\ell(\what) - \ell(\wstar))$ (the vertical axis) versus
    the number of communication rounds (the horizontal axis)
taken by each algorithm.}
\label{fig:real-compare}
\end{figure}

For comparison, 
we solve three binary classification tasks using logistic regression.
The datasets are obtained from the LIBSVM datasets~\cite{chang2011libsvm}
and summarized in Table~\ref{tab:data-summary}. 
These datasets are selected to cover different relations between 
the sample size $N=mn$ and the feature dimensionality~$d$:
$N \gg d$ (Covtype \cite{covertype}), 
$N \approx d$ (RCV1 \cite{RCV1}) and 
$N \ll d$ (News20 \cite{News20binary,Lang95News20}).
For each dataset, our goal is to minimize the
regularized empirical loss function:
\begin{align*}
	\ell(w) = \frac{1}{N} \sum_{i=1}^N \log(1 + \exp( - y_i ( w^T x_i ))) + \frac{\gamma}{2}\ltwos{w}^2
\end{align*}
where $x_i\in\R^d$ and $y_i\in\{-1, 1\}$. 
The data have been normalized so that $\|x_i\|=1$ for all $i=1,\ldots,N$.
The regularization parameter is set to be $\gamma = 10^{-5}$.

We describe some implementation details.
In Section~\ref{sec:classification}, the theoretical analysis suggests that 
we scale the function $\ell(w)$ by a factor $\eta = {B^2}/({4\gamma})$.
Here we have $B=1$ due to the normalization of the data.
In practice, we find that DiSCO converges faster without rescaling.
Thus, we use $\eta = 1$ for all experiments.
For Algorithm~\ref{alg:DiSCO}, we choose the input parameters 
$\mu = m^{1/2}\mu_0$, where $\mu_0$ is chosen manually.
In particular, we used $\mu_0 = 0$ for Covtype,
$\mu_0 = 4\times 10^{-4}$ for RCV1, and $\mu_0 = 2\times 10^{-4}$  for News20.
For the distributed PCG method (Algorithm~\ref{alg:distributed-pcg}), 
we choose the stopping precision $\epsilon_k = \ltwos{\fp(w_k)}/10$.

Among other methods in comparison, 
we manually tune the penalty parameter of ADMM
and the regularization parameter~$\mu$ for DANE to optimize their performance.
For AFG, we used an adaptive line search scheme 
\cite{Nesterov13composite,LinXiao14acclprox} to speed up its convergence.
For L-BFGS, we adopted the memory size $30$ 
(number of most recent iterates and gradients stored)
as a general rule of thumb suggested in \cite{NocedalWrightbook},

We want to evaluate DiSCO not only on $w_k$, but also in the middle of
calculating $v_k$, to show its progress after each round of communication.
To this end, we follow equation~\eqref{eqn:update-wk-before-rescaling}
to define an intermediate solution $\what_k^t$ for each 
iteration $t$ of the distributed PCG method 
(Algorithm~\ref{alg:distributed-pcg}):
\[
    \what_k^{t} = w_k - \frac{v^{(t)}}{1 + \sqrt{\eta}\bigl(v^{(t)})^T \ell''(w_k) v^{(t)}\bigr)^{1/2}},
\]
and evaluate the associated objective function $\ell(\what_k^{t})$. 
This function value is treated as a measure of progress after each round of 
communication. 

\subsection{Performance evaluation}

\begin{figure}
\centering
\begin{tabular}{ccc}
\includegraphics[width = 0.3\textwidth]{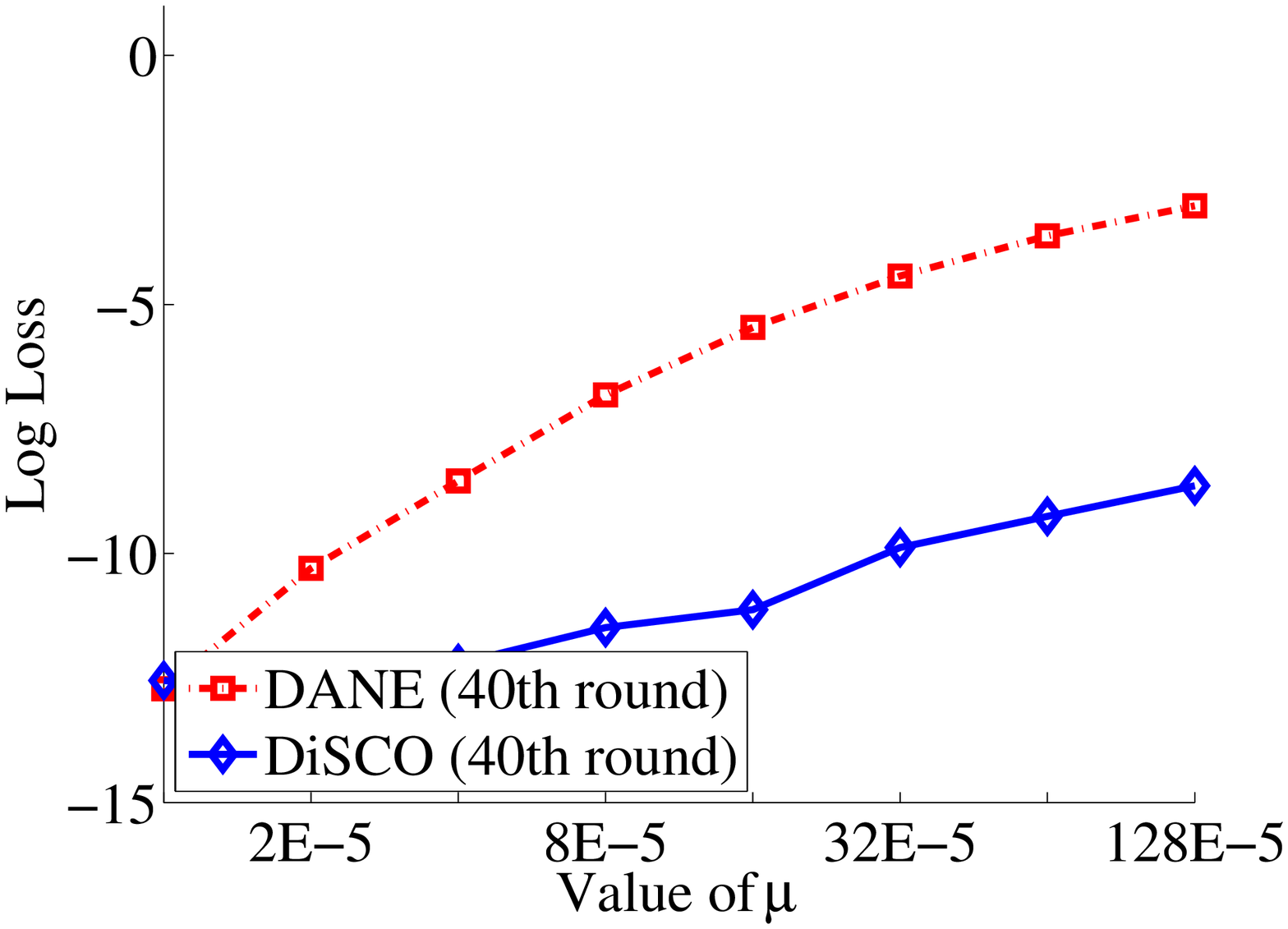} &
\includegraphics[width = 0.3\textwidth]{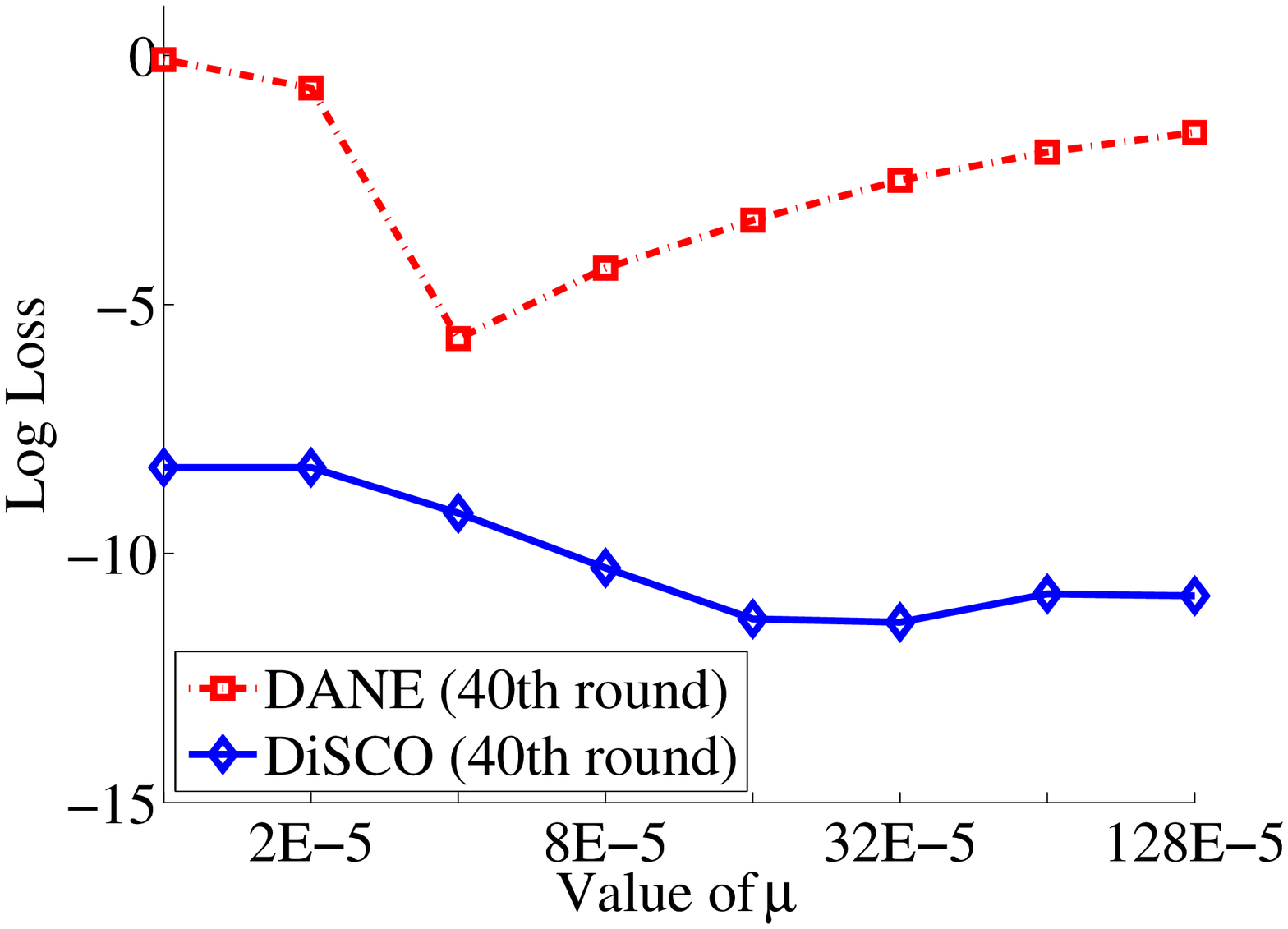} &
\includegraphics[width = 0.3\textwidth]{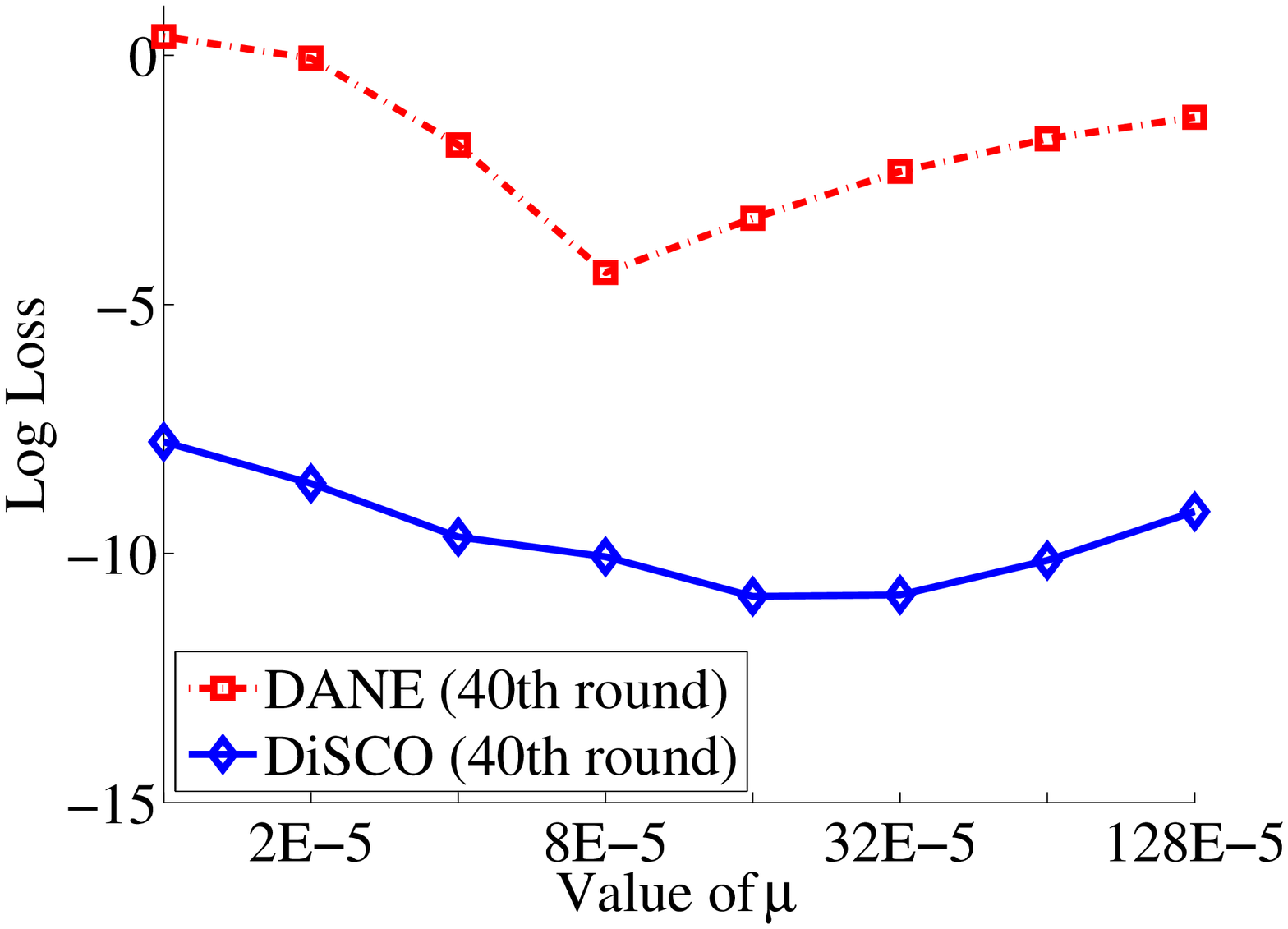}\\
Covtype & RCV1 & News20
\end{tabular}
\caption{Comparing the sensitivity of DiSCO and DANE with respect to the
    regularization parameter~$\mu$, when the datasets are split on
    $m=16$ machines.
    We varied~$\mu$ from $10^{-5}$ to $128\times 10^{-5}$. 
The vertical axis is the logarithmic gap $\log_{10}(\ell(\what)-\ell(\wstar))$
after~$40$ rounds of communications.}
\label{fig:sensitivity-compare}
\end{figure}

It is important to note that different algorithms take different number of
communication rounds per iteration.
ADMM requires one round of communication per iteration.
For AFG and L-BFGS, each iteration consists of at least two rounds of communications: one for finding the descent direction,
and another one or more for searching the stepsize. For DANE, there are also two rounds of communications
per iteration, for computing the gradient and for aggregating the local solutions. For DiSCO, 
each iteration in the inner loop takes one round of communication, and there is an additional round
of communication at the beginning of each inner loop. Since we are interested in the communication efficiency
of the algorithms, we plot their progress in reducing the objective value
with respect to the number of communication rounds taken. 

We plot the performance of ADMM, AFG, L-BFGS, DANE and DiSCO in
Figure~\ref{fig:real-compare}. According to the plots, DiSCO 
converges substantially faster than ADMM and AFG. 
It is also notably faster than L-BFGS and DANE. 
In particular, the convergence speed (and the communication efficiency) of 
DiSCO is more robust to the number of machines in the distributed system. 
For $m=4$, the performance of DiSCO is somewhat comparable to that of DANE. 
As~$m$ grows to~$16$ and~$64$, the convergence of DANE becomes significantly 
slower, while the performance of DiSCO only degrades slightly.
This coincides with the theoretical analysis: 
the iteration complexity of DANE is proportional to~$m$, 
but the iteration complexity of DiSCO is proportional to $m^{1/4}$.

Since both DANE and DiSCO take a regularization parameter~$\mu$, 
we study their sensitivity to the choices of this parameter. 
Figure~\ref{fig:sensitivity-compare} shows the performance of DANE and DiSCO 
with the value of~$\mu$ varying from $10^{-5}$ to $128\times 10^{-5}$. 
We observe that the curves of DiSCO are relatively smooth and stable. 
In contrast, the curves of DANE exhibit sharp valley at particular 
values of~$\mu$.
This suggests that DiSCO is more robust to the non-optimal choice of parameters. 

\section{Extension to distributed composite minimization}
\label{sec:composite-minimization}

Thus far, we have studied the problem of minimizing empirical loss functions
that are standard self-concordant.
In this section, we sketch how to extend the DiSCO algorithm to solve 
distributed composite minimization problems. 
By composite minimization, 
we consider the minimization of
\begin{align}\label{eqn:composite-obj}
	F(w) \eqdef f(w) + \Psi(w),
\end{align}
where~$f$ is a standard self-concordant function taking the 
form of~\eqref{eqn:average-obj},
and~$\Psi$ a closed convex function 
with a simple structure (see discussions in \cite{Nesterov13composite}).
For solving the Lasso~\cite{tibshirani1996regression}, for example, 
the $\ell_1$-penalty $\Psi(w)=\sigma\|w\|_1$ with $\sigma>0$ 
is nonsmooth but admits a simple proximal mapping.

We modify Algorithm~\ref{alg:damped-newton-method} and 
Algorithm~\ref{alg:distributed-pcg} to minimize the composite function $F(w)$.
To modify Algorithm~\ref{alg:damped-newton-method}, 
we update $w_{k+1}$ using an inexact version of the proximal-Newton 
method (\eg, \cite{LeeSunSaunders14proxNewton,tran2013composite}).
More specifically, the two steps in each iteration of 
Algorithm~\ref{alg:damped-newton-method} are replaced with:
\begin{enumerate}
	\item Find a vector $v_k$ that is an approximate solution of
        \begin{equation}\label{eqn:prox-Newton-min}
        \minimize_{v\in\R^d} \quad \left \{\frac{1}{2}v^T \fpp(w_k)v 
        - v^T \fp(w_k) + \Psi(w_k-v)\right\}.
    \end{equation}
    \item Update $w_{k+1} = w_k - \frac{v_k}{1 + \sqrt{v^T \fpp(w_k) v_k}}$.
\end{enumerate}
Note that for $\Psi(w)\equiv 0$, the above proximal-Newton method reduces to 
Algorithm~\ref{alg:damped-newton-method}.
Since~$v_k$ only needs to be an inexact solution to 
problem~\eqref{eqn:prox-Newton-min}, 
we need a measure to quantify the approximation error.
For this purpose, we define the following gradient mapping
\[
	g(v_k) \eqdef \arg\min_{g\in \R^d} \Big\{\frac{L}{2}\ltwos{g}^2 
+ \langle \fpp(w_k)v_k - \fp(w_k), g \rangle + \Psi(w_k-v_k+g) \Big\}.
\]
If $v_k$ is an exact minimizer of~\eqref{eqn:prox-Newton-min}, 
then we have $\ltwos{g(v_k)} = 0$.
In the distributed setting, we only need to find a vector 
$v_k$ such that $\ltwos{g(v_k)}\leq\epsilon_k$.

It remains to devise an distributed algorithm to compute an inexact 
minimizer~$v_k$.
Since the objective function in~\eqref{eqn:prox-Newton-min} is not quadratic, 
we can no longer employ the distributed PCG method in 
Algorithm~\ref{alg:distributed-pcg}.
Instead, we propose a preconditioned accelerated proximal gradient method.
In particular, we modify the algorithm on the master machine
in Algorithm~\ref{alg:distributed-pcg} as follows:
\begin{align}
    v^\suptp & = \arg\min_{v\in\R^d} ~ \Big\{\frac{1}{2}(v - s^\supt)^T [\fpp_1(w_k) + \mu I] (v - s^\supt) \nonumber\\
	&\qquad\qquad\qquad  + \langle \fpp(w_k)s^\supt - \fp(w_k) , v - s^\supt\rangle + \Psi(w_k + v) \Big\},\label{eqn:accel-proximal-gradient}\\
	s^\suptp & = v^\suptp + \frac{\sqrt{1 + 2\mu/\reg} - 1}{\sqrt{1 + 2\mu/\reg} + 1} (v^\suptp - v^\supt),\nonumber
\end{align}
where $s^{(t+1)}$ is an auxiliary vector.
We output $v_k=v^\suptp$ once the condition $\ltwos{g(v^\suptp)}\leq\epsilon_k$
is satisfied.
Each update takes one round of communication to compute the vector 
$\fpp(w_k)s^\supt$. 
Then, the sub-problem~\eqref{eqn:accel-proximal-gradient} is locally solved by
the master machine.
This problem has similar structure as problems~\eqref{eqn:local-solution}
and~\eqref{eqn:computation-linear-system}, and can be solved in
$\order\bigl((n+\frac{L+\mu}{\lambda+\mu})\log(1/\epsilon)\bigr)$ time
using the methods proposed in 
\cite{SSZhang13SDCA,XiaoZhang14ProxSVRG,DefazioBach14SAGA}.

If we replace the first term on the right-hand side of 
equation~\eqref{eqn:accel-proximal-gradient}
by $\frac{L}{2}\ltwos{v - v^\supt}^2$ and set $\mu = \hessianbound$, 
then the above algorithm is exactly the accelerated proximal 
gradient algorithm~\cite{Nesterov13composite,LinXiao14acclprox}, 
which converges in $\widetilde \order(\sqrt{\hessianbound/\reg})$ iterations. 
By utilizing the similarity between $\fpp_1(w_k)$ and $\fpp(w_k)$, 
and assuming $\ltwos{\fpp_1(w_k) - \fpp(w_k)}\leq \mu$ for all $k\geq 0$,
it can be shown that our algorithm in~\eqref{eqn:accel-proximal-gradient}
converges in $\widetilde \order(1+\sqrt{\mu/\lambda})$ iterations, 
which is of the same order as the PCG algorithm.

In summary, to minimize the composite function $f(w)+\Psi(w)$,
we replace Algorithm~\ref{alg:damped-newton-method} 
by the inexact proximal Newton method, and replace
Algorithm~\ref{alg:distributed-pcg} by a distributed implementation
of the above preconditioned accelerated proximal gradient method.
Under the same assumptions on~$f$, we can obtain similar guarantees
on the communication efficiency as stated in
Theorems~\ref{thm:DiSCO-stochastic} and~\ref{thm:ada-DiSCO-stochastic}.

\section{Conclusions}\label{sec:conclusion}

We considered distributed convex optimization problems originated from
SAA or ERM,
which involve large amount of i.i.d.\ data stored on a distributed computing
system.
Since the cost of inter-machine communication is very high in practice,
communication efficiency is a critical measure in evaluating the performance
of a distributed algorithm.
For algorithms based on first-order methods,
including accelerated gradient methods and ADMM, 
the required number of communication rounds grows with the condition number
of the objective function. 
The condition number itself often grows with 
the number of samples due to weaker regularization required.
This causes the total number of communication rounds to grow with
the overall sample size.

In this paper, 
we proposed and analyzed DiSCO, a communication-efficient distributed algorithm
for minimizing self-concordant empirical loss functions, 
and discussed its application to linear regression and classification.
DiSCO is based on an inexact damped Newton method, where the inexact Newton 
steps are computed by a distributed preconditioned conjugate gradient method.
In a standard setting for supervised learning, its required number of
communication rounds does not increase with the sample size,
but only grows slowly with the number of machines in the distributed system.
There are three main thrusts in our approach:
\begin{itemize}
    \item \emph{Self-concordant analysis.} 
        We showed that several popular empirical loss functions used in 
        machine learning are either self-concordant or can be well approximated
        by self-concordant functions.
        We gave complexity analysis of the inexact damped Newton method,
        and characterized the conditions for both linear and superlinear 
        convergence.
       
    \item \emph{Preconditioned conjugate gradient (PCG) method.}
        We proposed a distributed implementation of the PCG method for computing
        the inexact Newton step. In particular, the preconditioner based on
        similarity between local and global Hessians is very effective in 
        reducing the number of communication rounds, 
        both in theory and practice.

    \item \emph{Stochastic analysis of communication efficiency.}
        Our main theoretical results combine two consequences of averaging over
        a large number of i.i.d.\ samples. 
        One is the expected reduction of the initial objective value, 
        which counters the effect of objective scaling required to make
        the objective function standard self-concordant.
        The other is a high-probability bound that characterizes the similarity
        between the local and global Hessians.
\end{itemize}
Our numerical experiments on real datasets confirmed the superior communication
efficiency of the DiSCO algorithm.
In addition, we also proposed an extension for solving distributed optimization
problems with composite empirical loss functions.


\newpage
\appendix

\centerline{\huge Appendices}

\section{Proof of Theorem~\ref{thm:damped-newton-convergence}}
\label{sec:proof-newton-method-convergence-rate}

First, we recall the definitions of the two auxiliary functions
\begin{align*}
    \omega(t)   & = t - \log(1+t), \qquad  t\geq 0, \\
    \omega_*(t) & = -t - \log(1-t), \qquad 0\leq t < 1,
\end{align*}
which form a pair of convex conjugate functions.

%
We notice that 
Step~2 of Algorithm~\ref{alg:damped-newton-method} is equivalent to
\[
    w_{k+1} - w_k = \frac{v_k}{1+\delta_k} = \frac{v_k}{1+\ltwos{\vtilde_k}},
\]
which implies
\begin{align}\label{eqn:step-size-not-too-large}
	\ltwos{[\fpp(w_k)]^{1/2}(w_{k+1} - w_k)} = \frac{\ltwos{\vtilde_k}}{1 + \ltwos{\vtilde_k}} < 1.
\end{align}
When inequality~\eqref{eqn:step-size-not-too-large} holds, 
Nesterov~\cite[Theorem~4.1.8]{Nesterov04book} has shown that
\begin{align*}
	f(w_{k+1}) \leq f(w_k) + \langle \fp(w_k), w_{k+1} - w_k \rangle 
    + \omega_*\bigl(\ltwos{[\fpp(w_k)]^{1/2}(w_{k+1} - w_k)}\bigr) .
\end{align*}
Using the definition of functions $\omega$ and $\omega_*$, and with some algebraic operations, we obtain
\begin{align}\label{eqn:copy-nesterov-4-1-12}
 f(w_{k+1}) &\leq f(w_k) - \frac{\langle \utilde_k, \vtilde_k \rangle}{1 + \ltwos{\vtilde_k}} - \frac{\ltwos{\vtilde_k}}{1 + \ltwos{\vtilde_k}}
 + \log(1 + \ltwos{\vtilde_k} ) \nonumber\\
&= f(w_k) - \omega(\ltwos{\utilde_k}) + \big( \omega(\ltwos{\utilde_k}) - \omega(\ltwos{\vtilde_k})\big) + \frac{\langle \vtilde_k - \utilde_k, \vtilde_k \rangle}{1 + \ltwos{\vtilde_k}}.
\end{align}

By the second-order mean-value theorem, we have
\[
    \omega(\ltwos{\utilde_k}) - \omega(\ltwos{\vtilde_k}) 
    = \omega'(\ltwos{\vtilde_k})(\ltwos{\utilde_k}-\ltwos{\vtilde_k})
    +\frac{1}{2}\omega''(t)\left(\ltwos{\utilde_k}-\ltwos{\vtilde_k}\right)^2
\]
for some~$t$ satisfying
\[
    \min\{\ltwos{\utilde_k},\ltwos{\vtilde_k}\}
    \leq t \leq
    \max\{\ltwos{\utilde_k},\ltwos{\vtilde_k}\} .
\]
Using the inequality~\eqref{eqn:bound-vk-by-uk}, 
we can upper bound the second derivative $\omega''(t)$ as
\[
    \omega''(t) = \frac{1}{(1+t)^2} \leq \frac{1}{1+t} 
    \leq \frac{1}{1+\min\{\ltwos{\utilde_k},\ltwos{\vtilde_k}\} }
    \leq \frac{1}{1+(1-\beta)\ltwos{\utilde_k}}.
\]
Therefore, 
\begin{align*}
    \omega(\ltwos{\utilde_k}) - \omega(\ltwos{\vtilde_k}) 
&=\frac{(\ltwos{\utilde_k}-\ltwos{\vtilde_k})\ltwos{\vtilde_k}}{1+\ltwos{\vtilde_k}} +\frac{1}{2}\omega''(t)\left(\ltwos{\utilde_k}-\ltwos{\vtilde_k}\right)^2\\
&\leq \frac{\ltwos{\utilde_k-\vtilde_k} \ltwos{\vtilde_k}}{1+(1-\beta)\ltwos{\utilde_k}} + \frac{(1/2)\ltwos{\utilde_k-\vtilde_k}^2}{1+(1-\beta)\ltwos{\utilde_k}}\\
&\leq \frac{\beta(1+\beta)\ltwos{\utilde_k}^2+(1/2)\beta^2\ltwos{\utilde_k}^2}{1+(1-\beta)\ltwos{\utilde_k}}
\end{align*}
In addition, we have
\[
\frac{\langle \vtilde_k - \utilde_k, \vtilde_k \rangle}{1 + \ltwos{\vtilde_k}}
\leq \frac{\ltwos{\utilde_k-\vtilde_k} \ltwos{\vtilde_k}}{1+\ltwos{\vtilde_k}}
\leq \frac{\beta(1+\beta)\ltwos{\utilde_k}^2}{1+(1-\beta)\ltwos{\utilde_k}}.
\]
Combining the two inequalities above, and using the relation
$t^2/(1+t)\leq 2 \omega(t)$ for all $t\geq 0$, we obtain
\begin{align*}
    \omega(\ltwos{\utilde_k}) - \omega(\ltwos{\vtilde_k}) 
+\frac{\langle \vtilde_k - \utilde_k, \vtilde_k \rangle}{1 + \ltwos{\vtilde_k}}
&\leq \left(2\beta(1+\beta)+(1/2)\beta^2\right) \frac{\ltwos{\utilde_k}^2}{1+(1-\beta)\ltwos{\utilde_k}} \\
&= \left(\frac{2\beta+(5/2)\beta^2}{(1-\beta)^2}\right) \frac{(1-\beta)^2\ltwos{\utilde_k}^2}{1+(1-\beta)\ltwos{\utilde_k}} \\
&\leq \left(\frac{2\beta+(5/2)\beta^2}{(1-\beta)^2}\right) 2 \omega\bigl( (1-\beta)\ltwos{\utilde_k} \bigr) \\
&\leq \left(\frac{4\beta+5\beta^2}{1-\beta}\right) \omega\bigl( \ltwos{\utilde_k} \bigr) .
\end{align*}
In the last inequality above, we used the fact that for any $t\geq 0$ we have
$\omega((1-\beta)t) \leq (1-\beta)\omega(t)$, which is the result of 
convexity of $\omega(t)$ and $\omega(0)=0$; more specifically,
\[
    \omega((1-\beta)t) = \omega(\beta\cdot 0 + (1-\beta) t)
    \leq \beta\omega(0) + (1-\beta)\omega(t)
    = (1-\beta)\omega(t).
\]
Substituting the above upper bound into 
inequality~\eqref{eqn:copy-nesterov-4-1-12} yields
\begin{align}\label{eqn:newton-contraction-basic-inequality}
f(w_{k+1}) \leq f(w_k) - 
\left(1 - \frac{4\beta+5\beta^2}{1-\beta} \right) \omega(\ltwos{\utilde_k}).
\end{align}

With inequality~\eqref{eqn:newton-contraction-basic-inequality}, we are ready to prove the statements of the lemma.
In particular, Part~(a) of the Lemma holds for any $0\leq\beta\leq 1/10$. 

For part~(b), we assume that $\ltwos{\utilde_k} \leq 1/6$. 
According to~\cite[Theorem~4.1.13]{Nesterov04book},
when $\ltwos{\utilde_k} < 1$, it holds that for every $k\geq 0$,
\begin{align}\label{eqn:bound-regret-using-remainder}
	\omega(\ltwos{\utilde_k}) \leq f(w_k) - f(\wstar) \leq \omega_*(\ltwos{\utilde_k}) .
\end{align}
Combining this sandwich inequality with inequality~\eqref{eqn:newton-contraction-basic-inequality}, we have
\begin{align}
\omega(\ltwos{\utilde_{k+1}}) 
&\leq f(w_{k+1}) - f(\wstar) \nonumber \\
&\leq f(w_k) - f(\wstar) - \omega(\ltwos{\utilde_k})  + \frac{4\beta+5\beta^2}{1-\beta} \omega(\ltwos{\utilde_k})\nonumber\\
&\leq \omega_*(\ltwos{\utilde_k}) - \omega(\ltwos{\utilde_k}) + \frac{4\beta+5\beta^2}{1-\beta}\omega(\ltwos{\utilde_k}) .
    \label{eqn:newton-contraction-in-quadratic-zone}
\end{align}
It is easy to verify that $\omega_*(t) - \omega(t) \leq 0.26\,\omega(t)$ for all $t \leq 1/6$, and $(4\beta+5\beta^2)/(1-\beta) \leq 0.23$ if $\beta\leq 1/20$.
Applying these two inequalities to inequality~\eqref{eqn:newton-contraction-in-quadratic-zone} completes the proof.

It should be clear that other combinations of the value of~$\beta$ and bound 
on $\ltwos{\utilde_k}$ are also possible. 
For example, for $\beta=1/10$ and $\ltwos{\utilde_k}\leq 1/10$, we have
$\omega(\ltwos{\utilde_{k+1}})\leq 0.65\, \omega(\ltwos{\utilde_k})$.


\section{Super-linear convergence of Algorithm~\ref{alg:damped-newton-method}} 
\label{sec:superlinear-convergence}

\begin{theorem}\label{thm:superlinear-convergence}
    Suppose $f:\R^d\to\R$ is a standard self-concordant function 
    and Assumption~\ref{asmp:Hessian-bounds} holds. 
    If we choose the sequence $\{\epsilon_k\}_{k\geq 0}$ in
Algorithm~\ref{alg:damped-newton-method} as
\begin{align}\label{eqn:epsilon-k-superlinear}
    \epsilon_k =  \frac{\lambda^{1/2}}{2} \min\left\{ \frac{\omega(r_k)}{2},
    \frac{\omega^{3/2}(r_k)}{10} \right\},
    \qquad \mbox{where}\quad r_k =  L^{-1/2} \|f'(w_k)\|_2 ,
\end{align}
then:
\begin{enumerate}
    \item[(a)] For any $k\geq 0$, we have $f(w_{k+1}) \leq f(w_k) - \frac{1}{2}\omega(\ltwos{\utilde_k})$.
    \item[(b)] If $\ltwos{\utilde_k} \leq 1/8$, then we have 
$\omega(\ltwos{\utilde_{k+1}})\leq\sqrt{6}\,\omega^{3/2}(\ltwos{\utilde_{k}})$.
\end{enumerate}
\end{theorem}

Part~(b) suggests superlinear convergence when $\ltwos{\utilde_k}$ is small.
This comes at the cost of a smaller approximation tolerance~$\epsilon_k$ 
given in~\eqref{eqn:epsilon-k-superlinear}, 
compared with~\eqref{eqn:choose-epsilon-k}.
Roughly speaking, when $\ltwos{f'(w_k)}$ is relative large,
the tolerance $\epsilon_k$ in~\eqref{eqn:epsilon-k-superlinear} 
needs to be proportional to $\ltwos{f'(w_k)}^{3/2}$ since $\omega(t)=\order(t)$.
When $\ltwos{f'(w_k)}$ is very small, 
the tolerance $\epsilon_k$ in~\eqref{eqn:epsilon-k-superlinear} 
needs to be proportional to $\ltwos{f'(w_k)}^3$ because $\omega(t)\sim t^2$ as
$t\to 0$.
In contrast, for linear convergence, the tolerance 
in~\eqref{eqn:choose-epsilon-k} is proportional to $\ltwos{f'(w_k)}$.

\begin{proof}
We start with the inequality~\eqref{eqn:copy-nesterov-4-1-12},
and upper bound the last two terms on its right-hand side.
Since $\omega'(t)=\frac{t}{1+t} < 1$, we have 
\[
    \omega(\ltwos{\utilde_k})-\omega(\ltwos{\vtilde_k})
\leq \bigl| \ltwos{\utilde_k}-\ltwos{\vtilde_k}\bigr|
\leq \ltwos{\utilde_k-\vtilde_k} .
\]
In addition, we have
\[
    \frac{\langle \vtilde_k-\utilde_k, \vtilde_k\rangle}{1+\ltwos{\vtilde_k}}
\leq \frac{\ltwos{\vtilde_k}}{1+\ltwos{\vtilde_k}} \ltwos{\utilde_k-\vtilde_k}
\leq \ltwos{\utilde_k-\vtilde_k} .
\]
Applying these two bounds to~\eqref{eqn:copy-nesterov-4-1-12}, we obtain
\begin{equation}\label{eqn:bound-by-uv-norm}
    f(w_{k+1}) \leq f(w_k) - \omega(\ltwos{\utilde_k}) + 2\ltwos{\utilde_k-\vtilde_k} .
\end{equation}
Next we bound $\ltwos{\utilde_k-\vtilde_k}$ using the approximation 
tolerance~$\epsilon_k$ specified in~\eqref{eqn:epsilon-k-superlinear},
\begin{align*}
    \ltwos{\utilde_k-\vtilde_k} 
& = \left\| [f''(w_k)]^{-1/2}f'(w_k) - [f''(w_k)]^{1/2} v_k \right\|_2 \\
& = \left\| [f''(w_k)]^{-1/2} \bigl( f''(w_k) v_k - f'(w_k) \bigr)\right\|_2\\
&\leq \lambda^{-1/2} \left\| f''(w_k) v_k - f'(w_k) \right\|_2\\
&\leq \lambda^{-1/2} \epsilon_k \\
&= \frac{1}{2} \min\left\{ \frac{\omega(r_k)}{2}, \frac{\omega^{3/2}(r_k)}{10} \right\}.
\end{align*}
Combining the above inequality with~\eqref{eqn:bound-by-uv-norm},
and using $r_k = L^{-1/2}\ltwos{f'(w_k)} \leq \ltwos{\utilde_k}$
with the monotonicity of $\omega(\cdot)$,
we arrive at
\begin{equation}\label{eqn:ready-for-superlinear}
    f(w_{k+1}) \leq f(w_k) - \omega(\ltwos{\utilde_k}) 
    + \min\left\{ \frac{\omega(\utilde_k)}{2}, \frac{\omega^{3/2}(\utilde_k)}{10} \right\}.
\end{equation}

Part~(a) of the theorem follows immediately from
inequality~\eqref{eqn:ready-for-superlinear}. 

For part~(b), we assume that $\ltwos{\utilde_k}\leq 1/8$.
Combining~\eqref{eqn:bound-regret-using-remainder} 
with~\eqref{eqn:ready-for-superlinear}, we have
\begin{align} 
    \omega(\ltwos{\utilde_{k+1}}) & \leq f(w_{k+1}) - f(\wstar)
    \leq f(w_k)-f(\wstar) - \omega(\ltwos{\utilde_k}) + \frac{\omega^{3/2}(\ltwos{\utilde_k})}{10} \nonumber \\
    & \leq \omega_*(\ltwos{\utilde_k}) - \omega(\ltwos{\utilde_k}) + \frac{\omega^{3/2}(\ltwos{\utilde_k})}{10}.
    \label{eqn:omega-k-bound}
\end{align}
Let $h(t)\defeq\omega_*(t)-\omega(t)$ and consider only $t\geq 0$.
Notice that $h(0)=0$ and 
$h'(t)=\frac{2t^2}{1-t^2}<\frac{128}{63}t^2$ for $t\leq 1/8$. 
Thus, we conclude that $h(t)\leq \frac{128}{189}t^3$ for $t\leq 1/8$.
We also notice that $\omega(0)=0$ and $\omega'(t)=\frac{t}{1+t}\geq\frac{8}{9}t$
for $t\leq 1/8$.
Thus, we have $\omega(t)\geq\frac{4}{9}t^2$ for $t\leq 1/8$.
Combining these results, we obtain
\[
    \omega_*(t)-\omega(t) \leq \frac{128}{189}t^3 = \frac{128}{189}(t^2)^{3/2}
    \leq \frac{128}{189}\left(\frac{9}{4}\omega(t)\right)^{3/2}
    \leq \left(\sqrt{6}-\frac{1}{10}\right) \omega^{3/2}(t) .
\]
Applying this inequality to the right-hand side of~\eqref{eqn:omega-k-bound}
completes the proof.
\end{proof}

In classical analysis of inexact Newton methods
\cite{DennisMore74superlinear,Dembo82InexactNewton}, 
asymptotic superlinear convergence occurs with 
$\epsilon_k\sim\ltwos{f'(w_k)}^{3/2}$ 
(in fact with $\epsilon\sim\ltwos{f'(w_k)}^s$ for any $s>1$).
This agrees with our analysis since $\omega(t)=\order(t)$ when~$t$ is
not too small.
Our result can be very conservative asymptotically
because $\omega(t)\sim t^2$ as $t\to 0$.
However, using $\omega(t)$ and the associated self-concordance analysis,
we are able to derive a much better global complexity result.

\begin{corollary}\label{coro:superlinear-complexity}
    Suppose $f:\R^d\to\R$ is a standard self-concordant function 
    and Assumption~\ref{asmp:Hessian-bounds} holds. 
    If we choose the sequence $\{\epsilon_k\}$ 
    in Algorithm~\ref{alg:damped-newton-method}
    as in~\eqref{eqn:epsilon-k-superlinear},
    then for any $\epsilon \leq 1/(3e)$, 
    we have $f(w_k)-f(w_\star)\leq\epsilon$ whenever
\begin{align}\label{eqn:superlinear-complexity}
k \geq \left\lceil\frac{f(w_0) - f(\wstar)}{\frac{1}{2}\omega(1/8)}\right\rceil
+ \left\lceil \frac{\log\log(1/(3\epsilon))}{\log(3/2)} \right\rceil .
\end{align}
where $\lceil t\rceil$ denotes the smallest \emph{nonnegative} integer that 
is larger than or equal to~$t$.
\end{corollary}

\begin{proof}
    By part~(a) of Theorem~\ref{thm:superlinear-convergence}, if $\omega(\ltwos{\utilde_k})\geq 1/8$, then each iteration of Algorithm~\ref{alg:damped-newton-method} decreases the function value at least by the constant $\frac{1}{2}\omega(1/8)$.
So within at most 
$K_1\defeq\left\lceil\frac{f(w_0)-f(\wstar)}{\frac{1}{2}\omega(1/8)}\right\rceil$
iterations, we are guaranteed to have $\ltwos{\utilde_k}\leq 1/8$.

Part~(b) of Theorem~\ref{thm:superlinear-convergence} implies
$6\,\omega(\ltwos{\utilde_{k+1}}) \leq \left( 6\,\omega(\ltwos{\utilde_k})\right)^{3/2}$ when $\ltwos{\utilde_k}\leq 1/8$, and hence
\[
    \log\bigl(6\,\omega(\ltwos{\utilde_k})\bigr)
    \leq\left(\frac{3}{2}\right)^{k-K_1} 
    \log\left(6\,\omega(1/8)\right), 
    \qquad k \geq K_1.
\]
Note that both sides of the above inequality is negative.
Therefore, after $k\geq K_1 + \frac{\log\log(1/(3\epsilon))}{\log(3/2)}$ 
iterations (assuming $\epsilon\leq 1/(3e)$), we have
\[
    \log\bigl(6\,\omega(\ltwos{\utilde_k})\bigr) 
    \leq  \log(1/(3\epsilon)) \log(6\,\omega(1/8))
    \leq -\log(1/(3\epsilon)),
\]
which implies $\omega(\ltwos{\utilde_k}) \leq \epsilon/2$. 
Finally using~\eqref{eqn:bound-regret-using-remainder} 
and the fact that $\omega_*(t)\leq 2\,\omega(t)$ for $t\leq 1/8$, 
we obtain
\[
    f(w_k)-f(\wstar) \leq \omega_*(\ltwos{\utilde_k})
    \leq 2\,\omega(\ltwos{\utilde_k}) \leq \epsilon.
\]
This completes the proof.
\end{proof}

\section{Proof of Lemma~\ref{lemma:pcg-complexity}}
\label{sec:proof-lemma-pcg}

It suffices to show that the algorithm terminates at iteration $ t\leq T_{\mu}-1$, because
when the algorithm terminates, it outputs a vector $v_k$ which satisfies
$\ltwos{H v_k - \fp(w_k)} = \ltwos{r^{(t+1)}} \leq \epsilon_k$.
Denote by $v^* = H^{-1} \fp(w_k)$ the solution of the linear system
$H v_k = f'(w_k)$. 
By the classical analysis on the preconditioned conjugate gradient 
method (e.g., \cite{Luenberger73,Avriel76}),
Algorithm~\ref{alg:distributed-pcg} has the convergence rate
\begin{align}\label{eqn:nestrove-orignal-rate}
    (v^{(t)} - v^*)^T H (v^{(t)} - v^*) \leq 4 \left( \frac{\sqrt{\kappa} -1 }{\sqrt{\kappa}+1}\right)^{2t} (v^*)^T H v^*,
\end{align}
where $\kappa = 1 + 2\mu / \reg$ is the condition number  of $\pcd^{-1}H$ given
in~\eqref{eqn:preconditioned-kappa}.
For the left-hand side of inequality~\eqref{eqn:nestrove-orignal-rate}, we have
\begin{align*}
	(v^{(t)} - v^*)^T H (v^{(t)} - v^*) = (r^{(t)})^T H^{-1} r^{(t)} \geq \frac{ \ltwos{r^{(t)}}^2}{\hessianbound}.
\end{align*}
For the right-hand side of inequality~\eqref{eqn:nestrove-orignal-rate},
we have
\begin{align*}
	 (v^*)^T H v^* & = (f'(w_k))^T H^{-1} f'(w_k) \leq  \frac{\ltwos{f'(w_k)}^2}{\reg} .
\end{align*}
Combining the above two inequalities with inequality~\eqref{eqn:nestrove-orignal-rate}, we obtain
\begin{align*}
    \ltwos{r^{(t)}} \leq 2 \sqrt{\frac{L}{\lambda}} 
    \left( \frac{\sqrt{\kappa} -1 }{\sqrt{\kappa}+1}\right)^{t} \ltwos{f'(w_k)}
    \leq 2 \sqrt{\frac{L}{\lambda}} 
    \left(1 - \sqrt{\frac{\reg}{\reg + 2\mu}}\right)^{t}\ltwos{f'(w_k)} .
\end{align*}
To guarantee that $\ltwos{r^{(t)}} \leq \epsilon_k$, it suffices to have
\begin{align*}
    t ~\geq~ \frac{\log\Bigl(\frac{2 \sqrt{L/\lambda} \ltwos{f'(w_k)}}{\epsilon_k}\Bigr)}{- \log\left(1 - \sqrt{\frac{\reg}{\reg + 2\mu}}\right)} 
    ~\geq~ \sqrt{1+\frac{2\mu}{\lambda}}\, \log\biggr(\frac{2 \sqrt{L/\lambda} \ltwos{f'(w_k)}}{\epsilon_k}\biggr),
\end{align*}
where in the last inequality we used $-\log(1-x) \geq x$ for $0<x<1$.
Comparing with the definition of $T_\mu$, this is the desired result.


\section{Proof of Lemma~\ref{lemma:initialization-accuracy}}
\label{sec:initialization-accuracy-proof}

First, we prove inequality~\eqref{eqn:w0-norm-bound}.
Recall that $\wstar$ and  $\what_i$ minimizes $f(w)$ and $f_i(w) + \frac{\rho}{2}\ltwos{w}^2$.
Since both function are $\reg$-strongly convex, we have
\begin{align*}
	\frac{\reg}{2}\ltwos{\wstar}^2 &\leq f(\wstar)  \leq f(0) \leq V_0,\\
	\frac{\reg}{2}\ltwos{\what_i}^2 &\leq f_i(\what_i) + \frac{\rho}{2}\ltwos{\what_i}^2 \leq f_i(0) \leq V_0,
\end{align*}
which implies $\ltwos{\wstar} \leq \sqrt{\frac{2V_0}{\reg}}$ and $\ltwos{\what_i} \leq \sqrt{\frac{2V_0}{\reg}}$.
Then inequality~\eqref{eqn:w0-norm-bound} follows 
since $w_0$ is the average over $\{\what_i\}_{i=1}^m$.

In the rest of Appendix~\ref{sec:initialization-accuracy-proof},
we prove inequality~\eqref{eqn:initial-obj-bound}.
Let $z$ be a random variable in $\mathcal{Z}\subset\R^p$ with an unknown 
probability distribution.
We define a regularized population risk:
\[
    R(w) = \E_z[\phi(w, z)] + \frac{\lambda+\rho}{2}\ltwos{w}^2. 
\]
Let~$S$ be a set of~$n$ i.i.d.\ samples in~$\mathcal{Z}$
from the same distribution. 
We define a regularized empirical risk
\[
r_S(w) = \frac{1}{n}\sum_{z\in S}\phi(w,z) + \frac{\lambda+\rho}{2}\ltwos{w}^2,
\]
and its minimizer
\[
    \what_S = \arg\min_w ~r_S(w).
\]
The following lemma states that the population risk of $\what_S$ is
very close to its empirical risk.
The proof is based on the notion of \emph{stability}
of regularized empirical risk minimization \cite{BousquetElisseeff02}.

\begin{lemma}\label{lemma:stability}
    Suppose Assumption~\ref{asmp:smoothness} holds 
    and~$S$ is a set of~$n$ i.i.d.\ samples in~$\mathcal{Z}$.
    Then we have
   \[
       \E_S\bigl[R(\what_S) - r_S(\what_S)\bigr] \leq \frac{2G^2}{\rho n}.
   \]
\end{lemma}

\begin{proof}
Let $S=\{z_1,\ldots,z_n\}$.
For any $k\in\{1,\ldots,n\}$, we define a modified training set $S^{(k)}$
by replacing~$z_k$ with another sample~$\ztilde_k$, which is
drawn from the same distribution and is independent of~$S$.
The empirical risk on $S^{(k)}$ is defined as
\[
    r_S^{(k)}(w) = \frac{1}{n}\sum_{z\in S^{(k)}} \phi(w,z) + \frac{\lambda+\rho}{2}\ltwos{w}^2.
\]
and let $\what_S^{(k)}= \arg\min_w r_S^{(k)}(w)$.
Since both $r_S$ and $r_S^{(k)}$ are $\rho$-strongly convex, we have
\begin{align*}
    r_S(\what_S^{(k)}) - r_S(\what_S) &\geq \frac{\rho}{2} \ltwos{\what_S^{(k)} - \what_S}^2\\
    r_S^{(k)}(\what_S) - r_S^{(k)}(\what_S^{(k)}) &\geq \frac{\rho}{2} \ltwos{\what_S^{(k)} - \what_S}^2.
\end{align*}
Summing the above two inequalities, and noticing that 
\[
    r_S(w) - r_S^{(k)}(w) = \frac{1}{n}(\phi(w,z_k) - \phi(w,\ztilde_k)),
\]
we have
\begin{equation}\label{eqn:what-stability}
    \ltwos{\what_S^{(k)}- \what_S}^2 
    \leq \frac{1}{\rho n}\left( \phi(\what_S^{(k)},z_k) - \phi(\what_S^{(k)},\ztilde_k)	- \phi(\what_S,z_k) + \phi(\what_S, \ztilde_k)\right) .
\end{equation}
By Assumption~\ref{asmp:smoothness} (ii)
and the facts $\ltwos{\what_S}\leq\sqrt{2V_0/\lambda}$ and
$\ltwos{\what_S^{(k)}}\leq\sqrt{2V_0/\lambda}$,  
we have
\[
    \bigl|\phi(\what_S^{(k)},z) - \phi(\what_S,z)\bigr| 
    \leq \gradbound \ltwos{\what_S^{(k)} - \what_S}, \qquad
    \forall\, z\in \mathcal{Z}.
\]
Combining the above Lipschitz condition with~\eqref{eqn:what-stability},
we obtain
\[
    \ltwos{\what_S^{(k)}- \what_S}^2 
    \leq \frac{2 \gradbound}{\rho n} \ltwos{\what_S^{(k)} - \what_S}.
\]
As a consequence, we have
$\ltwos{\what_S^{(k)} - \what_S} \leq \frac{2 \gradbound}{\rho n}$,
and therefore
\begin{equation}\label{eqn:uniform-stability}
    \bigl|\phi(\what_S^{(k)},z) - \phi(\what_S,z)\bigr| 
    \leq \frac{2 \gradbound^2}{\rho n}, 
    \qquad \forall\, z\in \mathcal{Z}.
\end{equation}
In the terminology of learning theory,
this means that empirical minimization over the regularized loss $r_S(w)$
has \emph{uniform stability} $2G^2/(\rho n)$ with respect to the 
loss function~$\phi$; see \cite{BousquetElisseeff02}.

For any fixed $k\in\{1,\ldots,n\}$, since $\ztilde_k$ is independent of~$S$, we have
\begin{align*}
    \E_S\bigl[R(\what_S) - r_S(\what_S)\bigr] 
    &= \E_S \biggl[ \E_{\ztilde_k}[\phi(\what_S,\ztilde_k)]
       - \frac{1}{n}\sum_{j=1}^n \phi(\what_S, z_j) \biggr] \\
    &= \E_{S,\ztilde_k}\bigl[\phi(\what_S,\ztilde_k)-\phi(\what_S,z_k)\bigr] \\
    &= \E_{S,\ztilde_k}\bigl[\phi(\what_S,\ztilde_k)-\phi(\what_S^{(k)},\ztilde_k)\bigr],
\end{align*}
where the second equality used the fact that 
$\E_S[\phi(\what_S,z_j)$ has the same value for all $j=1,\ldots,n$,
and the third equality used the symmetry between the pairs
$(S, z_k)$ and $(S^{(k)}, \ztilde_k)$
(also known as the \emph{renaming} trick; see \cite[Lemma~7]{BousquetElisseeff02}).
Combining the above equality with~\eqref{eqn:uniform-stability} yields
the desired result.
\end{proof}

Next, we consider a distributed system with~$m$ machines, where each machine
has a local dataset $S_i$ of size~$n$, for $i=1,\ldots,m$. 
To simplify notation, we denote the local regularized empirical loss function
and its minimizer by $r_i(w)$ and $\what_i$, respectively.
We would like to bound the excessive error when applying $\what_i$ 
to a different dataset $S_j$. Notice that
\begin{align}\label{eqn:v1-v2-v3}
    \E_{S_i,S_j}\bigl[r_j(\what_i) - r_j(\what_j)\bigr] 
    = \underbrace{\E_{S_i,S_j}\bigl[r_j(\what_i) - r_i(\what_i)\bigr]}_{v_1} 
    + \underbrace{\E_{S_i,S_j}\bigl[r_i(\what_i) - r_j(\what_R)\bigr]}_{v_2} 
    + \underbrace{\E_{S_j}\bigl[r_j(\what_R) - r_j(\what_j)\bigr]}_{v_3}
\end{align}
where $\what_R$ is the constant vector minimizing $R(w)$. 
Since $S_i$ and $S_j$ are independent, we have
\begin{align*}
    v_1 = \E_{S_i}\bigl[\E_{S_j}[r_j(\what_i)] - r_i(\what_i)\bigr]
= \E_{S_i}\bigl[R(\what_i) - r_i(\what_i)]  
\leq \frac{2 \gradbound^2}{\rho n} ,
\end{align*}
where the inequality is due to Lemma~\ref{lemma:stability}.
For the second term, we have    
\begin{align*}
    v_2 = \E_{S_i}\bigl[r_i(\what_i) - \E_{S_j}[r_j(\what_R)]\bigr]
    = \E_{S_i}\bigl[r_i(\what_i) - r_i(\what_R)\bigr] \leq 0.
\end{align*}
It remains to bound the third term~$v_3$.
We first use the strong convexity of $r_j$ to obtain
(e.g., \cite[Theorem~2.1.10]{Nesterov04book})
\begin{align}\label{eqn:r-j-bound}
	r_j(\what_R) - r_j(\what_j) \leq \frac{\ltwos{r_j'(\what_R)}^2}{2\rho},
\end{align}
where $r'_j(\what_R)$ denotes the gradient of $r_j$ at $\what_R$.
If we index the elements of $S_j$ by $z_1,\ldots,z_n$, then 
\begin{equation}\label{eqn:r-j-gradient}
    r_j'(\what_R) = \frac{1}{n} \sum_{k=1}^n \left(\phi'(\what_R, z_k) + (\reg+\rho) \what_R \right).
\end{equation}
By the optimality condition of $\what_R=\arg\min_w R(w)$, we have
for any $k\in\{1,\ldots,n\}$,
\[
    \E_{z_k}\bigl[ \phi'(\what_R,z_k) + (\lambda+\rho)\what_R\bigr] = 0.
\]
Therefore, according to~\eqref{eqn:r-j-gradient}, the gradient $r_j(\what_R)$
is the average of~$n$ independent and zero-mean random vectors.
Combining~\eqref{eqn:r-j-bound} and~\eqref{eqn:r-j-gradient} with the definition
of~$v_3$ in~\eqref{eqn:v1-v2-v3}, we have
\begin{align*}
	v_3 
    &\leq \frac{\E_{S_j}\!\left[\sum_{k=1}^n \ltwos{\phi'(\what_R,z_k) + (\reg+\rho) \what_R}^2\right]}{2 \rho n^2} \\
    &= \frac{\sum_{k=1}^n \E_{S_j}\!\left[\ltwos{\phi'(\what_R,z_k) + (\reg+\rho) \what_R}^2\right]}{2 \rho n^2} \\
    &\leq \frac{\sum_{k=1}^n \E[\ltwos{\phi'(\what_R,z_k)}^2]}{2 \rho n^2}\\
	&\leq \frac{\gradbound^2}{2 \rho n}.
\end{align*}
In the equality above, we used the fact that 
$\phi'(\what_R,z_k)+(\lambda+\rho)\what_R$ are i.i.d.\ zero-mean random variables; so the variance of their sum equals the sum of their variances.
The last inequality above is due to Assumption~\ref{asmp:smoothness}~(ii)
and the fact that
$\ltwos{\what_R}\leq \sqrt{2V_0/(\lambda+\rho)}\leq\sqrt{2V_0/\lambda}$.
Combining the upper bounds for $v_1$, $v_2$ and $v_3$, we have
\begin{align}\label{eqn:expect-Si-Sj}
    \E_{S_i,S_j} \left[r_j(\what_i) - r_j(\what_j)\right] \leq \frac{3 \gradbound^2}{\rho n}.
\end{align}

Recall the definition of $f(w)$ as
\[
    f(w) = \frac{1}{mn}\sum_{i=1}^m\sum_{k=1}^n \phi(w, z_{i,k})+\frac{\lambda}{2}\ltwos{w}^2,
\]
where $z_{i,k}$ denotes the $k$th sample at machine~$i$.
Let $r(w) = \frac{1}{m}\sum_{j=1}^m r_j(w)$; then we have
\begin{equation}\label{eqn:rw-fw}
    r(w) = f(w) + \frac{\rho}{2} \ltwos{w}^2 .
\end{equation}
We compare the value $r(\what_i)$, for any $i\in\{1,\ldots,m\}$,
with the minimum of $r(w)$:
\begin{align*}
   r(\what_i) -\min_w r(w) 
   & = \frac{1}{m}\sum_{j=1}^m r_j(\what_i) - \min_w\frac{1}{m}\sum_{j=1}^m r_j(w) \\
   & \leq \frac{1}{m}\sum_{j=1}^m r_j(\what_i) - \frac{1}{m}\sum_{j=1}^m \min_w r_j(w) \\
   & = \frac{1}{m}\sum_{j=1}^m \left( r_j(\what_i) - r_j(\what_j)\right) .
\end{align*}
Taking expectation with respect to all the random data sets $S_1,\ldots, S_m$
and using~\eqref{eqn:expect-Si-Sj}, we obtain
\begin{align}\label{eqn:ri-rmin}
	\E[ r(\what_i) - \min_w r(w)] \leq \frac{1}{m} \sum_{j=1}^n \E[ r_j(\what_i)-  r_j(\what_j)] \leq \frac{3 \gradbound^2}{\rho n}.
\end{align}
Finally, we bound the expected value of $f(\what_i)$:
\begin{align*}
    \E[f(\what_i)] 
    &\leq \E[r(\what_i)] 
    \leq \E\left[\min_w r(w)\right] + \frac{3G^2}{\rho n} \\
    &\leq \E\left[f(\wstar)+\frac{\rho}{2}\ltwos{\wstar}^2\right]+ \frac{3G^2}{\rho n} \\
    & \leq \E\left[f(\wstar)\right]+\frac{\rho D^2}{2}+ \frac{3G^2}{\rho n},
\end{align*}
where the first inequality holds because of~\eqref{eqn:rw-fw},
the second inequality is due to~\eqref{eqn:ri-rmin},
and the last inequality follows from the assumption that 
$\E[\ltwos{\wstar}]\leq D^2$.
Choosing $\rho = \sqrt{\frac{6 \gradbound^2}{n D^2 }}$ results in
$\E[f( \what_i) - f(\wstar)] \leq \frac{\sqrt{6}\gradbound D}{\sqrt{n}}$
for every $i\in \{1,\ldots,m\}$. 
Since $w_0=\frac{1}{m}\sum_{i=1}^m \what_i$,
using the convexity of function~$f$ yields
$\E[f(w_0) - f(\wstar)] \leq \frac{\sqrt{6}\gradbound D}{\sqrt{n}}$,
which is the desired result.

\section{Proof of Lemma~\ref{lemma:uniform-matrix-concentration}}
\label{sec:proof-uniform-matrix-concentration}

We consider the regularized empirical loss functions $f_i(w)$ defined 
in~\eqref{eqn:regularized-loss}.
For any two vectors $u,w\in \R^d$ satisfying $\ltwos{u-w}\leq \varepsilon$, 
Assumption~\ref{asmp:smoothness}~(iv) implies
\begin{align*}
	\ltwos{\fpp_i(u) - \fpp_i(w)} \leq \tensorbound \varepsilon.
\end{align*}
Let $B(0,r)$ be the ball in $\R^d$ with radius~$r$, centered at the origin.
Let $N_\varepsilon^{\mathrm{cov}}(B(0,r))$ be the \emph{covering number} 
of $B(0,r)$ by balls of radius~$\varepsilon$, 
\ie, the minimum number of balls of radiusr~$\varepsilon$ required to 
cover $B(0,r)$. 
We also define $N_\varepsilon^{\mathrm{pac}}(B(0,r))$ as the 
\emph{packing number} of $B(0,r)$, \ie, 
the maximum number of disjoint balls whose centers belong to $B(0,r)$. 
It is easy to verify that
\begin{align*}
    N_{\varepsilon}^{\mathrm{cov}}(B(0,r)) \leq 
    N_{\varepsilon/2}^{\mathrm{pac}}(B(0,r)) \leq 
    \left( 1 + {2r}/{\varepsilon}\right)^d.
\end{align*}
Therefore, there exist a set of points $U\subseteq \R^d$ with cardinality at most $( 1 + {2r}/{\varepsilon})^d$, such that for any vector $w\in B(0,r)$, we have
\begin{align}\label{eqn:difference-by-convering}
	\min_{u\in U} \ltwos{\fpp_i(w) - \fpp_i(u)} \leq \tensorbound\varepsilon.
\end{align}

We consider an arbitrary point $u\in U$ and the associated Hessian matrices 
for the functions $f_i(w)$ defined in~\eqref{eqn:regularized-loss}.
We have
\[
\fpp_i(u) = \frac{1}{n} \sum_{j=1}^n \left(\phi''(u, z_{i,j}) + \reg I\right),
\qquad i=1,\ldots,m.
\]
The components of the above sum are i.i.d.~matrices which are upper bounded by 
$\hessianbound I$.
By the matrix Hoeffding's inequality~\cite[Corollary~4.2]{mackey2014matrix}, 
we have
\begin{align*}
	\Prob\left[\ltwos{\fpp_i(u) - \E[\fpp_i(u)]} > t\right] \leq d \cdot e^{- \frac{n t^2}{2\hessianbound^2}} .
\end{align*}
Note that $\E[\fpp_1(w)] = \E[\fpp(w)]$ for any $w\in B(0,r)$. 
Using the triangular inequality and 
inequality~\eqref{eqn:difference-by-convering}, we obtain
\begin{align}
\ltwos{\fpp_1(w) - \fpp(w)]} 
&\leq \ltwos{\fpp_1(w) - \E[\fpp_1(w)]} + \ltwos{\fpp(w) - \E[\fpp(w)]} \nonumber \\[1ex]
& \leq 2\max_{i\in\{1,\ldots,m\}} \ltwos{\fpp_i(w) - \E[\fpp_i(w)]} \nonumber\\
&\leq 2\max_{i\in\{1,\ldots,m\}} \Big(\max_{u\in U}\ltwos{\fpp_i(u) - \E[\fpp_i(u)]} + \tensorbound \varepsilon \Big).\label{eqn:union-bound-derivation-detail}
\end{align}
Applying the union bound, we have with probability at least 
\[
 1 - m d ( 1 + {2r}/{\varepsilon})^d\cdot e^{- \frac{n t^2}{2\hessianbound^2}},
\]
the inequality $\ltwos{\fpp_i(u) - \E[\fpp_i(u)]} \leq t$ holds for every 
$i\in\{1,\ldots,m\}$ and every $u\in U$. 
Combining this probability bound with 
inequality~\eqref{eqn:union-bound-derivation-detail}, we have
\begin{align}\label{eqn:union-bound-outcome}
	\Prob\Big[\sup_{w\in B(0,r)}\ltwos{\fpp_1(w) - \fpp(w)} > 2 t + 2 \tensorbound\varepsilon \Big] \leq 
	m d \left( 1 + {2r}/{\varepsilon}\right)^d \cdot e^{- \frac{n t^2}{2\hessianbound^2}}.
\end{align}
As the final step, we choose 
$\varepsilon = \frac{\sqrt{2}\hessianbound}{\sqrt{n}\tensorbound}$
and then choose~$t$ to make the right-hand side of 
inequality~\eqref{eqn:union-bound-outcome} equal to $\delta$.
This yields the desired result.

\section{More analysis on the number of PCG iterations}
\label{sec:analysis-T-L}

Here we analyze the number of iterations of the distributed PCG method
(Algorithm~\ref{alg:distributed-pcg}) when $\mu$ is misspecified, \ie,
when~$\mu$ used in $\pcd=H_1+\mu I$ is not an upper bound on $\|H_1-H\|_2$.
For simplicity of discussion, we assume that 
Assumption~\ref{asmp:Hessian-bounds} holds,
$\|H_1-H\|_2\leq L$ and $\mu\leq L$.
In this case, we can show (using similar arguments for proving 
Lemma~\ref{lemma:reduce-condition-number}):
\begin{align*}
  \sigma_\mathrm{max}( (H_1+\mu I)^{-1} H ) &\leq \frac{2L}{L+\mu}, \\
  \sigma_\mathrm{min}( (H_1+\mu I)^{-1} H ) &\geq \frac{\lambda}{L+\mu+\lambda}.
\end{align*}
Hence the condition number of the preconditioned linear system is
\[
    \kappa_{\mu,L} = \frac{2L}{\lambda}\left(1+\frac{\lambda}{L+\mu}\right)
    \leq 2 + \frac{2L}{\lambda}, 
\]
and the number of PCG iterations is bounded by 
(\cf~Appendix~\ref{sec:proof-lemma-pcg})
\[
    \left\lceil\sqrt{\kappa_{\mu,L}}\log\left(\frac{2L}{\beta\lambda}\right)\right\rceil
    \leq \sqrt{2+\frac{2L}{\lambda}} \log\left(\frac{2L}{\beta\lambda}\right). 
\]
This gives the bound on number of PCG iterations in~\eqref{eqn:T-L}.

\newpage
\bibliographystyle{abbrv}
\bibliography{DiSCO_paper}

\begin{thebibliography}{10}

\bibitem{agarwal2011distributed}
A.~Agarwal and J.~C. Duchi.
\newblock Distributed delayed stochastic optimization.
\newblock In {\em Advances in Neural Information Processing Systems}, pages
  873--881, 2011.

\bibitem{Avriel76}
M.~Avriel.
\newblock {\em Nonlinear Programming: Analysis and Methods}.
\newblock Prentice-Hall, 1976.

\bibitem{Bach10selfconcordance}
F.~Bach.
\newblock Self-concordant analysis for logistic regression.
\newblock {\em Electronic Journal of Statistics}, 4:384--414, 2010.

\bibitem{Bekkerman2011scaling}
R.~Bekkerman, M.~Bilenko, and J.~Langford.
\newblock {\em Scaling up Machine Learning: Parallel and Distributed
  Approaches}.
\newblock Cambridge University Press, 2011.

\bibitem{BertsekasTsitsiklis89book}
D.~P. Bertsekas and J.~N. Tsitsiklis.
\newblock {\em Parallel and Distributed Computation: Numerical Methods}.
\newblock Prentice-Hall, 1989.

\bibitem{covertype}
J.~A. Blackard, D.~J. Dean, and C.~W. Anderson.
\newblock Covertype data set.
\newblock In K.~Bache and M.~Lichman, editors, {\em {UCI} Machine Learning
  Repository}, URL: http://archive.ics.uci.edu/ml, 2013. University of
  California, Irvine, School of Information and Computer Sciences.

\bibitem{BousquetElisseeff02}
O.~Bousquet and A.~Elisseeff.
\newblock Stability and generalization.
\newblock {\em Journal of Machine Learning Research}, 2:499--526, 2002.

\bibitem{Boyd10ADMM}
S.~Boyd, N.~Parikh, E.~Chu, B.~Peleato, and J.~Eckstein.
\newblock Distributed optimization and statistical learning via the alternating
  direction method of multipliers.
\newblock {\em Foundations and Trends in Machine Learning}, 3(1):1--122, 2010.

\bibitem{BoydVandenberghe04book}
S.~Boyd and L.~Vandenberghe.
\newblock {\em Convex optimization}.
\newblock Cambridge university press, 2004.

\bibitem{ChambollePock11}
A.~Chambolle and T.~Pock.
\newblock A first-order primal-dual algorithm for convex problems with
  applications to imaging.
\newblock {\em Journal of Mathematical Imaging and Vision}, 40(1):120--145,
  2011.

\bibitem{chang2011libsvm}
C.-C. Chang and C.-J. Lin.
\newblock Libsvm: a library for support vector machines.
\newblock {\em ACM Transactions on Intelligent Systems and Technology (TIST)},
  2(3):27, 2011.

\bibitem{CortesVapnik95svm}
C.~Cortes and V.~Vapnik.
\newblock Support-vector networks.
\newblock {\em Machine Learning}, 20(3):273--297, 1995.

\bibitem{Dean08MapReduce}
J.~Dean and S.~Ghemawat.
\newblock {MapReduce}: Simplfied data processing on large clusters.
\newblock {\em Communications of the ACM}, 51(1):107--113, 2008.

\bibitem{DefazioBach14SAGA}
A.~Defazio, F.~Bach, and S.~Lacoste-Julien.
\newblock {SAGA}: A fast incremental gradient method with support for
  non-strongly convex composite objectives.
\newblock In Z.~Ghahramani, M.~Welling, C.~Cortes, N.~Lawrence, and
  K.~Weinberger, editors, {\em Advances in Neural Information Processing
  Systems 27}, pages 1646--1654. Curran Associates, Inc., 2014.

\bibitem{dekel2012optimal}
O.~Dekel, R.~Gilad-Bachrach, O.~Shamir, and L.~Xiao.
\newblock Optimal distributed online prediction using mini-batches.
\newblock {\em The Journal of Machine Learning Research}, 13(1):165--202, 2012.

\bibitem{Dembo82InexactNewton}
R.~S. Dembo, S.~C. Eisenstat, and T.~Steihaug.
\newblock Inexact {Newton} methods.
\newblock {\em SIAM Journal on Numerical Analysis}, 19(2):400--408, April 1982.

\bibitem{DengYin12linearADMM}
W.~Deng and W.~Yin.
\newblock On the global and linear convergence of the generalized alternating
  direction method of multipliers.
\newblock CAAM Technical Report 12-14, Rice University, 2012.

\bibitem{DennisMore74superlinear}
J.~E. Dennis and J.~J. Mor\'e.
\newblock A characterization of superlinear convergence and its application to
  quasi-{Newton} methods.
\newblock {\em Mathematics of Computation}, 28(126):549--560, April 1974.

\bibitem{duchi2012dual}
J.~C. Duchi, A.~Agarwal, and M.~J. Wainwright.
\newblock Dual averaging for distributed optimization: convergence analysis and
  network scaling.
\newblock {\em IEEE Transactions on Automatic Control}, 57(3):592--606, 2012.

\bibitem{GolubVanLoan96book}
G.~H. Golub and C.~F. {Van Loan}.
\newblock {\em Matrix Computations}.
\newblock The John Hopkins University Press, Baltimore, MD, third edition,
  1996.

\bibitem{HornJohnson85}
R.~A. Horn and C.~R. Johnson.
\newblock {\em Matrix Analysis}.
\newblock Cambridge University Press, 1985.

\bibitem{News20binary}
S.~S. Keerthi and D.~DeCoste.
\newblock A modified finite {Newton} method for fast solution of large scale
  linear svms.
\newblock {\em Journal of Machine Learning Research}, 6:341--361, 2005.

\bibitem{Lang95News20}
K.~Lang.
\newblock Newsweeder: Learning to filter netnews.
\newblock In {\em Proceedings of the Twelfth International Conference on
  Machine Learning (ICML)}, pages 331--339, 1995.

\bibitem{LeeSunSaunders14proxNewton}
J.~D. Lee, Y.~Sun, and M.~Saunders.
\newblock Proximal {Newton}-type methods for minimizing composite functions.
\newblock {\em SIAM Journal on Optimization}, 24(3):1420--1443, 2014.

\bibitem{RCV1}
D.~D. Lewis, Y.~Yang, T.~Rose, and F.~Li.
\newblock {RCV1}: A new benchmark collection for text categorization research.
\newblock {\em Journal of Machine Learning Research}, 5:361--397, 2004.

\bibitem{LinTsaiLeeLin14}
C.-Y. Lin, C.-H. Tsai, C.-P. Lee, and C.-J. Lin.
\newblock Large-scale logistic regression and linear support vector machines
  using {Spark}.
\newblock In {\em Proceedings of the IEEE Conference on Big Data}, Washington
  DC, USA, 2014.

\bibitem{LinLuXiao14APCG}
Q.~Lin, Z.~Lu, and L.~Xiao.
\newblock An accelerated proximal coordinate gradient method and its
  application to regularized empirical risk minimization.
\newblock Technical Report MSR-TR-2014-94, Microsoft Research, 2014.
\newblock arXiv:1407.1296.

\bibitem{LinXiao14acclprox}
Q.~Lin and L.~Xiao.
\newblock An adaptive accelerated proximal gradient method and its homotopy
  contiuation for sparse optimization.
\newblock {\em Computational Optimization and Applications}, published online,
  September 2014.

\bibitem{Luenberger73}
D.~G. Luenberger.
\newblock {\em Introduction to Linear and Nonlinear Programming}.
\newblock Addison-Wesley, New York, 1973.

\bibitem{mackey2014matrix}
L.~Mackey, M.~I. Jordan, R.~Y. Chen, B.~Farrell, J.~A. Tropp, et~al.
\newblock Matrix concentration inequalities via the method of exchangeable
  pairs.
\newblock {\em The Annals of Probability}, 42(3):906--945, 2014.

\bibitem{MahajanKeerthi13}
D.~Mahajan, S.~S. Keerthi, S.~Sundararajan, and L.~Bottou.
\newblock A functional approximation based distributed learning algorithm.
\newblock arXiv:1310.8418, 2013.

\bibitem{MPIForum}
{MPI Forum}.
\newblock {MPI:} a message-passing interface standard, {Version} 3.0.
\newblock Document available at \texttt{http://www.mpi-forum.org}, 2012.

\bibitem{NemirovskiYudin83book}
A.~Nemirovsky and D.~Yudin.
\newblock {\em Problem Complexity and Method Efficiency in Optimization}.
\newblock J.\ Wiley \& Sons, New York, 1983.

\bibitem{Nesterov04book}
Y.~Nesterov.
\newblock {\em Introductory Lectures on Convex Optimization: A Basic Course}.
\newblock Kluwer, Boston, 2004.

\bibitem{Nesterov13composite}
Y.~Nesterov.
\newblock Gradient methods for minimizing composite functions.
\newblock {\em Mathematical Programming, Ser.\ B}, 140:125--161, 2013.

\bibitem{NesterovNemirovski94book}
Y.~Nesterov and A.~Nemirovski.
\newblock {\em Interior Point Polynomial Time Methods in Convex Programming}.
\newblock SIAM, Philadelphia, 1994.

\bibitem{NocedalWrightbook}
J.~Nocedal and S.~J. Wright.
\newblock {\em Numerical Optimization}.
\newblock Springer, New York, 2nd edition, 2006.

\bibitem{RamNedic2010distributed}
S.~S. Ram, A.~Nedi{\'c}, and V.~V. Veeravalli.
\newblock Distributed stochastic subgradient projection algorithms for convex
  optimization.
\newblock {\em Journal of optimization theory and applications},
  147(3):516--545, 2010.

\bibitem{recht2011hogwild}
B.~Recht, C.~Re, S.~Wright, and F.~Niu.
\newblock Hogwild: A lock-free approach to parallelizing stochastic gradient
  descent.
\newblock In {\em Advances in Neural Information Processing Systems}, pages
  693--701, 2011.

\bibitem{LeRouxSchmidtBach12}
N.~L. Roux, M.~Schmidt, and F.~Bach.
\newblock A stochastic gradient method with an exponential convergence rate for
  finite training sets.
\newblock In {\em Advances in Neural Information Processing Systems 25}, pages
  2672--2680. 2012.

\bibitem{SchmidtLeRouxBach13}
M.~Schmidt, N.~L. Roux, and F.~Bach.
\newblock Minimizing finite sums with the stochastic average gradient.
\newblock Technical Report HAL 00860051, INRIA, Paris, France, 2013.

\bibitem{SSSS09stochastic}
S.~Shalev-Shwartz, O.~Shamir, N.~Srebro, and K.~Sridharan.
\newblock Stochastic convex optimization.
\newblock In {\em Proceedings of the 22nd Annual Conference on Learning Theory
  (COLT)}, 2009.

\bibitem{SSZhang13acclSDCA}
S.~Shalev-Shwartz and T.~Zhang.
\newblock Accelerated proximal stochastic dual coordinate ascent for
  regularized loss minimization.
\newblock arXiv:1309.2375.

\bibitem{SSZhang13SDCA}
S.~Shalev-Shwartz and T.~Zhang.
\newblock Stochastic dual coordinate ascent methods for regularized loss
  minimization.
\newblock {\em Journal of Machine Learning Research}, 14:567--599, 2013.

\bibitem{Shalf10}
J.~Shalf, S.~Dosanjh, and J.~Morrison.
\newblock Exascale computing technology challenges.
\newblock In {\em Proceedings of the 9th International Conference on High
  Performance Computing for Computational Science}, VECPAR'10, pages 1--25.
  Springer-Verlag, 2011.

\bibitem{ShamirSrebro14}
O.~Shamir and N.~Srebro.
\newblock On distributed stochastic optimization and learning.
\newblock In {\em Proceedings of the 52nd Annual Allerton Conference on
  Communication, Control, and Computing}, 2014.

\bibitem{ShamirSrebroZhang14DANE}
O.~Shamir, N.~Srebro, and T.~Zhang.
\newblock Communication efficient distributed optimization using an approximate
  {Newton}-type method.
\newblock In {\em Proceedings of the 31st International Conference on Machine
  Learning (ICML)}. JMLR: W\&CP volume 32, 2014.

\bibitem{Shapiro09book}
A.~Shapiro, D.~Dentcheva, and A.~Ruszczy\'nski.
\newblock {\em Lectures on Stochastic Programming: Modeling and Theory}.
\newblock MPS-SIAM Series on Optimization. SIAM-MPS, Philadelphia, PA, 2009.

\bibitem{tibshirani1996regression}
R.~Tibshirani.
\newblock Regression shrinkage and selection via the lasso.
\newblock {\em Journal of the Royal Statistical Society. Series B
  (Methodological)}, pages 267--288, 1996.

\bibitem{tran2013composite}
Q.~Tran-Dinh, A.~Kyrillidis, and V.~Cevher.
\newblock Composite self-concordant minimization.
\newblock {\em arXiv preprint arXiv:1308.2867}, 2013.

\bibitem{XiaoZhang14ProxSVRG}
L.~Xiao and T.~Zhang.
\newblock A proximal stochastic gradient method with progressive variance
  reduction.
\newblock {\em SIAM Journal on Optimization}, 24(4):2057--2075, 2014.

\bibitem{zhang2012communication}
Y.~Zhang, M.~J. Wainwright, and J.~C. Duchi.
\newblock Communication-efficient algorithms for statistical optimization.
\newblock In {\em Advances in Neural Information Processing Systems}, pages
  1502--1510, 2012.

\bibitem{ZhangXiao14SPDC}
Y.~Zhang and L.~Xiao.
\newblock Stochastic primal-dual coordinate method for regularized empirical
  risk minimization.
\newblock Technical Report MSR-TR-2014-123, Microsoft Research, 2014.
\newblock arXiv:1409.3257.

\bibitem{ZhuangChinJuanLin14}
Y.~Zhuang, W.-S. Chin, Y.-C. Juan, and C.-J. Lin.
\newblock Distributed newton method for regularized logistic regression.
\newblock Technical report, Department of Computer Science, National Taiwan
  University, 2014.

\end{thebibliography}

\end{document}